\documentclass[9pt,a4paper]{article}
\usepackage[utf8]{inputenc}
\usepackage[T1]{fontenc}
\usepackage{lmodern} 
\usepackage[francais,english]{babel}
\usepackage{amsmath,amsthm}
\usepackage{mathrsfs}
\usepackage{bbm} 
\usepackage{amssymb}
\usepackage{amsfonts}
\usepackage{pifont}
\usepackage{stmaryrd}
\usepackage{dsfont}
\usepackage{graphicx,pstricks}
\usepackage{boites,boites_exemples}
\usepackage{enumerate}
\usepackage{lscape}
\usepackage[all]{xy}
\usepackage{hyperref}
\input xy
\xyoption{all}
\DeclareFontFamily{U}{wncy}{}
    \DeclareFontShape{U}{wncy}{m}{n}{<->wncyr10}{}
    \DeclareSymbolFont{mcy}{U}{wncy}{m}{n}
    \DeclareMathSymbol{\Sh}{\mathord}{mcy}{"58} 

\newcommand{\be}{\begin{enumerate}}
\newcommand{\ee}{\end{enumerate}}
\newcommand{\R}{\mathbb{R}}

\newcommand{\N}{\mathbb{N}}
\newcommand{\C}{\textnormal{C}}
\newcommand{\W}{\textnormal{W}}

\newcommand{\D}{\textnormal{D}}
\newcommand{\U}{\textnormal{U}}
\newcommand{\E}{\textnormal{E}}
\newcommand{\A}{\textnormal{A}}
\newcommand{\ep}{\varepsilon}

\newcommand{\M}{\textnormal{M}}

\renewcommand{\H}{\textnormal{H}}
\newcommand{\Ll}{\textnormal{L}_{\textnormal{loc}}}

\newcommand{\I}{\textnormal{I}}

\newcommand{\T}{\mathbb{T}}
\newcommand{\X}{\textnormal{X}}
\newcommand{\Id}{\textnormal{Id}}
\newcommand{\V}{\textnormal{V}}

\newcommand{\Y}{\textnormal{Y}}

\renewcommand{\div}{\textnormal{div}}

\newcommand{\ffi}{\varphi}
\newcommand{\dd}{\mathrm{d}}
\newcommand{\Ld}{\textnormal{L}}
\newcommand{\Ldiv}{\Ld_{\textnormal{div}}}
\newcommand{\Hdiv}{\H_{\textnormal{div}}}
\newtheorem{thm}{Theorem}[section]

\newtheorem{coro}[thm]{Corollary}
\newtheorem{propo}[thm]{Proposition}
\newtheorem{lem}[thm]{Lemma}
\newtheorem{rem}{Remark}[section]
\newtheorem{nota}{Notation}[section]
\newtheorem{defi}{Definition}[section]

\newcommand{\eps}{\varepsilon}

\newcommand{\conv}[2]{\operatorname*{\longrightarrow}_{#1 \rightarrow
    #2}}

\numberwithin{equation}{section}

\title{Large time behavior of the Vlasov-Navier-Stokes system on the torus}
\author{Daniel Han-Kwan\footnote{Centre de Math\'ematiques Laurent Schwartz (UMR 7640), Ecole Polytechnique, Institut Polytechnique de Paris, 91128 Palaiseau Cedex, France (\href{mailto:daniel.han-kwan@polytechnique.edu}{daniel.han-kwan@polytechnique.edu})}, Ayman Moussa\footnote{Sorbonne Université, Université Paris-Diderot, CNRS, INRIA, LJLL, F-75005 Paris, France  (\href{mailto:ayman.moussa@sorbonne-universite.fr}{ayman.moussa@sorbonne-universite.fr})} and Iv\'an Moyano\footnote{Universit\'e de Nice Sophia-Antipolis Parc Valrose, Laboratoire J.A. Dieudonn\'e, UMR7351, 06108 Nice Cedex 02, France  (\href{mailto:Ivan.Moyano@unice.fr}{Ivan.Moyano@unice.fr})}}

\begin{document}
\maketitle 

\begin{abstract}
We study the large time behavior of Fujita-Kato type solutions to the Vlasov-Navier-Stokes system set on $\T^3 \times \R^3$. %, of Leray type for $d=2$ and of Fujita-Kato type for $d=3$. 
Under the assumption that the initial so-called modulated energy is small enough, %compared to the viscosity in the Navier-Stokes equations, 
we prove that the distribution function converges to a Dirac mass in velocity, with exponential rate.
The proof is based on the fine structure of the system and on a bootstrap analysis allowing to get global bounds on moments.
\end{abstract}

\tableofcontents

\section{Introduction}
We consider the Vlasov-Navier-Stokes system in $\T^3 \times \R^3$:
\begin{align}
\label{eq:vlasov}\partial_t f + v \cdot \nabla_x  f + \div_v [f(u-v)] &=0,\\
\label{eq:ns}\partial_t u + (u\cdot \nabla)u -  \Delta u + \nabla p &= j_f-\rho_f u, \\
\label{eq:ns2}\div \, u & = 0,
%f(0,\cdot,\cdot) &= f_0(\cdot,\cdot), \\
%u(0,\cdot) &= u_0(\cdot),
\end{align}
where 
\begin{align*}
\rho_f(t,x) &:= \int_{\R^3} f(t,x,v)\,\dd v, \\
j_f(t,x) &:= \int_{\R^3} vf(t,x,v)\,\dd v.
\end{align*}
This system of nonlinear PDEs describes the transport of particles (described by their density function $f$) within a fluid (described by its velocity $u$ and its pressure $p$). It belongs to the broad family of \emph{fluid-kinetic systems}, which were introduced in the pioneering works of O'Rourke \cite{oro} and Williams \cite{wil}. Among all  possible couplings (we refer to the introduction of \cite{GHM} for other examples), the Vlasov-Navier-Stokes has been intensively studied because of both its physical relevance (see \cite{BGLM} for instance) and the mathematical challenges that it offers. The Vlasov-Navier-Stokes system is fully coupled: both unknowns $f$ and $u$ depend on each other. This is due to the Brinkman force (the source term in the fluid equation) and the drag acceleration (the inertial term in the kinetic equation). We refer to \cite{BGLM} for the physical justification of these, and to \cite{goderi,BDGR1,BDGR2,hill,hillmousu} for the (partial) mathematical derivation of the former. The physical constants are all normalized in \eqref{eq:vlasov} -- \eqref{eq:ns2}. %, except the viscosity $\nu>0$ which appears (only) in the fluid equation. 
%In some results of our study we will ask $\nu$ to be ``large enough'' and this suggests that all the other (here normalized) constants of the system remain of order $1$. The only constants affected by this assumption are the one appearing in the Brinkman force and the drag acceleration. Indeed, following \cite{BGLM}, we have for a particle of position $x$ and velocity $x$, a drag acceleration given by the formula
%\begin{align*}
%a(t,x,v) = \frac{6\pi \nu r}{m} (u(t,x)-v),
%\end{align*}
%while the Brinkman force which results from the whole spray is given by
%\begin{align*}
%F(t,x) = -m \int_{\R^3} f(t,x,v) a(t,x,v)\,\dd v,
%\end{align*}
%where $m = \frac43 \pi r^3 \rho_p$ is the mass of the particles and $r$ (resp. $\rho_p$) is their radius (resp. density). Therefore, the regime that we consider here corresponds to the scaling $r\sim 1/\nu$ and $\rho_p \sim \nu^3$. This particular setting affects only a limited number of results of our study, namely the ones resting on a ``large $\nu$ assumption''.

\vspace{2mm}

The mathematical analysis of the Vlasov-Navier-Stokes system has been for a long time focused on the existence of (weak or strong) solutions on rather academic domains \cite{bou-des-grand-mou,CK,art:wangyu} like the flat torus that we consider in this paper, or more realistic ones \cite{ham-98,bougramou}. Most of the previous results provide global existence of weak solutions in the following sense:  a Leray solution for the fluid equation and a renormalized one (in the sense of DiPerna-Lions \cite{DPL}) for the kinetic equation (for a more precise definition, see Definition~\ref{def:sol} below). These global weak solutions are all built by an approximation-compactness argument which is based on the kinetic energy dissipation of the system. More regular solutions can also be constructed. In 2D, thanks to the uniqueness result of \cite{HM3}, they coincide with the weak solutions. In 3D, regular solutions are only known to exist locally (see \cite{CK} for instance). This issue is mainly due to the Navier-Stokes part of the system.

Very few articles deal with the long time behavior of this system. At the formal level, one expects a monokinetic behavior in velocity for the distribution function (in other words, concentration to a Dirac mass in velocity), due to the damping of the fluid component and the friction term acting on the kinetic phase. %(see Subsection~\ref{subsec:fri}). 
This behavior however has never been completely proven for the Vlasov-Navier-Stokes system. The closest attempt is the paper \cite{CK} of Choi and Kwon in which a conditional theorem is provided: the monokinetic behavior is shown to occur under a boundedness assumption that has not been established for any non-trivial global solution up to now. We intend to fill this gap by using the functional introduced by Choi and Kwon in \cite{CK} and proving that this boundedness property (in fact, a stronger one) indeed holds, for appropriate solutions of the Vlasov-Navier-Stokes system, under the assumption that the initial data are (in a sense to be made precise) sufficiently close to equilibrium. 

Concerning the long-time behavior of other fluid-kinetic systems, when a Fokker-Planck dissipation is added in the kinetic equation, the situation is less involved because the equilibria are all Maxwellians, which are non-singular  and (at least locally) attract all solutions. This has been investigated for instance in \cite{gou-he-mou-zha,cardumou}. Without this dissipation term, apart  from \cite{CK}, we can mention the work of Jabin \cite{Jab} in which the Navier-Stokes is replaced by a stationary Stokes equation (and a different coupling term) and \cite{GHM} in which a specific geometry is considered for the Vlasov-Navier-Stokes system, allowing for non-singular stationary solutions. 

\vspace{2mm}

As far as our knowledge goes, the results that we present below constitute the first complete and rigorous proof of asymptotic monokinetic behavior for the Vlasov-Navier-Stokes system.

\subsection{Weak solutions of the Vlasov-Navier-Stokes system}

Let us start with a short review of the notion of weak solutions for the Vlasov-Navier-Stokes system, which will give us the opportunity to introduce some notations.

\begin{defi}\label{defi:func}
The \emph{kinetic energy} of the system \eqref{eq:vlasov} -- \eqref{eq:ns2} is given for $t\geq 0$ by
\begin{align}
\E(t) := \frac12\int_{\T^3}|u(t,x)|^2\,\dd x +\frac12\int_{\T^3\times\R^3} f(t,x,v)|v|^2\,\dd v\,\dd x,
\end{align}
and the \emph{dissipation} is defined as
\begin{align}
\label{eq:dissip}
\D(t) := \int_{\T^3\times\R^3} f(t,x,v)|u(t,x)-v|^2\,\dd v\,\dd x +  \int_{\T^3} |\nabla u(t,x)|^2 \, \dd x .
\end{align}
\end{defi}
%In the following, we shall sometimes write 
%$$
%\mathscr{E}(t)   = \widetilde{\mathscr{E}}(t)  +  \frac12 \int_{\T^3} |u(t,x)-\langle u(t) \rangle|^2 \, \dd x
%$$
%with 
%\begin{equation}
%\label{def:Etilde}
%\widetilde{\mathscr{E}}(t) := \frac{1}{2}\left[ \int_{\T^3 \times \R^3} f|v-\langle j_f \rangle|^2 \, \dd v  \dd x +\frac{1}{2} |\langle j_f \rangle-\langle u \rangle|^2\right].
%\end{equation}
The kinetic energy and dissipation stem  from the seminal papers on the Vlasov-Navier-Stokes sytem \cite{ham-98,bou-des-grand-mou}. %The modulated one is more recent and was first introduced in \cite{CK}.
%Let us also assume for simplicity that $u_0 \in W^{d,q}_x$ for $q$ large enough.
%Here $c_P^{-d}$ is the best constant in the P-W inequality (=1?) 
One can check that the identity
$$
\frac{\dd}{\dd t} \E(t) + \D(t) =0,
$$
formally holds, which paves the way for a theory of global weak solutions.

\begin{defi}\label{def:adm} We shall say that $(f_0, u_0)$ is an admissible initial condition if 
\begin{align}
\label{eq:initu}u_0 &\in\Ldiv^2(\T^3)= \{ \U \in \Ld^2(\T^3), \, \div \, \U=0\}, \\
\label{eq:initf}0\leq f_0 &\in\Ld^1 \cap\Ld^\infty(\T^3\times\R^3),\\
\label{eq:initfbis}(x,v)\mapsto f_0(x,v)|v|^2 &\in\Ld^1(\T^3\times\R^3),\\
\int_{\T^3 \times \R^3} f_0 \, \dd v \, \dd x &=1.
\end{align}
\end{defi}
\begin{rem}
The last condition does not play any role for what concerns the properties of existence, uniqueness and long time behavior that we are about to discuss. However, this normalization allows to simplify the formulas.
\end{rem}
We shall also denote 
$\H^1_\div(\T^3)=\H^1(\T^3) \cap \Ldiv^2(\T^3)$.

\begin{defi}\label{def:sol}
Consider an admissible initial data $(u_0,f_0)$ in the sense of Definition \ref{def:adm}. A weak solution of the Vlasov-Navier-Stokes system with initial condition $(u_0,f_0)$ is a pair $(u,f)$ with the regularity 
\begin{align*}
u &\in \Ll^\infty(\R_+;\Ld^2(\T^3))\cap\Ll^2(\R_+;\H^1_\div(\T^3)),\\
f&\in  \Ll^\infty(\R_+;\Ld^1\cap\Ld^\infty(\T^3\times\R^3)),\\
j_f-\rho_f u &\in\Ll^2(\R_+;\H^{-1}(\T^3)),
\end{align*}
with $u$ being a Leray solution of \eqref{eq:ns} -- \eqref{eq:ns2} (initiated by $u_0$) and $f$ a renormalized solution of \eqref{eq:vlasov} (initiated by $f_0$), and such that the following energy estimate holds for almost all $t\geq s\geq 0$ (including $s=0$),
\begin{align}
\label{ineq:nrj}\E(t) + \int_s^t \D(\sigma) \, \dd \sigma \leq \E(s),
\end{align}
where the functionals $\E$ and $\D$ are the energy and dissipation introduced in Definition~\ref{defi:func}.
\end{defi}
The existence of weak solutions $(u,f)$ (in the sense of Definition~\ref{def:sol}) to the Vlasov-Navier-Stokes system has been established in \cite{bou-des-grand-mou} (and even on general domains in \cite{art:wangyu,bougramou}).% In the bidimensional case, uniqueness holds under an additional decay condition on the initial distribution function.

\begin{defi}\label{def:decay}
We say that an initial condition satisfies the \emph{pointwise decay assumption} of order $q>0$ if
\begin{align}
\label{f0}
(x,v)\longmapsto (1+|v|^q)f_0(x,v) \in \Ld^\infty(\T^3\times \R^3),
\end{align}
and in that case we denote 
\begin{equation}
N_q(f_0):=\sup_{x\in\T^3,v\in\R^3} (1+|v|^q)f_0(x,v).
\label{eq:Nf0definition}
\end{equation}
\end{defi}

We finally introduce some useful notations for moments in velocity and averages on the torus.
\begin{defi}
\label{def:moments}For all $\alpha\geq 0$ and any measurable non-negative function $\textnormal{f}: \, \T^3 \times\R^3 \to \R_+$, we set
\begin{align}
m_\alpha \textnormal{f}(t,x) &:= \int_{\R^3} \textnormal{f}  |v|^\alpha  \, \dd v, \\
M_\alpha \textnormal{f} (t) &:= \int_{\T^3 \times \R^3} \textnormal{f}  |v|^\alpha  \, \dd v \,\dd x. 
\end{align}
 For any measurable non-negative function $\textnormal{h}: \, \T^3  \to \R^d$ (for any $d\in \N\setminus\{0\}$), we denote its average by
 \begin{equation}
\langle \textnormal{h} \rangle := \int_{\T^3} \textnormal{h} \, \dd x.
\end{equation}

\end{defi}

\subsection{Heuristics for the long time behavior}% for the decoupled equations}%{Long time behavior for the linearized Vlasov equations with friction}\label{subsec:fri}

In this paper, we focus on the description of the long time behavior of weak solutions to the Vlasov-Navier-Stokes system. To this end, it is enlightning to first have a look at  the linear Vlasov equation with friction, around the trivial equilibrium $(0,0)$. This reads
\begin{equation}
\label{eq:vlasovlin}\partial_t g + v \cdot \nabla_x  g - \div_v [g v] =0.
\end{equation}
Endowed with an initial condition $g_0$ at $t=0$, this equation admits the explicit solution
\begin{equation}
\label{eq:explilin}
g(t,x,v) = e^{3t}g_0(x-(e^t-1)v, e^t v).
\end{equation}

\begin{defi}
For $\U \in \R^3$, we denote by $\delta_{\U}$ the Dirac measure in velocity supported at $\U$, defined by 
\begin{align*}
\forall \varphi\in\mathscr{D}(\R^3),\qquad \langle \delta_{\U}, \varphi \rangle = \varphi(\U).
\end{align*}
\end{defi}

The long time behavior of the solution to~\eqref{eq:vlasovlin}  is explicit, as we observe from~\eqref{eq:explilin} that
$$
g(t,x,v) \rightharpoonup_{t \to +\infty} \left( \int_{\R^3} g_0 (x-v,v) \, \dd v\right) \otimes \delta_{0}.
$$
More generally, given $\U \in \R^3$, for the equation
\begin{equation}
\label{eq:vlasovlin2}\partial_t g + v \cdot \nabla_x  g + \div_v [g (\U-v)] =0,
\end{equation}
the long time behavior of the solution is also explicit and described by
$$
g(t,x,v)  - \left( \int_{\R^3} g_0 (x-v-t \U,v+\U) \, \dd v\right) \otimes \delta_{\U} \rightharpoonup_{t \to +\infty} 0
$$
The mechanism at stake in~\eqref{eq:vlasovlin} and~\eqref{eq:vlasovlin2} is a competition between transport and friction. Friction always wins in the end, causing 
concentration to a Dirac mass in velocity.
In view of this behavior, we may expect a similar concentration phenomenon  in velocity for the full Vlasov-Navier-Stokes system, at least in a regime close to some equilibrium.

It is actually even possible to push the heuristics a little further. Taking for granted that the kinetic phase concentrates in velocity, with the behavior $ f(t,x,v) \sim \rho_f(t,x) \otimes \delta_{ u (t,x)}$ as $t \to +\infty$, we observe in particular that the Brinkman force in the Navier-Stokes equations vanishes as $t \to +\infty$.  Since it is well known that the solution $u(t)$ of the Navier-Stokes without forcing  tends to homogenize to its average in space $\langle u \rangle(t)$, we may expect that $ f(t,x,v) \sim \rho_f(t,x) \otimes \delta_{\langle u\rangle (t)}$ as $t \to +\infty$.
In particular this entails $\langle j_f\rangle(t) \sim \langle \rho_f \rangle(t) \langle u\rangle (t)$ as $t \to +\infty$.
But then, by the conservation laws
$$
\int_{\T^3\times \R^3} f(t) \, \dd v \, \dd x =1, \quad
\langle u+j_f  \rangle (t) = \langle u_0+j_{f_0} \rangle,
$$
(see~\eqref{consL1} and~\eqref{consuj} in Lemma~\ref{conse}), we deduce that
$
\langle u\rangle(t) \sim  \frac{\langle u_0+j_{f_0} \rangle}{2}
$
as $t \to +\infty$. To summarize, it follows from this heuristic argument that one can expect  
$$f(t,x,v) \sim \rho_f(t,x) \otimes \delta_{\frac{\langle u_0+j_{f_0} \rangle}{2}}$$
as $t \to +\infty$,
that corresponds to concentration to the constant velocity $\frac{\langle u_0+j_{f_0} \rangle}{2}$.

\subsection{The modulated energy of Choi and Kwon}

In \cite{CK}, Choi and Kwon introduced a \emph{modulated} version of the energy $\E(t)$ of Definition~\ref{defi:func}:
\begin{defi}\label{def:moden}
We define the \emph{modulated energy} as
\begin{multline}
\mathscr{E}(t) := \frac{1}{2} \int_{\T^3 \times \R^3} f(t,x,v)|v-\langle j_f(t,x) \rangle|^2 \, \dd v\, \dd x \\
+ \frac12 \int_{\T^3} |u(t,x)-\langle u(t) \rangle|^2 \, \dd x +\frac{1}{4} |\langle j_f(t) \rangle-\langle u(t) \rangle|^2.
\end{multline}
%where, for any integrable function $g$ on the torus $\T^3$, we denote its average on $\T^3$ by $\langle g \rangle$. 
\end{defi}
It is proved in \cite{CK} that the identity
$$
\frac{\dd}{\dd t} \mathscr{E}(t) + \D(t) =0,
$$
formally holds. Controlling the modulated energy is interesting in view of the expected long time monokinetic dynamics for the kinetic phase, because of the following statement. 
\begin{lem}
\label{lem:W1}
 With the previous notations, we have that for all $t \geq 0$,
\begin{equation}
\label{W1trick}
\begin{aligned}
\W_1 \left(f(t), \rho_f(t) \otimes \delta_{\frac{\langle u_0 + j_{f_0}  \rangle}{2}}\right) + \left\| u(t) - \frac{\langle u_0 + j_{f_0}  \rangle}{2}\right\|_{\Ld^2(\T^3)}   &\lesssim ({\mathscr{E}}(t))^{1/2},
\end{aligned}
\end{equation}
where $\W_1$ is the Wasserstein(-1) distance.
\end{lem}
The definition and basic properties of the Wasserstein distance $\W_1$ are given in the Appendix (see Section \ref{sec:Wasserstein}). The proof of the previous lemma is postponed to Section~\ref{subsec:lemw1}.

\section{Main results}
%As explained earlier 
Our main result provides a sharp description of the long time behavior of weak solutions to the Vlasov-Navier-Stokes system.

%Thanks to \cite{HM3} we know that the Vlasov-Navier-Stokes system presented above is well-posed. Our aim is to describe the long-time behavior of its solution. In dimension $2$, as mentioned in \cite{HM3}, the term $j_f-\rho_f u$ belongs to $\Ll^2(\R_+;\H^{-1}(\T^2))$ and the fluid component $u$ of the system is in fact the corresponding Leray solution, so that $u\in\mathscr{C}^0(\R_+;\Ld^2(\T^2))$.  In dimension $3$, under the additional assumption that the 
%$\dot \H^{1/2}(\T^3)$  norm of $u_0$ is small enough, we shall also see that the corresponding Leray solution satisfies $u\in\mathscr{C}^0(\R_+;\Ld^2(\T^3))$: this fact is not trivial and appears to be related to the long time behavior of the solutions we consider.
%For the kinetic component, we see $f$ as a measure-valued curve and use the aforementioned Wasserstein distance to describe the evolution of $f(t)$.

\begin{thm}
\label{thm}
There exists $\C_\star>0 $ and a nondecreasing onto function $\ffi: \R_+\rightarrow\R_+$ such that the following holds.
Let $(u_0,f_0)$ be an admissible initial condition such that $N_q(f_0)<+\infty$ for some $q>4$, $M_\alpha f_0<+\infty$ for some $\alpha>3$ and $u_0\in\H^{1/2}(\T^3)$. Then, if
$$
\| u_0 \|_{\dot\H^{1/2}(\T^3)}<\frac{1}{\C_\star^2},
$$ 
and if the initial modulated energy $\mathscr{E}(0)$ is small enough, in the sense that
\begin{multline}\label{ineq:Emodsmall3D}
\ffi \left(N_q(f_0)+M_\alpha f_0+ \E(0)+\|u_0\|_{\H^{1/2}(\T^3)}+1  \right) \mathscr{E}(0)\\
 <\min\left(1,\frac{1}{\C_\star^2}-\|u_0\|_{\dot\H^{1/2}(\T^3)}^2\right),
\end{multline}
then for any weak solution $(u,f)$ to the Vlasov-Navier-Stokes system, there exists a profile $\rho^\infty \in \Ld^\infty(\T^3)$ and  $\lambda,\C_\lambda>0$ such that for all $t\geq 0$,
\begin{multline}
\label{eq-expodecay}
  \left\| u(t) - \frac{\langle u_0 + j_{f_0}  \rangle}{2}\right\|_{\Ld^2(\T^3)} +  \W_1\left(f(t),{\rho}^\infty\left(x- t\frac{\langle u_0+j_{f_0}\rangle}{2} \right)\otimes \delta_{\frac{\langle u_0+j_{f_0}\rangle}{2} }\right) \\
\leq \sqrt{\mathscr{E}(0)} \C_\lambda \exp \left(- \lambda t  \right),
\end{multline}
where $\W_1$ is the Wasserstein distance. 
\end{thm}
We refer to the solutions that we consider as \emph{Fujita-Kato} type, as we require small initial $\dot\H^{1/2}$ norm for the fluid velocity.

\begin{rem}
The constant $\C_\star$  is the universal constant given in Proposition~\ref{coro:ns3D}.
\end{rem}
  
 We deduce that when ${\langle u_0 + j_{f_0} \rangle}=0$, the distribution function $f(t)$ weakly converges to a stationary solution, %This implies that in this case the Leray solution  $(u,f)$ converges to  a stationary solution of the Vlasov-Navier-Stokes system. 
 %, as soon as the conclusion of  Theorem~\ref{realthm} holds (therefore, in particular, as soon as any of the assumptions of Theorem~\ref{thm} or Theorem~\ref{thm3} are verified).
whereas when ${\langle u_0 + j_{f_0} \rangle}\neq 0$, the asymptotic behavior is that of a travelling wave.

\begin{rem}
As already said, existence of weak solutions %both in dimensions $2$ and $3$ 
follows from \cite{bou-des-grand-mou} (note by the way that both the pointwise decay assumption and  the higher order Sobolev assumption are not relevant for this part).
%Uniqueness in dimension $2$ is a consequence of \cite{HM3}. %Uniqueness in dimension $3$
%under these assumptions is new, and is tightly related to the asymptotic behavior of the system.

\end{rem}

%\begin{rem}
%In dimension $3$, in the regime of general $\nu$, $\| u_0 \|_{\dot \H^{1/2}(\T^3)}$ is actually also small by an interpolation argument:
%$$
%\| u_0 \|_{\dot\H^{1/2}(\T^3)} \lesssim \| u_0 - \langle u_0 \rangle\|_{\Ld^2(\T^3)}^{\frac{2\eps}{1+2\eps}}\| u_0 \|_{\dot\H^{1/2+\eps}(\T^3)}^{\frac{1}{1+2\eps}} \lesssim \mathscr{E}(0)^{\frac{2\eps}{1+2\eps}}\| u_0 \|_{\dot\H^{1/2+\eps}(\T^3)}^{\frac{1}{1+2\eps}}.
%.
%$$ 
%\end{rem}

%\begin{rem}
%As a matter of fact, the results we prove are more precise in terms of the various smallness requirements. As will be clear from the proofs, the required smallness are with respect to the size of (semi-)norms of the initial data.
%\end{rem}
%
%\begin{rem} TO DO
%The rate of convergence is almost sharp considering that for the linearized equation~\eqref{eq:vlasovlin2}, one can check that
%$$
%\W_1\left(g(t,x,v), \left( \int_{\R^3} g_0 (x-v-t \U,v+\U) \, \dd v\right) \otimes \delta_{\U} \right) \lesssim e^{-t}.
%$$
%In~\eqref{kin-t-infini} and~\eqref{kin-t-infini3}, we reach the rate $e^{-\lambda t}$ for all $0<\lambda<1$.
%\end{rem}

\begin{rem}
The fact that the asymptotic state for the distribution function is a Dirac mass in velocity, and thus is singular, virtually forbids the use of standard PDE techniques, such as high order Sobolev energy estimates, to prove this result.
\end{rem}

\begin{rem}

This result proves  that for the Vlasov-Navier-Stokes system, the large time behavior on the torus is very different from that on a  domain with partially dissipative boundary conditions (and under adequate geometric control conditions): in \cite{GHM}, it is indeed proved that in the latter case there exist smooth non-trivial equilibria that are locally stable.

\end{rem}

Theorem~\ref{thm} will be a consequence of the following result, bearing on the large time behavior of the modulated energy $\mathscr{E}(t)$. %We refer to Lemmas~\ref{lem:W1} and~\ref{lem:decay fluid} in Section~\ref{sec:formal} to see why Theorem~\ref{realthm} indeed implies Theorems~\ref{thm} and~\ref{thm3}.
\begin{thm}
\label{realthm}
Under the assumptions of Theorem~\ref{thm}, the following holds.  
There exists $\lambda,\C'_\lambda>0$ 
such that for all $t \geq 0$,
 \begin{equation}
 \label{eq:expodecayE}
 \begin{aligned}
\mathscr{E}(t)&\leq \mathscr{E}(0)  \C'_\lambda  e^{- \lambda t}.
%{\mathscr{E}}(t)   &\leq  \C_\lambda e^{-\lambda t},
\end{aligned}
\end{equation}
%where $\lambda_0$ is a universal positive constant appearing in Lemma~\ref{lem:decay fluid}. 
Furthermore, we have the global bounds
 \begin{align}
\label{eq:rhobound}
\sup_{t\geq 0} \| \rho_f(t) \|_{\Ld^\infty(\T^3)} < +\infty,
 \end{align}
 and
 \begin{align}\label{ineq:nabu}
   \int_1^{+\infty} \| \nabla u(\tau) \|_{\Ld^\infty(\T^3)} \, \dd\tau \leq \eta(\mathscr{E}(0)),
 \end{align}
where $\eta$ is a continuous nonnegative function such that $\eta(0)=0$.
 %where $\ll 1$ refers to the smallness assumption on $\mathscr{E}(0)$.
%\begin{equation}
%\mathscr{E}(t)  \conv{t}{+\infty} 0, 
%\end{equation}
%with exponential decay.
\end{thm}

\begin{rem}
The constants  $\C_\lambda, \C'_\lambda$ appearing in Theorems~\ref{thm} and \ref{realthm} are uniform with respect to the various (semi-)norms of $u_0$ and $f_0$ that appear in the assumptions.
\end{rem}

It is actually even possible to describe the structure of the final density $\rho^\infty$.

\begin{propo}
\label{propo-infini}
For $\delta$ small enough, under the assumptions of Theorem~\ref{thm}, if furthermore $u\in\mathscr{C}^0(\R_+;\W^{1,\infty}(\T^3))$  and 
\begin{align}\label{ineq:assnab}
\int_0^{+\infty} \|\nabla u(\tau)\|_{\Ld^\infty(\T^3)} \,\dd \tau\leq \delta,
\end{align}
then there exists a vector field
\begin{align*}
  \R_+\times\T^3\times\R^3 &\longrightarrow \R^3\\
  (s,x,v) &\longmapsto {\Y}^s_{\infty,x,v},
\end{align*}
belonging to $\mathscr{C}^0(\R_+;\mathscr{C}^1(\T^3\times\R^3))$ and such that the following holds.
We have
\begin{align}
\label{eq-rhoinfini}
\rho^\infty(x)= \int_{\R^3} f_0 \left(\Y^0_{\infty,x,v},v\right)  \left|\det \mathscr{A} \left(\infty,x,v\right) \right|  \, \dd v,
\end{align}
with 
\begin{align}
\label{eq-infini}
 \mathscr{A}\left(\infty,x,v\right) = \I_3 + \int_0^{+\infty}  e^s \nabla u\left(\tau, \Y^\tau_{\infty,x,v} \right) \D_v \Y^\tau_{\infty,x,v} \, \dd \tau,
\end{align}
and  $s\mapsto \Y^s_{\infty,x,v}$ satisfies
\begin{multline}
\label{eq-Xinfini} 
\Y^s_{\infty,x,v} 
  = x -e^{-s}v  + \frac{\langle u_0+j_{f_0}\rangle}{2} \left( e^{-s}+  s\right) \\
 -\int_0^{+\infty} \Big[\mathbf{1}_{[0,s]}(\tau)e^{\tau-s}+\mathbf{1}_{\tau\geq s}\Big] \left( u(\tau, \Y^\tau_{\infty,x,v}) - \frac{\langle u_0+j_{f_0}\rangle}{2}\right) \, \dd\tau.
\end{multline}
\end{propo}
\begin{rem}
The assumption~\eqref{ineq:assnab} on $u$ is restrictive  in the sense that the integral goes down to the time $t=0$. Indeed, the parabolic regularization and the estimates obtained in Section~\ref{sec:higherorder} prove that with the assumptions of Theorem~\ref{thm} alone,  $u\in\mathscr{C}^0([\ep,+\infty);\W^{1,\infty}(\T^3))$ for all $\eps>0$. The assumption~\eqref{ineq:assnab} therefore requires higher regularity for the initial fluid velocity $u_0$.
It is also possible to avoid this extra regularity assumption, replacing $f_0$ by the value of $f$ at time $t=1$, and all  integrals starting from $s=0$ by the same starting from $s=1$. The relevant  assumption  replacing~\eqref{ineq:assnab}, namely
\begin{align}\label{ineq:assnab1}
\int_1^{+\infty} \|\nabla u(\tau)\|_{\Ld^\infty(\T^3)} \,\dd \tau\leq \delta,
\end{align}
is then obtained as a consequence of Theorem~\ref{realthm}, see \eqref{ineq:nabu}, when $\mathscr{E}(0)$ is taken small enough.
  \end{rem}
  
  \begin{rem}\label{rem:annonce}
Proposition~\ref{propo-infini} is proved in Section~\ref{sec-asymp}. For the sake of clarity the proof focuses on the case $\langle u_0+j_{f_0} \rangle =0$. The proof of the general case is  similar and adds in only a few lines of computations, see Remark~\ref{rem:Zt}.
\end{rem}

%where we recall $\widetilde{\mathscr{E}}$ is defined in~\eqref{def:Etilde}.

 %\bigskip

There are mainly two stabilization mechanisms at stake in the large time dynamics of solutions to the Vlasov-Navier-Stokes system. The first one is due to \emph{friction} in the Vlasov equation, that forces the distribution function to concentrate in velocity. 
The second stabilization mechanism comes from the \emph{dissipation} in the Navier-Stokes equations.
There is  a competition in the Navier-Stokes equations between this dissipation and the possible growth of the non-linearity  and the Brinkman force $F= j_f- \rho_f u$.
%As is well known, the non-linearity is (energy-)subcritical in dimension $2$ and supercritical in dimension $3$, which explains the higher Sobolev regularity assumption needed in  the latter case.
Loosely speaking, the smallness assumptions we make allow to tame  the influence of the forcing.  
%: the fact that these assumptions are more demanding in dimension $3$ comes from the fact that the forcing is somehow stronger due to less favorable exponents in the Gagliardo-Nirenberg-Sobolev inequalities.

As already briefly discussed in the introduction, thanks to the fine structure of the system, there happens to be a modulated energy/dissipation identity that follows from the energy identity and the conservation laws of the system, as exhibited by Choi and Kwon \cite{CK}. This identity somehow reflects 
the two stabilization mechanisms we have just discussed.
 %The dissipation controls the modulated energy under the \emph{conditional} control of the $\Ld^\infty_{t,x}$ norm of the local density $\rho_f$, yielding \emph{exponential} decay.

\subsection{The case of dimension $2$}

For the sake of conciseness and physical relevance, we  focus in this paper on the case of dimension $3$. However, with the same method that we develop, it is possible to study the Vlasov-Navier-Stokes system on $\T^2 \times \R^2$ with weaker regularity assumptions on the initial data. Namely, we can treat admissible data (in all the Definitions, statements or equations discussed in this section, one has to replace $\T^3, \R^3$ by $\T^2, \R^2$ when necessary), without requiring the higher $\dot{\H}^{1/2}$ regularity for $u_0$ like in dimension $3$, see Theorem~\ref{thm} (the fact that more stringent regularity assumptions are required in dimension $3$ is due to the well-known difficulties related to the resolution of the Navier-Stokes equations).
%
%\begin{rem}
%A similar result holds in dimension $2$ with the following improvements:
%\begin{itemize}
%\item The $\H^{1/2}$ higher regularity is not required.
%\item The Leray solution is unique, thanks to \cite{HM3}.
%\end{itemize}
%The fact that more stringent regularity assumptions are required in dimension $3$ is due to the well-known difficulties related to the resolution of the Navier-Stokes equations.
%We will not discuss the proof in dimension $2$ as it follows the same lines as the analysis outlined in this paper, with some simplifications due to the subcritical nature of Navier-Stokes in dimension $2$.
%\end{rem}
For the record, we gather in the following statement what we may obtain in dimension $2$.

\begin{thm}
There exists a nondecreasing onto function $\ffi: \R_+\rightarrow\R_+$ such that the following holds.
Let $(f_0, u_0)$ be an admissible initial condition such that $N_q(f_0)<+\infty$ for some $q>4$. Assume that the initial modulated energy $\mathscr{E}(0)$ is small enough, in the sense that
\begin{align}
\label{ineq:Emodsmall}\ffi \left(N_q(f_0)+ \E(0)+1  \right)\mathscr{E}(0) < 1.
\end{align}
Then the weak solution $(u,f)$ to the Vlasov-Navier-Stokes system satisfies~\eqref{eq-expodecay}.
\end{thm}

Note that the statement is also strengthened compared to dimension $3$ since there is uniqueness of the weak solution of the Vlasov-Navier-Stokes system: it has been indeed established in \cite{HM3} that in dimension $d=2$, under the pointwise decay assumption of order $q>4$ of Definition~\ref{def:decay} (and in fact an even less stringent condition is sufficient), uniqueness holds for weak solutions of the Vlasov-Navier-Stokes system. 

The proof developed in dimension $3$ applies \emph{mutatis mutandis}, with the following significant simplifications:
\begin{itemize}
\item the $\dot{\H}^{1/2}$ regularity for the fluid velocity is not required in order to get higher order energy estimates for positive times, and therefore in particular we do not need to propagate  $\dot{\H}^{1/2}$ estimates for all times;
\item we can rely on various estimates already proved in \cite{HM3};
\item several indices in the Sobolev embeddings are more favorable in dimension $2$.
\end{itemize}
Let us finally mention that Proposition~\ref{propo-infini} holds as well in dimension $2$. 

 \subsection{Outline of the proof and organisation of the paper}

To conclude this section, let us  provide a (non-technical) outline of the proof of Theorem~\ref{realthm}. This also gives the opportunity to describe how this paper is organized.

The purpose of the first Section~\ref{sec:formal} is to explain how Theorem~\ref{thm} can be deduced from Theorem~\ref{realthm}, and more strikingly how the proof of the latter boils down to \emph{one} single uniform estimate on the local density of the kinetic phase. In Section~\ref{subsec:conslaw} we gather conservation laws for the Vlasov-Navier-Stokes system. Section~\ref{subsec:lemw1} emphasizes the role of the modulated energy : we prove therein Lemma~\ref{lem:W1} which explains how the decay of this functional leads to  concentration in velocity for the particles. Sections~\ref{subsec:dissmod} and Sections~\ref{subsec:condlong} detail the following key observation of \cite{CK} : up to a control of the $\Ld^\infty(\R_+;\Ld^\infty(\T^3))$ norm of the local density $\rho_f = \int_{\R^3} f \, \dd v$, the modulated energy is essentially controlled by its dissipation, yielding exponential decay. We also explain, following an argument of Jabin \cite{Jab}, how one can recover the existence of the asymptotic profile $\rho^\infty$ appearing in \eqref{eq-expodecay}, once the exponential decay is established.

As a consequence of Section~\ref{sec:formal}, the proof of Theorem~\ref{realthm} (and therefore Theorem~\ref{thm}) relies only on obtaining the following global bound for $\rho_f$ :
\begin{equation}
\label{eq:globalrho-intro}
\sup_{t\geq 0} \| \rho_f(t) \|_{\Ld^\infty(\T^3)} <+ \infty,
\end{equation}
and in fact our bootstrap strategy actually will prove the same estimate for $j_f=\int_{\R^3} fv\,\dd v$.

In Section~\ref{sec:changesofvariables}, we present the main tools we used to obtain such bounds on moments. They are based on the method of characteristics, which allows, considering the characteristics curves $(\X, \V)$ solving the system
\begin{equation}
\label{eq:carac-intro}
\begin{aligned}
\dot{\X}(s;t,x,v) &= \V(s;t,x,v),\\
\dot{\V}(s;t,x,v) &= u(s,\X(s;t,x,v))-\V(s;t,x,v),
\end{aligned}
\end{equation}
with $(\X(t;t,x,v),\V(t;t,x,v))=(x,v)$, to write solutions to the Vlasov equation as
\begin{equation}
f(t,x,v) = e^{3t} f_0(\X(0;t,x,v), \V(0;t,x,v)).
\end{equation}
We deduce that
\begin{align}
\label{eq:rhof-intro}
\rho_f(t,x) = e^{3t} \int_{\R^3} f_0(t,\X(0;t,x,v),\V(0;t,x,v))\,\dd v.
\end{align}
In order to study~\eqref{eq:rhof-intro}, we rely on a change of variables in velocity, referred to as the  \emph{straightening} change of variables, namely $v \mapsto \V(0;t,x,v)$. It is not obvious that this map is a diffeomorphism. In, Section~\ref{sec:changesofvariables}, we provide a sufficient condition to ensure this: there exists a constant $\delta>0$ such that, if
\begin{align}\label{ineq:nabla5-intro}
\int_0^t \|\nabla u(s)\|_{\Ld^\infty(\T^3)}\,\dd s < \delta,
\end{align}
then indeed the straightening change of variable is admissible. Under this smallness condition, the outcome is the estimate
\begin{align*}
\| \rho_f\|_{\Ld^\infty(0,t; \Ld^\infty(\T^3))} \lesssim   N_q(f_0).
%, \\
%&\| j_f\|_{\Ld^\infty(0,t; \Ld^\infty(\T^3))} \leq 4  \left( \int_0^{t}  \|u(s)-\langle u(s)\rangle\|_\infty \,\dd s + e^{-t}\left(1+\int_0^t e^s |\langle u(s)\rangle|\dd s\right)\right)N_q(f_0),
\end{align*} 
%Compared to the previous change of variables, we need here a higher order control on $u$ but with a milder constraint on the decay rate. %We will able to obtain such a control in the case $\mathscr{E}(0)$ small.
Similar bounds for $j_f = \int_{\R^3} v f \, \dd v$ can be obtained as well.

This change of variables is inspired by that used by Bardos and Degond \cite{BD} for the study of global small solutions to the Vlasov-Poisson system on $\R^3 \times \R^3$.

\bigskip

As a consequence of Section~\ref{sec:changesofvariables}, the remaining task is now to prove that for small enough initial modulated energy, the estimate \eqref{ineq:nabla5-intro} holds for all $t\geq 1$ (small times are handled by local estimates). Sections~\ref{sec:LinfH1} to~\ref{sec:E(0)small} are dedicated to this task. To this end, we set up a bootstrap argument. Loosely speaking, we consider
\begin{align}
\label{def:tstar-intro}t^\star := \sup\left\{t\geq 1, \, \int_1^t \|\nabla u(s)\|_{\Ld^\infty(\T^3)}\,\dd s < \delta\right\},
\end{align}
 and the aim is to show that $t^\star =+\infty$.
The general strategy is as follows. Assuming $t^\star<+\infty$, we work on the interval of time $[1,t^\star]$.
We shall obtain regularity estimates for $u$ using higher order energy estimates for the Navier-Stokes equations and maximal parabolic estimates for the Stokes equations. Such bounds are not relevant in terms of decay in time but
on $[1,t^\star]$ we have thanks to the straightening change of variables 
\begin{align*}
\sup_{t\in[1,t^\star]} \| \rho_f(t) \|_{\Ld^\infty(\T^3)} \lesssim 1 .
\end{align*}
Therefore, by  Choi-Kwon's key observation, $\mathscr{E}(t)$ decays exponentially fast on $[1,t^\star)$.
 The idea is then to interpolate the higher regularity estimate with the pointwise $\Ld^2(\T^3)$ bound bearing on $u-\langle u \rangle$ which is provided by the exponential decay of the modulated energy. More precisely,
we use the Gagliardo-Nirenberg-Sobolev interpolation inequalities to obtain
\begin{equation}
\label{eq:GNSintro}
 \|u(s)-\langle u(s)\rangle\|_{\Ld^\infty(\T^3)} \lesssim \|\D^2 u(s)\|_{\Ld^2(\T^3)}^{\alpha}\|u(s)-\langle u(s)\rangle\|_{\Ld^2(\T^3)}^{1-\alpha},
\end{equation}
for $\alpha\in (0,1)$; we argue similarly for the control of $\nabla u$.

To apply  the previous bootstrap strategy, we need enough regularity and integrability on the solutions of the Vlasov-Navier-Stokes system. We prove in Section~\ref{sec:LinfH1} that any weak solution of the system instantaneously satisfies adequate estimates, which includes:
\begin{itemize}

\item a short time control of $\rho_f$ and $j_f$ in $\Ld^\infty(\T^3)$, using local bounds; 

\item $\Ld^\infty_t \H^1_x \cap \Ld^2_t \H^{2}_x$ estimates for $u$, on time intervals \emph{away} from zero, that is to say for $t\geq 1$
\begin{align}\label{ineq:estvnsgood-intro}
\|\nabla u(t)\|_{\Ld^2(\T^3)}^2 +\int_{1}^t \|\Delta u(s)\|_{\Ld^2(\T^3)}^2\,\dd s \lesssim 1+\sup_{s\in[1,t]} \| \rho_f(s)\|_{\Ld^\infty(\T^3)}.
\end{align}

\end{itemize}
%Here, we shall need to distinguish between dimension $2$ and $3$. In dimension $3$, 
We introduce the convenient notion of \emph{strong existence times}
in order to be able to propagate regularity.

% In Section~\ref{sec:nularge}, we obtain the global bound~\eqref{eq:glob1-intro}
% in the case $\nu \gg 1$ by interpolating between~\eqref{ineq:estvnsgood-intro} and~\eqref{ineq:fluid-intro} (using the Gagliardo-Nirenberg inequality~\eqref{eq:GNSintro}). We end up with
% the following inequalities.
% 
% \begin{itemize}
% 
% \item In dimension $2$, for all $t \geq 0$,
% \begin{align}
% \label{eq:rhosous-intro}
%\sup_{[0,t]} \| \rho_f\|_{\Ld^\infty(\T^2)} \lesssim \mathscr{E}(0)^{1/2}(1+ \sup_{[0,t]} \| \rho_f \|_{\Ld^\infty(\T^2)}^{1/2}),
%\end{align}
% 
% \item in dimension $3$, for all strong existence times $t \geq 0$,
%  \begin{align}
%\label{eq:rhosuper-intro}
%\sup_{[0,t]} \| \rho_f\|_{\Ld^\infty(\T^3)} \lesssim 1+\frac{\mathscr{E}(0)^{3/8}}{\nu^{9/8}}(1+ \sup_{[0,t]} \| \rho_f\|_{\Ld^\infty(\T^3)})^{9/8}.
%\end{align}
% 
% \end{itemize}
%Observe that \eqref{eq:rhosous-intro} is a \emph{sub-linear} inequality, which allows to conclude straightaway that~\eqref{eq:globalrho-intro} holds. On the contrary, in dimension $3$, \eqref{eq:rhosuper-intro} is \emph{super-linear}, and we rely on a bootstrap argument, requiring to take $\nu$ large enough.
%We finally have to check that all $t \geq 0$ are strong existence times.
%
In Section~\ref{sec:higherorder}, we start to implement the  interpolation strategy, relying this time on higher order maximal parabolic estimates for the Stokes equation. The outcome is a control of $\D^2 u$ in $\Ll^p(\R_+;\Ld^q(\T^3))$ by $(u\cdot\nabla)u$ and $j_f-\rho_f u$ in the same space. 

Then Section~\ref{sec:E(0)small} is dedicated to the proof of the global bound~\eqref{eq:globalrho-intro} : we explain therein how the previous control of $\D^2 u$ can be iteratively used to produce an estimate of the form
\begin{align*}
\int_1^{t^\star}   \|\nabla u(s)\|_{\Ld^\infty(\T^3)} \lesssim \mathscr{E}(0)^{\gamma}.
\end{align*}
%Combining this estimate with the straightening change of variables, the result is that 
Consequently, if $\mathscr{E}(0)$ is small enough, then we must have $t^\star= +\infty$, which concludes the proof of Theorem~\ref{realthm}.

\bigskip

Finally Section~\ref{sec-asymp} is devoted to the proof of Proposition~\ref{propo-infini} which provides a sharper description of the asymptotic behavior.
The analysis comes down to the study of the limit as $t\to \infty$ of characteristics (more precisely of renormalized versions of them). For the sake of clarity, the proof is written in the particular case $\langle u_0+j_{f_0}\rangle =0$ to lighten the computations (see Remark~\ref{rem:Zt}).
 
%
%Next, at the expense of taking $\nu$ large enough and $\mathscr{E}(0)$ small enough, we are able to prove a modified scattering result, giving a more precise description of the function $\overline{\rho}$.  We show the existence of \emph{limit characteristics}  $(\widetilde{\X}_{s,\infty}(x,v), \widetilde{\V}_{s,\infty}(x,v) )$ satisfying an explicit system of integral equations (see~\eqref{XVinfi}), such that 
% setting
% \begin{equation}
% \widetilde{\rho}(t,x):= \int_{\R^3} f_0 \left( \widetilde{\X}_{0,\infty}(x,v) - t \frac{\langle u_0 + j_{f_0} \rangle}{2},  \widetilde{\V}_{0,\infty}(x,v) \right)\, \dd v,
% \end{equation}
% we have 
% \begin{equation}
% \W_1 \left(f(t),  \widetilde{\rho}(t,x)\otimes \delta_{\frac{\langle u_0 + j_{f_0} \rangle}{2}} \right)   \conv{t}{+\infty} 0.
% \end{equation}
%  with sharp exponential decay (of course this yields $\overline\rho= \widetilde\rho$). It turns out that the characteristics  $(\widetilde{\X}_{s,\infty}(x,v), \widetilde{\V}_{s,\infty}(x,v) )$ are small perturbations of the vector field 
%  $$(x,v) \mapsto \left(x-v- (t-1)\frac{\langle u_0 + j_{f_0} \rangle}{2}, v\right),$$
%   which corresponds to the linearized dynamics.

\bigskip

To conclude the paper, Section~\ref{sec:appendix} is an Appendix where we provide some  reminders (in particular, we shortly review some well-known basic facts about the Wasserstein distance)
and justify $\H^1$ energy estimates for the Navier-Stokes equations with source.

\section{Conservation laws, energy dissipation identities and consequences}
\label{sec:formal}

% We gather in this section several identities and \emph{a priori} estimates for the Vlasov-Navier-Stokes system, which are valid in all dimensions:
% \begin{itemize}
% \item we first give basic conservation laws for the system ;
% \item we then introduce the modulated energy of the system which controls the distance between $u$ and its average in $\Ld^2(\T^3)$ but also quantifies the monokinetic behavior of $f$ (Lemma REF) ; ,
% \item eventually we prove that equipped with a global control of the $\Ld^\infty$ norm (in time and space) of the kinetic phase density $\rho_f$, it is
% possible to prove the exponential decay of the modulated energy (and thus recover Theorem~\ref{realthm}). Obtaining this global bound will be our main task which we will  focus on in the following of the paper.
% \end{itemize}

\subsection{Conservation laws}
\label{subsec:conslaw}

We discuss here some conservations laws for the Vlasov-Navier-Stokes system. 
We start by describing some basic ones in a first lemma: the first two ones come from the structure of the Vlasov equation alone, while the third one is a consequence of the fine structure of the complete system.
\begin{lem}
\label{conse}
Any weak solution (in the sense of Definition~\eqref{def:sol}) satisfies the following conservations laws. 
For almost all $t\geq 0$,
\begin{align}
\label{conssign} &f(t) \geq 0, \quad \textnormal{for almost all  } (x,v) \in \T^3 \times \R^3, \\
\label{consL1} &\int_{\T^3\times \R^3} f(t) \, \dd v \, \dd x = \int_{\T^3\times \R^3} f_0 \, \dd v \, \dd x =1, \\
\label{consuj} &\langle u+j_f  \rangle (t) = \langle u_0+j_{f_0} \rangle.
\end{align}
\end{lem}

\begin{proof}Considering the results of~\cite{bou-des-grand-mou}, the only item to prove is~\eqref{consuj}.
Let us assume that both $u$ and $f$ are smooth functions.
 Integrating the Vlasov equation against $v$, the conservation law satisfied by $j_f$ reads
\begin{equation}
\partial_t j_f + \div \left( \int_{\R^3} f v \otimes v \,dv \right) = \rho_f u -j_f,
\end{equation}
so that $\langle j_f \rangle$ satisfies
$$
\frac{\dd}{\dd t}  \langle j_f \rangle = \langle  \rho_f u - j_f \rangle.
$$
On the other hand, from~\eqref{eq:ns}, $\langle u \rangle$ satisfies
$$
\frac{\dd}{\dd t}  \langle u \rangle  = \langle j_f - \rho_f u \rangle,
$$
from which we deduce
$\frac{\dd}{\dd t} \langle u+j_f  \rangle =0$, and consequently~\eqref{consuj}.

\bigskip

In the general case, for the fluid equation we can directly use $\varphi=1$ as an admissible test function to recover a.e.
\begin{align*}
\langle u(t)\rangle - \langle u_0\rangle = \int_0^t \langle j_f-\rho_f u\rangle(s)\,\dd s.
  \end{align*}
For the kinetic equation we use an approximation argument relying on DiPerna-Lions theory \cite{DPL} for linear transport equations : we consider a sequence of nonnegative distribution functions $(f_n)_n$ solving the Vlasov equation with regularized vector fields $(u_n)_n$  and regularized and truncated initial conditions $(f_{0,n})_n$, and such that for all $n \geq 1$ and all $t \geq 0$,
$$
\int_{\R^3 \times \R^3} f_n |v|^2 \, \dd v \,  \dd x \lesssim 1.
$$
By the DiPerna-Lions theory, $f$ is the (strong) limit of $(f_n)_n$ in $\Ld^\infty(0,T; \Ld^p(\T^3 \times \R^3))$ for all finite values of $p$ ; interpolating with the previous bound we infer that $(j_{f_n})_n\rightarrow j_f$ strongly in $\Ld^\infty(0,T;\Ld^1(\T^3\times\R^3))$ and $(\rho_{f_n} u_n)_n \rightarrow \rho_f u$ at least in $\Ld^1(0,T;\Ld^1(\T^3))$. This is sufficient to pass to the limit in the following identity (which is justified at the regularized level)
\begin{align*}
\langle j_{f_n}(t) \rangle -\langle j_{f_0}\rangle  = \int_0^t \langle \rho_{f_n}u_n-j_{f_n}\rangle(s) \,\dd s,
\end{align*}
and finally, \eqref{consuj} follows for almost every $t$. $\qedhere$
\end{proof}

A straightforward consequence of \eqref{consuj}  in Lemma~\ref{conse} is the following formula:
\begin{lem}
\label{lem:moy}
For almost all $t \geq 0$:
\begin{equation}
\label{eqmoy}
\frac{1}{4} |\langle j_f \rangle(t)-\langle u \rangle(t)|^2 = \left| \langle j_f \rangle(t)-\frac{\langle u_0 + j_{f_0} \rangle}{2}\right|^2 =  \left| \langle u\rangle(t)-\frac{\langle u_0 + j_{f_0} \rangle}{2}\right|^2.
\end{equation}
\end{lem}

%\subsection{DiPerna-Lions theory}

\begin{rem}\label{rem:dp}
We shall use in this paper several times the DiPerna-Lions theory \cite{DPL}, in the same fashion as in the proof of Lemma~\ref{conse}. Thanks to the property of strong stability of renormalized solutions, this allows to systematically   argue as if both $f$ and $u$ are smooth when looking to establish estimates for the kinetic phase. The  argument, as already outlined in the proof of Lemma~\ref{conse}, is the following:
 \begin{itemize}
 \item consider an approximating sequence $(u_n)_n$ for $u$ and $(f_n)_n$ the associated solution to the Vlasov equation, with a regularized initial condition; 
  \item prove the desired estimate for the solution $f_n$ (without explicitly using the higher regularity of $f_n$ or $u_n$);
  \item pass to the limit using the strong stability property of renormalized solutions (and Fatou's lemma).%, using e.g. Fatou's Lemma.
 \end{itemize}
In the following, for brevity, we will never write down this argument explicitly but will repeatedly refer to the current remark.
\end{rem}

 %We refer to \cite[Lemma 1 and Proposition 4]{HM3},

%Lemma~\ref{conse}.
%Therefore, proving that $\mathscr{E}(t)$ decays exponentially fast implies that the same holds for the quantity $\W_1 (f(t), \rho_f(t) \delta_{\langle j \rangle})$.

%\subsection{Energy and modulated energy identities}
%\subsubsection{The classical energy}
%The theory of weak solutions for the Vlasov-Navier-Stokes system is based on study of the following functionnal
%\begin{align}
%\label{def:nrj}\textnormal{E}(t) &:= \frac{1}{2}\Big[\int_{\T^3\times\R^3} f(t)|v|^2\,\dd v\,\dd x + \int_{\T^3} |u(t)|^2\,\dd x\Big].
%\end{align}
%Indeed, introducing the dissipation
%\begin{align}
%\label{def:dissip}\textnormal{D}(t) &:=\Big[\int_{\T^3\times\R^3} f(t)|u(t)-v|^2\,\dd v\,\dd x + \nu \int_{\T^3} |\nabla_x u(t)|^2\,\dd x\Big],
%\end{align}
%one can prove the following identity
%\begin{align}\label{eq:nrj}
%\frac{\dd}{\dd t}\E(t) +\D(t) = 0,
%\end{align}
%which paves the way for the usual approximation-compactness scheme for proving existence of weak solutions, see \cite{bou-des-grand-mou} for instance.

%
%
%\begin{lem}The following holds for the weak solutions we consider. For all $t \geq 0$,
%\begin{equation}
%\label{eq:energy}
%\E(t) + \int_0^t \D(s) \, \dd s \leq  \E(0).
%\end{equation}
%\end{lem}

\subsection{The role of the modulated energy : proof of Lemma~\ref{lem:W1}}
\label{subsec:lemw1}
\begin{proof} 
%Using \eqref{kin-t-infini} in Theorem~\ref{thm} and and Theorem~\ref{realthm} we have
By the Monge-Kantorovich duality for the $\W_1$ distance (see Proposition~\ref{MK} in the Appendix), we have 
\begin{align*}
\W_1 \left(f(t), \rho_f(t) \otimes \delta_{\langle j_f\rangle} \right) &= \sup_{\| \nabla_{x,v} \phi\|_\infty\leq 1} \left\{\int_{\T^3 } \left ( \int_{\R^3}  f(t,x,v)  \phi(x,v) \, \dd v -  \rho(t,x) \phi(x,\langle j_f \rangle) \right)  \dd x\right\}\\
&= \sup_{\| \nabla_{x,v} \phi\|_\infty\leq 1} \left\{\int_{\T^3 \times \R^3}   f(t,x,v) ( \phi(x,v) -   \phi(x,\langle j_f \rangle) ) \dd v\, \dd x\right\}\\
&\leq   \int_{\T^3 \times \R^3} f|v-\langle j_f \rangle| \, \dd v\, \dd x. 
\end{align*}
We thus infer, using the Cauchy-Schwarz inequality, the normalization \eqref{consL1} and the definition of the modulated energy $\mathscr{E}(t)$
 \begin{align*}
 \W_1\left(f(t), \rho_f(t) \otimes \delta_{\langle j_f\rangle} \right)
 \leq \left(\int_{\T^3 \times \R^3} f|v-\langle j_f \rangle|^2 \, \dd v \, \dd x\right)^{1/2} \left(\int_{\T^3 \times \R^3} f \, \dd v \, \dd x\right)^{1/2} \leq \sqrt{2}\,\mathscr{E}(t)^{1/2}.
 \end{align*}
Likewise, 
\begin{align*}
\W_1 \Big(\rho_f\otimes \delta_{\langle j_f \rangle}, \rho_f \otimes \delta_{\frac{\langle u_0 + j_{f_0}  \rangle}{2}} \Big) &= \sup_{\| \nabla_{x,v} \phi\|_\infty\leq 1} \int_{\T^3}   \rho_f(t,x) \Big(\phi(x,\langle j_f \rangle) - \phi\left(x,\frac{\langle u_0 + j_{f_0}  \rangle}{2}\right) \Big)\dd x \\
&\leq \left|\langle j_f \rangle - \frac{\langle u_0 + j_{f_0}  \rangle}{2}  \right| \int_{\T^3 }   \rho_f(t,x) \dd x.
\end{align*}
We therefore deduce, using the normalization~\eqref{consL1} and the identity~\eqref{eqmoy}
\begin{align*}
\W_1 \Big(\rho_f\otimes \delta_{\langle j_f \rangle}, \rho_f \otimes \delta_{\frac{\langle u_0 + j_{f_0}  \rangle}{2}} \Big) \leq \frac12 \left|\langle j_f \rangle - \langle u \rangle  \right| &\leq  \mathscr{E}(t)^{1/2}, %\to_{t\to +\infty} 0,
\end{align*}
so that by triangular inequality we have established
$$
\W_1 \left(f(t), \rho_f(t) \otimes \delta_{\frac{\langle u_0 + j_{f_0}  \rangle}{2}}\right) \lesssim  \mathscr{E}(t)^{1/2}.
$$
On the other hand, using again~\eqref{eqmoy}, we can also estimate
\begin{align*}
\left\| u(t) - \frac{\langle u_0 + j_{f_0}  \rangle}{2}\right\|_{\Ld^2(\T^3)} &\leq \left\| u(t) - \langle u(t) \rangle \right\|_{\Ld^2(\T^3)} + \left\| \langle u(t) \rangle - \frac{\langle u_0 + j_{f_0}  \rangle}{2}\right\|_{\Ld^2(\T^3)} \\
&\leq \left\| u(t) - \langle u(t) \rangle \right\|_{\Ld^2(\T^3)} + \frac{1}{4} |\langle j_f \rangle-\langle u \rangle|^2,
%&\lesssim  ({\mathscr{E}}(t))^{1/2},
\end{align*}
and the result follows. $\qedhere$
\end{proof} 

%We now describe in the two following subsections two key consequences of the modulated energy/dissipation identity.

%Lemmas~\ref{lem:W1} and~\ref{lem:decay fluid} explain why Theorem~\ref{realthm} indeed implies Theorems~\ref{thm} and~\ref{thm3}.

\subsection{Dissipation of the modulated energy}\label{subsec:dissmod}
As already said in the introduction, Choi and Kwon noticed\footnote{As a matter of fact, they consider the more general Vlasov-\emph{inhomogeneous} Navier-Stokes system but we recover the system~\eqref{eq:vlasov}--\eqref{eq:ns2} as soon we stick to the case of constant fluid density.} in \cite{CK} that the modulated energy (see Definition~\ref{def:moden}) satisfies the following formal identity
\begin{align}
\label{eq:choikwon}\frac{\dd}{\dd t} \mathscr{E}(t) + \D(t) = 0.
\end{align}
At the level of weak solutions, we are only able to obtain the inequality version of~\eqref{eq:choikwon}, as stated in the next lemma.
\begin{lem}\label{lem:choikwon}
For any weak solution $(u,f)$ in the sense of Definition~\ref{def:sol}, for almost all $t\geq 0$, 
$$\mathscr{E}(t)-\E(t)=  -\frac14 |\langle u_0 +  j_{f_0}\rangle|^2 .$$ 
In particular, we have the following modulated energy/dissipation inequality for almost all $0\leq s \leq t <+\infty$ (including $s=0$),
\begin{equation}
\label{eq:choikwonweak}
\mathscr{E}(t) + \int_s^t \D(\sigma) \, \dd \sigma \leq  \mathscr{E}(s).
\end{equation}
\end{lem}
\begin{proof}
Let us first write 
\begin{align*}
\mathscr{E}(t)  =   {\E}(t) +  \frac{1}{2}  \left( \int_{\T^3 \times \R^3} f \, \dd v \, \dd x\right) \langle j_f \rangle^2 - \langle j_f \rangle^2  - \frac12 \langle u \rangle^2  + \frac{1}{4} |\langle j_f \rangle-\langle u \rangle|^2,
\end{align*}
that we can simplify in the following way thanks to \eqref{consL1}
\begin{align*}
\mathscr{E}(t)  &=   {\E}(t)  -\frac{1}{2} \langle j_f \rangle^2  - \frac12 \langle u \rangle^2  + \frac{1}{4} |\langle j_f \rangle-\langle u \rangle|^2 \\
&={\E}(t) -\frac14 |\langle j_f\rangle + \langle u \rangle|^2,
\end{align*}
so that $\mathscr{E}(t)-\E(t)$ does not depend on $t$ thanks to \eqref{consuj}.  Estimate \eqref{eq:choikwonweak} follows then from the energy estimate \eqref{ineq:nrj}. $\qedhere$
\end{proof}

\subsection{Conditional long time behavior}\label{subsec:condlong}

\begin{defi}
Let $c_P$ be the Poincar\'e constant, that is the best constant such that the  Poincar\'e-Wirtinger inequality holds: 
\begin{equation}
\label{eq:PW}
\|g-\langle g\rangle \|_{\Ld^2(\T^3)} \leq c_P \|\nabla g\|_{\Ld^2(\T^3)}, \quad \forall g\in \H^1(\T^3).
\end{equation} 
\end{defi}

The following result relating the dissipation and the modulated energy
is a variant of  \cite[Theorem 1.2]{CK}.

\begin{lem}
\label{decay+}
There exists a continuous nonincreasing function $\psi:\R_{+}\rightarrow\R_{+}$ such that the following holds, for any weak solution of the VNS system (in the sense of Definition~\ref{def:sol}) for which $\rho_f\in\Ll^\infty(\R_+;\Ld^\infty(\T^3))$. Fix $T>0$ and define
\begin{equation}
\label{eq:lambda}\lambda := \psi\left(\sup_{[0,T]}\| \rho_f(t) \|_{\Ld^\infty(\T^3)}\right).
\end{equation}
Then 
\begin{align}
\label{ineq:lowbo}\forall t\in[0,T],\quad \D(t) \geq \lambda \mathscr{E}(t),
\end{align}
and we have the exponential estimate 
\begin{align}
\label{ineq:condec}\forall t\in[0,T],\quad \mathscr{E}(t) \lesssim e^{-\lambda t}\mathscr{E}(0),
\end{align}
where $\lesssim$ depends only on $\lambda$.
%  In particular ... 

% Then  for all $t\in [0,T]$ and all $\alpha \in (0,1)$,
% \begin{equation}
% \textnormal{D}(t) \geq 2 (1-\alpha) \widetilde{\mathscr{E}}(t) + \left(\nu- \frac{1-\alpha}{\alpha}   \| \rho_f \|_{\Ld^\infty(0,T; \Ld^\infty(\T^3))} \right) \| u(t) -\langle u(t)\rangle \|_{\Ld^2(\T^3)}^2,
% \label{lower bound D}
% \end{equation}
% where we recall $\widetilde{\mathscr{E}}(t)$ is defined in~\eqref{def:Etilde}.
% As a result, 
% for all 
% $$
%  \frac{c_P(d)  \| \rho_f \|_{\Ld^\infty(0,T; \Ld^\infty(\T^3))}}{\nu+ c_P(d)  \| \rho_f \|_{\Ld^\infty(0,T; \Ld^\infty(\T^3))}} \leq \alpha <1,
% $$
% setting
% \begin{equation}
% \lambda_\alpha = \min \left( 2 (1-\alpha), \nu- \frac{1-\alpha}{\alpha}   \| \rho_f \|_{\Ld^\infty(0,T; \Ld^\infty(\T^3))}  \right)
% \end{equation}
% there is $\C_\alpha>0$ independent of $T$ such that for all $t\in [0,T]$,
% \begin{equation}
% {\mathscr{E}}(t) \leq \C_\alpha e^{- \lambda_\alpha   t}  \mathscr{E}(0),
% %\widetilde{\mathscr{E}}(t) \leq \mathscr{E}(0) \left(1  +  \frac{1-\alpha}{\alpha} \frac{1}{2 \lambda_0 }   \| \rho_f \|_{\Ld^\infty(0,T; \Ld^\infty(\T^3))}  \right) e^{-2(1-\alpha) t}, 
% \label{conditional exponential decay }
% \end{equation} 

% %for $\lambda_0$ given in Lemma~\ref{lem:decay fluid}. %by (\ref{ineq:fluid}).
\end{lem}

\begin{proof}
First we note that \eqref{ineq:lowbo} $\Rightarrow$ \eqref{ineq:condec}. Indeed, combining with estimate \eqref{eq:choikwonweak} of Lemma~\ref{lem:choikwon}, we get for almost all $0\leq s\leq t\leq T$,
\begin{align*}
\mathscr{E}(t) + \lambda \int_s^t \mathscr{E}(\sigma) \,\dd \sigma \leq \mathscr{E}(s),
\end{align*}
so that from Lemma~\ref{lem:gronexp} of the Appendix 
we get $\mathscr{E}(t) \lesssim \mathscr{E}(0) e^{-\lambda t}$, where $\lesssim$ depends only on $\lambda$. We therefore focus on \eqref{ineq:lowbo} and try to find $\lambda>0$ of the form \eqref{eq:lambda}.

\vspace{2mm}

Define 
\begin{align*}
\widetilde{\mathscr{E}}(t) := \mathscr{E}(t) - \frac12 \|u(t)-\langle u(t)\rangle \|_{\Ld^2(\T^3)}^2.
\end{align*}
The Poincaré-Wirtinger inequality gives us a constant $c_P>0$ such that 
\begin{align*}
\D(t) \geq \frac12 \int_{\T^3\times\R^3} f(t)|v-u(t)|^2\,\dd v\,\dd x  + \frac12 c_P \|u(t)-\langle u(t)\rangle\|_{\Ld^2(\T^3)}^2.
\end{align*}
Therefore to get \eqref{ineq:lowbo} for some $\lambda>0$, it is sufficient to prove that for some $\gamma,\beta>0$ we have
\begin{align}
\label{ineq:gambet}\int_{\T^3\times\R^3} f(t)|v-u(t)|^2\,\dd v\,\dd x \geq \gamma\, \widetilde{\mathscr{E}}(t) - \beta \|u(t)-\langle u(t)\rangle \|_{\Ld^2(\T^3)}^2,
\end{align}
with $\beta$ small enough (namely $\beta<c_P$) : in that case we have $\D(t)\geq \lambda \mathscr{E}(t)$ with $\lambda := \min(\gamma,c_P-\beta)$.

\vspace{2mm}

For the sake of clarity, we omit the time variable for a few lines. We also denote $\|\rho\|_{\infty,T}:=\sup_{[0,T]}\|\rho_f(s)\|_{\Ld^\infty(\T^3)}$. We start with the following identity
\begin{align*}
|v-u|^2 = |v-\langle u \rangle |^2 + 2 (v-\langle u \rangle) \cdot (\langle u \rangle - u) + |\langle u \rangle - u |^2, 
\end{align*} 
from which we infer
\begin{multline}
\label{vmoinsubis}\int_{\T^3 \times \R^3} f |v-u|^2 \, \dd v\, \dd x  = \int_{\T^3\times\R^3} f|v-\langle u\rangle|^2\,\dd v\,\dd x  + \int_{\T^3} \rho_f |\langle u \rangle - u |^2 \dd x \\+ 2 \int_{\T^3 \times \R^3} f (v - \langle u \rangle) \cdot (\langle u \rangle -   u ) \,\dd v \,\dd x.  
\end{multline} 
Now for any $\alpha  \in (0,1)$, Young's inequality entails that 
\begin{align*}
& 2 \int_{\T^3 \times \R^3} f  (\langle u\rangle - v) \cdot (u-\langle u\rangle)  \, \dd v\, \dd x \\
& \qquad \geq -  {\alpha}  \int_{\T^3 \times \R^3} f |v- \langle u\rangle |^2 \, \dd v\, \dd x - \alpha^{-1} \int_{\T^3} \rho_f |u- \langle u\rangle |^2 \, \dd x.
\end{align*} 
Combining with \eqref{vmoinsubis} we have therefore 
\begin{multline}
\label{vmoinsu}\int_{\T^3 \times \R^3} f |v-u|^2 \, \dd v\, \dd x  \geq (1-\alpha) \int_{\T^3\times\R^3} f|v-\langle u\rangle|^2\,\dd v\,\dd x\\  - (\alpha^{-1}-1)\int_{\T^3} \rho_f |\langle u \rangle - u |^2 \dd x.
\end{multline}
On the other hand we have
\begin{align*}
|v-\langle u \rangle |^2  =  |\langle j_f \rangle -  \langle u \rangle |^2 + 2 (v - \langle j_f \rangle) \cdot (\langle j_f \rangle -  \langle u \rangle) + |v - \langle j_f \rangle |^2,
\end{align*}
from which we deduce
\begin{align*}
\int_{\T^3\times\R^3} f|v-\langle u\rangle|^2 \,\dd v\,\dd x =  |\langle j_f\rangle-\langle u\rangle|^2+ \int_{\T^3\times\R^3} f|v-\langle j_f\rangle|^2 \,\dd v\,\dd x,
\end{align*}
where we used the normalization property~\eqref{consL1} and
\begin{equation*}
\int_{\T^3 \times \R^3} f (v - \langle j_f \rangle) \cdot (\langle j_f \rangle -  \langle u \rangle) \,\dd v\, \dd x = 0.
\end{equation*}
In particular, we have 
\begin{align*}
\int_{\T^3\times\R^3} f(t)|v-\langle u(t)\rangle|^2 \,\dd v\,\dd x \geq \widetilde{\mathscr{E}}(t).
\end{align*}
Since $\alpha\in(0,1)$ we deduce from \eqref{vmoinsu}
\begin{align*}
\int_{\T^3 \times \R^3} f |v-u|^2 \, \dd v\, \dd x  &\geq (1-\alpha) \widetilde{\mathscr{E}}(t)  - (\alpha^{-1}-1)\int_{\T^3} \rho_f |\langle u \rangle - u |^2 \dd x \\
&\geq (1-\alpha) \widetilde{\mathscr{E}}(t)  -(\alpha^{-1}-1)\|\rho_f\|_{\infty,T} \int_{\T^3} |\langle u \rangle - u |^2 \dd x,
\end{align*}
which is exactly \eqref{ineq:gambet} with $\gamma := 1-\alpha$ and $\beta = (\alpha^{-1}-1)\|\rho_f\|_{\infty,T}$. Picking $\alpha$ close enough to $1$ (to ensure $\beta<c_P$), $\lambda:=\min(\gamma,c_P-\beta)$ satisfies \eqref{ineq:lowbo}. To check that $\lambda$ can indeed be chosen of the form \eqref{eq:lambda} we have to make more explicit the choice of $\alpha$ by imposing for instance the condition $\beta = c_P/2$ above, that is $\alpha^{-1} = c_P/(2\|\rho\|_{\infty,T}) + 1$ which is a continuous, nonincreasing, nonvanishing function of $\|\rho\|_{\infty,T}$ : $\alpha$ is then continuous and increasing and $\lambda := \min(1-\alpha,c_P/2)$ is of the form \eqref{eq:lambda}. $\qedhere$\end{proof}

Once exponential decay of the modulated energy is ensured, one can prove the existence of an asymptotic profile $\rho^\infty$ for which we have the following convergence statement.

\begin{propo}
\label{prop-rhobar}
For any weak solution $(u,f)$ to the Vlasov-Navier-Stokes system for which $\sup_{t\geq 0} \| \rho_f(t) \|_{\Ld^\infty(\T^3)} <+ \infty$ and $\mathscr{E}(t)\rightarrow_{t \to +\infty} 0$ with exponential decay,  there exists a  profile $\rho^\infty \in \Ld^\infty(\T^3)$  such that
 \begin{align}
 \label{eq-expoprofile}
 \W_1\left(f(t), \rho^\infty\left(x -t \frac{\langle u_0 + j_{f_0} \rangle}{2}\right) \otimes \delta_{\frac{\langle u_0 + j_{f_0} \rangle}{2}}\right) \operatorname*{\longrightarrow}_{t\rightarrow+\infty} 0,
 \end{align}
exponentially fast.
\end{propo}
\begin{proof}
We rely on an argument of Jabin \cite{Jab} used in the context of the large time behavior of the Vlasov-Stokes system.
 The proof heavily relies on the exponential decay of the modulated energy. %Theorem~\ref{realthm}.  
 Recall the conservation of the mass
\begin{align*}
 \partial_t \rho_f  =-  \nabla \cdot j_f.
\end{align*}
For any smooth function $\psi\in\mathscr{C}^\infty(\T^3)$ we have therefore for $0 \leq s\leq t$
\begin{align*}
\int_{\T^3} \psi \rho_f(t)-\int_{\T^3}\psi \rho(s) = \int_s^t \int_{\T^3} \nabla \psi \cdot j_f(\tau)\,\dd \tau.
\end{align*}
Keeping in mind the definition of the Wasserstein distance (see Section~\ref{sec:Wasserstein}), one sees that the large time convergence of $\rho_f$ (which would imply that the Cauchy criterion is verified for this metric) is in a way or another linked with the decay of $j_f(\tau)$ as $\tau\rightarrow +\infty$. In the general case, this property is not expected, as $j_f$ is ``supposed'' to converge to $\rho_f \langle u_0+j_{f_0}\rangle/2$. This justifies to consider the following renormalized density
 \begin{align*}
\overline\rho_f (t,x) := \rho_f \Big(t,x + t   \frac{\langle u_0 + j_{f_0} \rangle}{2}  \Big),
\end{align*}
for which we have, denoting as well $\overline{j}_f := j_f \left(t,x + t   \frac{\langle u_0 + j_{f_0} \rangle}{2}  \right)$,
\begin{align*}
 \partial_t \overline\rho_f  = \nabla \cdot \left(\overline\rho_f \frac{\langle u_0+j_{f_0}\rangle}{2}-\overline{j}_f\right).
\end{align*}
The previous computation implies 
\begin{align*}
   \int_{\T^3} \psi\overline\rho_f (t)   - \int_{\T^3} \psi \overline\rho_f (s)   =  \int_s^t \int_{\T^3} \nabla \psi \cdot \left( \overline j_f - \overline \rho_f \frac{\langle u_0 + j_{f_0} \rangle}{2} \right)(\tau) \,\dd \tau,
\end{align*}
and the integrand is now expected to decay for large time. More precisely if $\|\nabla \psi\|_\infty \leq 1$ we have, by translation invariance of the integration over $\T^3$ 
\begin{multline*}
 \left|  \int_{\T^3} \psi\overline\rho_f (t)   - \int_{\T^3} \psi \overline\rho_f (s)  \right| \leq  \int_s^t \int_{\T^3} \left| \overline j_f - \overline \rho_f \frac{\langle u_0 + j_{f_0} \rangle}{2} \right|(\tau) \,\dd \tau  \\= \int_s^t \int_{\T^3} \left| j_f -  \rho_f \frac{\langle u_0 + j_{f_0} \rangle}{2} \right|(\tau) \,\dd \tau,
\end{multline*}
and we thus deduce by Cauchy-Schwarz inequality 
\begin{align*}
 \left|  \int_{\T^3} \psi\overline\rho_f (t)   - \int_{\T^3} \psi \overline\rho_f (s)  \right| \leq   \int_s^t \left(\int_{\T^3\times\R^3} f\right)^{1/2}\left(\int_{\T^3\times\R^3} f \left| v -  \frac{\langle u_0 + j_{f_0} \rangle}{2} \right|^2\right)^{1/2}(\tau) \,\dd \tau.
\end{align*}
On the one hand, thanks to Lemma~\ref{conse}, the integral of $f$ over $\T^3\times\R^3$ equals $1$. On the other hand, thanks to Lemma~\ref{lem:moy} we have
\begin{align*}
\left| v -  \frac{\langle u_0 + j_{f_0} \rangle}{2} \right|^2 \lesssim  \left| v -  \langle j_f\rangle \right|^2 + \left| \langle j_f\rangle  -  \frac{\langle u_0 + j_{f_0} \rangle}{2} \right|^2 = \left| v -  \langle j_f\rangle \right|^2 + \frac14 \left| \langle j_f\rangle  -  \langle u \rangle \right|^2.
  \end{align*}
All in all, using the the Definition~\ref{def:moden} of the modulated energy we have established for any $\psi\in\mathscr{C}^\infty(\R^3)$ such that $\|\nabla \psi\|_\infty \leq 1$
\begin{align*}
  \left|  \int_{\T^3} \psi\overline\rho_f (t)   - \int_{\T^3} \psi \overline\rho_f (s)  \right| &\lesssim   \int_s^t \left(\int_{\T^3\times\R^3} f \left| v - \langle j_f\rangle \right|^2\right)^{1/2}(\tau) \,\dd \tau + \int_s^t |\langle j_f\rangle - \langle u \rangle|(\tau) \,\dd \tau\\
  &\lesssim \int_s^t \mathscr{E}(\tau)^{1/2} \,\dd \tau.
\end{align*}
This estimate extends to Lipschitz functions $\psi$ satisfying $\|\nabla \psi\|_\infty \leq 1$ by a standard approximation argument and the Monge-Kantorovitch duality formula allows us to write 
    \begin{align}\label{eq:jab}
    \W_1(\overline\rho_f(t),\overline\rho_f(s)) \lesssim \int_s^t \mathscr{E}(\tau)^{1/2}\,\dd \tau.
    \end{align}
    The exponential decay of the modulated energy leads to integrability of $\mathscr{E}^{1/2}$ and therefore the  Cauchy criterion for $\overline \rho_f(t)$ is verified for $t\rightarrow +\infty$ : we recover in this way the convergence of $\rho^\infty_f(t) \rightarrow \rho^\infty$ for some measure $\rho^\infty$ as $t\rightarrow +\infty$. Since $t\mapsto \rho^\infty(t)$ is uniformly bounded in $\Ld^\infty(\R_+;\Ld^\infty(\T^3))$, we must have $\rho^\infty\in\Ld^\infty(\T^3)$. Note that the convergence is indeed exponential, thanks to the exponential decay of $\mathscr{E}^{1/2}$ : this can be seen when letting $t\rightarrow +\infty$ in \eqref{eq:jab}. Now by a change of variable we have 
    \begin{align}\label{eq:jab2}
    \W_1(\overline\rho_f(s),\rho^\infty) = \W_1\Big(\rho_f(s),\rho^\infty\Big(x-s\frac{\langle u_0+j_{f_0}\rangle}{2}\Big)\Big),
    \end{align}
    which concludes the proof.
  \end{proof}

\section{Changes of variables and $\Ld^\infty$ bounds on moments}
\label{sec:changesofvariables}

In this section we aim at establishing tools for obtaining bounds on the moments $\rho_f$ and $j_f$. We first obtain rough unconditional integrability results for $\rho_f$ and $j_f$ thanks to some interpolation estimates. Next, using some adequate change of variables in velocity, %and suitable estimates for these moments, %and distinguishing two different regimes, 
 we get refined estimates on $\rho_f$ and $j_f$, which can be controlled along the flow in the following way.  Assuming a suitable control on the quantity $\| \nabla u \|_{\Ld^1(0,t; \Ld^\infty(\T^3))}$, it is possible to prove that (cf. Lemma \ref{lem:rhoj})
\begin{align*}
&\| \rho_f\|_{\Ld^\infty(0,t; \Ld^\infty(\T^3))} \lesssim 1
, \nonumber \\
&\| j_f\|_{\Ld^\infty(0,t; \Ld^\infty(\T^3))} \lesssim  \left( \int_0^{t}  \|u(s)-\langle u(s)\rangle\|_{\Ld^\infty(\T^3)} \,\dd s + e^{-t}\left(1+\int_0^t e^s |\langle u(s)\rangle|\dd s\right)\right), 
\end{align*} which can be exploited in long time : the core of the bootstrap argument presented in Section~\ref{sec:E(0)small} is to prove that the control on $\nabla u$ holds as long as $\mathscr{E}(0)$ is small.

Many proofs in this section rely on the representation of the solution to the Vlasov equation using characteristics, which holds at least when $u$ is a smooth vector field.
\begin{defi}
Assume $u$ is smooth (say $\mathscr{C}^1$). We define the characteristic curves $\X(s;t,x,v)$ and $\V(s;t,x,v)$ associated with $u$ as the solution to the system of ODEs
\begin{equation}
\label{eq:carac}
\begin{aligned}
\dot{\X}(s;t,x,v) &= \V(s;t,x,v),\\
\dot{\V}(s;t,x,v) &= u(s,\X(s;t,x,v))-\V(s;t,x,v),
\end{aligned}
\end{equation}
with the initial condition $(\X(t;t,x,v),\V(t;t,x,v))=(x,v)$. 
\end{defi}
By the method of characteristics, for a smooth vector field $u$, we can write the solution $f$ to the Vlasov equation as
\begin{equation}
\label{eq:charcsec3}
f(t,x,v) = e^{3t} f_0(\X(0;t,x,v), \V(0;t,x,v)).
\end{equation}
As explained in Remark~\ref{rem:dp}, we then rely on DiPerna-Lions theory to ensure that the estimates we are able to prove with this representation formula still hold even if $u$ is not smooth enough. For instance, a rough bound on the $\Ld^\infty$ norm of $f$ can  be directly deduced from~\eqref{eq:charcsec3}. %the method of characteristics.
\begin{lem}
\label{lem:firstLinf}
For almost all $t \geq 0$,
\begin{align}
\|f(t)\|_{\Ld^\infty(\T^3\times\R^3)}\leq \|f_0\|_{\Ld^\infty(\T^3\times\R^3)} e^{3 t}.
\end{align}
\end{lem}
In the remaining paragraphs of this section we will systematically use the approximation procedure described in Remark~\ref{rem:dp}, without refering to it explicitely. This is in particular the case for each of the proofs which rely on the characteristic curves.

\subsection{Rough local bounds on moments}

We recall the notations $M_\alpha$ and $m_\alpha$ introduced in Definition~\ref{def:moments}.

\begin{lem}\label{lem:propmo}
Consider $\alpha \geq 1$ such that $u\in\Ll^1(\R_+;\Ld^{\alpha+3}\cap \W^{1,1}(\T^3))$ and $M_\alpha f_0<\infty$. Then $M_\alpha f(t)<\infty$ and for all $t>0$ and 
\begin{align}
\label{ineq:Malpha}M_\alpha f(t) \lesssim_{\alpha}\left(M_\alpha f_0+ e^{\frac{3 t}{\alpha +3}}\int_0^t \|u(s)\|_{\Ld^{\alpha +3}(\T^3) }\,\dd s\right)^{\alpha +3}.
\end{align} 
\end{lem}
\begin{proof}
Multiplying the Vlasov equation by $|v|^\alpha$ and integrating over $\T^3\times\R^3$, we get 
\begin{align}
\label{eq:malph}
\frac{\dd}{\dd t}M_\alpha f(t) + \alpha  M_\alpha f(t)= \alpha \int_{\T^3} u(t,x)\cdot m_{\alpha-1}(t,x)\,\dd x.
\end{align}
Recall that, for $0\leq \ell\leq k$, the following interpolation estimate 
\begin{align}
  \label{ineq:interpo}\|m_\ell g\|_{\Ld^{\frac{k+3}{\ell+3}}(\T^3)} \lesssim (M_k g)^{\frac{\ell+3}{k+3}} \|g\|_{\Ld^\infty(\T^3)}^{\frac{k-\ell}{k+3}},
\end{align} 
holds for any non-negative $g \in L^{\infty}(\T^3\times\R^3)$. In particular for $(\ell,k)=(\alpha-1,\alpha)$ we get 
\begin{align*}
\|m_{\alpha-1} g\|_{\Ld^{\frac{\alpha+3}{\alpha+2}}(\T^3)} \lesssim (M_\alpha g)^{\frac{\alpha+2}{\alpha+3}} \|g\|_{\Ld^\infty(\T^3)}^{\frac{1}{\alpha+3}}.
\end{align*}
We can control $\|g \|_{\Ld^\infty(\T^3)}$ by Lemma~\ref{lem:firstLinf}, so that using Hölder's inequality in~\eqref{eq:malph}, we infer 
\begin{align*}
\frac{\dd}{\dd t}M_\alpha f(t)^{\frac{1}{\alpha+3}} + \frac{\alpha}{\alpha+3}M_\alpha f(t)^{\frac{1}{\alpha+3}}\lesssim e^{\frac{3t}{\alpha + 3}}\|u(t)\|_{\Ld^{\alpha+3}(\T^3)},
\end{align*}
from which we get
\begin{align*}
\frac{\dd}{\dd t}\left\{e^{\frac{\alpha t}{\alpha +3}}M_\alpha f(t)^{\frac{1}{\alpha+3}}\right\} \lesssim e^t \|u(t)\|_{\Ld^{\alpha +3}(\T^3)},
\end{align*} 
from which \eqref{ineq:Malpha} follows.
\end{proof}

\begin{lem}
\label{lem:rhojeasy}
%For $d=2$, we have $\rho_f, j_f \in \Ld^\infty_{\textnormal{loc}}(\R_+;\Ld^\infty(\T^2))$. 
Assuming $M_3 f_0 <+\infty$, we have the following
\begin{itemize}
\item[(i)] $M_3 f \in \Ld^{\infty}_{\textnormal{loc}}(\R_+)$;
\item[(ii)] $\rho_f \in \Ld^\infty_{\textnormal{loc}}(\R_+;\Ld^2(\T^3))$ ;%, and for all $t\in\R_+$, $\|\rho_f(t)\|_{\Ld^2(\T^3)} \leq \C_0 e^{3t/2}$;
\item[(iii)] $ j_f \in \Ld^\infty_{\textnormal{loc}}(\R_+;\Ld^{3/2}(\T^3))$. %, and for all $t\in\R_+$, $\|j_f(t)\|_{\Ld^{3/2}(\T^3)} \leq \C_0 e^{t}$ ;
\end{itemize}

\end{lem}
\begin{proof}
%For the case $d=2$, we refer to \cite[Lemma 3]{HM3}.

By Lemma~\ref{lem:propmo}, we have
\begin{align*}
M_3f(t) \lesssim \left( M_3 f_0 +  e^{\frac{t}{2}} \int_0^t  \|u(s)\|_{\Ld^6(\T^3)}\,\dd s\right)^6.
\end{align*}
But, using the Sobolev embedding $\H^1(\T^3)\hookrightarrow\Ld^6(\T^3)$ and the Poincaré-Wirtinger inequality and the energy estimate \eqref{ineq:nrj}, we infer 
\begin{align*}
\int_0^t  \|u(s)\|_{\Ld^6(\T^3)}\,\dd s &\leq \int_0^t \|u(s)-\langle u(s)\rangle \|_{\Ld^6(\T^3)}\,\dd s + \sqrt{t} \E(0)^{1/2} \\
&\lesssim \sqrt{t}  \left(\int_0^t \|\nabla u(s)\|_{\Ld^2(\T^3)}^2\,\dd s\right)^{1/2} + \sqrt{t}  \E(0)^{1/2} \\
&\lesssim \sqrt{t}  \E(0)^{1/2}.
\end{align*}
This concludes the proof of $(i)$. By the interpolation estimate \eqref{ineq:interpo} for $(\ell,k)=(0,3)$ and $(\ell,k)=(1,3)$ we have 
\begin{align*}
\|\rho_f(t)\|_{\Ld^2(\T^3)} &= \|m_0 f(t)\|_{\Ld^2(\T^3)} \lesssim M_3 f(t)^{1/2} \|f(t)\|_{\Ld^\infty}^{1/2},\\
  \|j_f(t)\|_{\Ld^{3/2}(\T^3)} &\leq \|m_1 f(t)\|_{\Ld^{3/2}(\T^3)} \lesssim M_3 f(t)^{2/3} \|f(t)\|_{\Ld^\infty}^{1/3}.
\end{align*}
We therefore obtain $(ii)$ and $(iii)$ thanks to $(i)$ and Lemma~\ref{lem:firstLinf}.
\end{proof}

\subsection{The straightening change of variables}\label{subsec:straight}

We  discuss in this section the change of variables in velocity that will allow us, as explained at the beginning of this section, to prove long time estimates. The idea is to come down to the ``free'' case (that is to say to the characteristics associated with the vector field $(x,v) \to (v,-v)$ here), by using an appropriate diffeomorphism in velocity. In doing this, a smallness condition bearing on $\| \nabla u \|_{\Ld^1(0,t;\Ld^\infty(\T^3))}$ will naturally appear in our calculations. 

This change of variables is close in spirit to that employed in \cite{BD} by Bardos and Degond in the study of small data solutions to the Vlasov-Poisson system on $\R^3 \times \R^3$. 
We note however that the stabilization mechanism for Vlasov-Poisson on $\R^3 \times \R^3$ is based on the dispersion properties of the free transport operator, which is significantly different from that used in our work. %Indeed, the stabilization mechanism for the Vlasov-Navier-Stokes on $\T^3$ that we have been able to exploit rely essentially on the dissipation in the Navier-Stokes equations. 

We also mention that similar ideas were recently used in the context of the inertialess limit of the Vlasov-Stokes system in \cite{Hof}.

\begin{lem}
\label{charac}
Fix $\delta>0$ such that $\delta e^\delta<1/9$. Then, for any $t\in\R_+$ satisfying
\begin{align}
\int_0^t \|\nabla u(s)\|_{\Ld^\infty(\T^3)}\,\dd s \leq\delta,
\label{ineq:nabla5}
\end{align}
and any $x \in \R^3$, the map
\begin{equation*}
\Gamma_{t,x} : v \mapsto \V(0;t,x,v),
\end{equation*}
is a  $\mathscr{C}^1$-diffeomorphism from $\R^3$ to itself satisfying furthermore
\begin{equation}
\forall v \in \R^3,\quad|\det \D_v \Gamma_{t,x}(v) | \geq \frac{e^{3t}}{2}.
\label{eq:straight diffeo}
\end{equation}
\end{lem}

%\begin{rem}
%The constant $\textnormal{c}$ can be explicitly computed from the proof below.
%For instance, we can take $\textnormal{c}(2)=\frac{1}{5}$.
%\end{rem}

\begin{proof}
%The proof follows closely the method of Bardos-Degond.
The proof is directly inspired from the arguments outlined in \cite[Proposition 1 and Corollary 1]{BD}.

\noindent $(i)$ Consider a generic vector-valued flow $\Y^s_{t,z}:=\Y(s;t,z)$ associated with a smooth vector field $w(t,z)$ defined on $\R_+\times X$ and assume that $\|\textnormal{D}_z w(t)\|_{\Ld^\infty(X)}\leq 1+\psi(t)$, for some function $\psi\in \Ll^1(\R_+)$. We have $\partial_s \Y_{t,z}^s = w(s,\Y_{t,z}^s)$ which after differentiation with respect to $z$ (introducing $\Theta_{t,z}^s:=\textnormal{D}_z \Y_{t,z}^s$) leads to 
\begin{align*}
\partial_s \Theta_{t,z}^s = \textnormal{D}_z w(s,\Y_{t,z}^s) \cdot \Theta_{t,z}^s,
\end{align*}
from which we get by Gronwall's inequality for $s\leq t$ 
\begin{multline}
\label{eq:estimTHETA}
\|\Theta_{t,z}^s\|_{\Ld^\infty(X)} \leq \|\Theta_{t,z}^t\|_{\Ld^\infty(X)} \exp\left(\int_{s}^{t} \|\textnormal{D}_z w(\sigma)\|_{\Ld^\infty(X)} \dd \sigma\right) \\
\leq e^{t-s} \exp\left(\int_s^t |\psi(\sigma)|\,\dd \sigma\right),
\end{multline}
where we used $\Theta_{t,z}^t = \Id$. 

\vspace{2mm} 

Now, let us get back to our system. Introducing the state variable $z:=(x,v)$ which belongs to $X=\T^3\times\R^3$, the vector field $w(t,z) := (v,u(t,x)-v)$ satisfies the assumption for the above abstract result, since $\|\textnormal{D}_z w(t)\|_{\Ld^\infty(X)} \leq 1+\|\nabla u\|_{\Ld^\infty(\T^3)}$. If we denote by $(\X(s;t,z),\V(s;t,z))$ the characteristics associated with $u$, integrating the equation defining $s\mapsto \V(s;t,z)$ we have 
\begin{align}
\label{eq:Vs}\V(0;t,z) = e^{t} v - \int_0^t e^{s} u(s,\X(s;t,z))\,\dd s,
\end{align}
which leads to
\begin{align*}
\textnormal{D}_v \V(0;t,z) - e^{t} \Id = -\int_0^t  e^{s} \nabla u(s,\X(s;t,z))\D_v \X(s,t;z) \,\dd s.
\end{align*}
We thus infer from~\eqref{eq:estimTHETA} with $\psi = \|\nabla u\|_{\Ld^\infty(\T^3)}$ that 
$$
\|\D_v \X(s;t,z)\|_{\Ld^\infty(\T^3 \times \R^3)} \leq e^{t-s} \exp\left(\int_s^t  \|\nabla u(\tau) \|_{\Ld^\infty(\T^3)} \,\dd \tau\right),
$$
and thus that
\begin{align*}
\|e^{-t}\textnormal{D}_v \V(0;t,z)-\Id\|_{\Ld^\infty(\T^3\times\R^3)} &\leq \exp\left(\int_0^t \|\nabla u(s)\|_{\Ld^\infty(\T^3)} \dd s\right)\int_0^t \|\nabla u(s)\|_{\Ld^\infty(\T^3)} \dd s.
\end{align*}
In particular, if \eqref{ineq:nabla5} holds with $\delta>0$ such that $\delta e^\delta \leq \frac19$, then Lemma~\ref{lem:diff} applies and we can conclude.
\end{proof}

Thanks to the change of variables of Lemma~\ref{charac}, we deduce the following control on moments. 

\begin{lem}
\label{lem:rhoj} If assumption~\eqref{ineq:nabla5} of Lemma~\ref{charac} is satisfied, we have for almost all $t\geq 0$,
\begin{align}
\label{cont-rho1} &\| \rho_f(t)\|_{\Ld^\infty(\T^3)} \leq  2  I_q N_q(f_0) 
, \\
  \label{cont-j1}&\| j_f(t)\|_{\Ld^\infty(\T^3)} \leq 2  I_q e^{-t}  \left( \int_0^{t} e^s \|u(s)\|_{\Ld^\infty(\T^3)} \,\dd s +1\right)N_q(f_0),
\end{align} where $N_q(f_0)$ is given by \eqref{eq:Nf0definition} and 
\begin{align*}
I_q:=\int_{\R^3} \frac{1+|v|}{1+|v|^q}\,\dd v.
\end{align*}
\end{lem}

\begin{proof}
Let $(\X(s,t;x,v),\V(s,t;x,v))$ be the characteristics \eqref{eq:carac} associated with $u$. We start again from the representation formula
\begin{align*}
\rho_f(t,x) = e^{3t} \int_{\R^3} f_0 (\X(0;t,x,v), \V(0;t,x,v)) \, \dd v.
\end{align*}
By Lemma~\ref{charac}, the mapping $v\mapsto\Gamma_{t,x}(v)=\V(0;t,x,v)$ defines an admissible change of variable of which we deduce
\begin{align*}
\rho_f(t,x)  = e^{3t} \int_{\R^3} f_0 (\X(0;t,x,\Gamma_{t,x}(w)),w) \left|\D_v(\Gamma_{t,x})(\Gamma_{t,x}(w))\right| \, \dd w,
\end{align*} which implies (the control of the jacobian is given by Lemma~\ref{charac})
\begin{align}
\label{ineq:rhof}\| \rho_f(t) \|_{\Ld^\infty(\T^3)} \leq  2 N_q(f_0) I_q.
\end{align}
For $j_f$ we  proceed similarly and write the representation formula (valid for the same reasons) 
\begin{align*}
j_f(t,x)  = e^{3t} \int_{\R^3} \Gamma_{t,x}(w)   f_0 (\X(0;t,x,\Gamma_{t,x}(w)),w) \left|\D_v(\Gamma_{t,x})(\Gamma_{t,x}w)) \right| \, \dd w.
\end{align*} 
By definition of $\Gamma_{t,x}(w)$, we have  the  identity
\begin{align}
\label{eq:Vs2}w = e^{t} \Gamma_{t,x}(w) -  \int_0^t e^{s} u(s,\X(s;t,x,\Gamma_{t,x}(w)) )\,\dd s,
\end{align}
from which we deduce
\begin{align*}
|\Gamma_{t,x}(w)| \leq   e^{-t} \left[ |w| +    \int_0^t e^{s} \|u(s) \|_{\Ld^\infty(\T^3)} \,\dd s\right],
\end{align*} 
hence the claimed result. \end{proof}

In the next lemma, we study how the pointwise decay condition of Definition~\ref{def:decay} can be locally propagated.

\begin{lem}\label{delay}
Let $t_0>0$.
If $f_0$ satisfies \eqref{f0} and $u\in\Ll^1(\R_+;\H^1\cap\Ld^\infty(\T^3))$, then $f_{t_0}:=f(t_0)$ satisfies also \eqref{f0} and 
\begin{align*}
N_q(f_{t_0}) \lesssim (1+\|u\|_{\Ld^1(0,t_0;\Ld^\infty(\T^3))}^q) N_q(f_0).  
\end{align*}
\end{lem}
\begin{proof}
We write 
\begin{align*}
f({t_0},x,v) = e^{3t_0} f_0(\X(0;t_0,x,v),\V(0;t_0,x,v)).
\end{align*}
Thanks to the differential equation satisfied by $s\mapsto\V(s;t,x,v)$ we have
\begin{align}
\label{eq:V}\V(0;{t_0},x,v) &= e^{t_0} v- \int_0^{t_0} e^s u(s,\X(0;s,x,v)) \, \dd s\\
\nonumber&= e^{t_0}\left(v-\int_0^{t_0} e^{s-{t_0}} \langle u(s)\rangle \,\dd s\right) -\int_0^{t_0} e^s\Big(u(s,\X(0;s,x,v))-\langle u(s)\rangle \,\dd s\Big).
\end{align}
We deduce
\begin{align*}
|v|\leq |V(0;t_0,x,v)|+ \int_0^{t_0} \|u(s)\|_{\Ld^\infty(\T^3)}\,\dd s,
\end{align*}
and therefore 
\begin{align*}
(1+|v|^q)f({t_0},x,v) \lesssim e^{3t_0} (1+\|u\|_{\Ld^1(0,t_0;\Ld^\infty(\T^3))}^q) N_q(f_0).
\end{align*}
\end{proof}
This allows to obtain another version of Lemma~\ref{lem:rhoj} with a control like~\eqref{ineq:nabla5} starting only from some time $t_0>0$.
\begin{lem}
\label{lem:delay}
Let $t_0>0$.
With the same assumptions and notations as in Lemma~\ref{charac}, except that 
we replace~\eqref{ineq:nabla5} by
\begin{align}\label{ineq:nabla5-t0}
\int_{t_0}^t \|\nabla u(s)\|_{\Ld^\infty(\T^3)}\,\dd s \leq\delta,
\end{align}
we have for all $t \geq t_0$
\begin{align}
\label{cont-rho} &\| \rho_f(t)\|_{\Ld^\infty(\T^3)} \lesssim   N_q(f_0) (1+\|u\|_{\Ld^1(0,t_0;\Ld^\infty(\T^3))}^q)
, \\
\label{cont-j}&\| j_f(t)\|_{\Ld^\infty(\T^3)} \lesssim   e^{-t}  \left( \int_0^{t} e^s \|u(s)\|_{\Ld^\infty(\T^3)} \,\dd s +1\right)N_q(f_0) (1+\|u\|_{\Ld^1(0,t_0;\Ld^\infty(\T^3))}^q).
\end{align} 
\end{lem}
\begin{proof}
We can reproduce Lemma~\ref{charac} and Lemma~\ref{lem:rhoj} replacing the initial time $t=0$ by $t=t_0$ and thus $f_0$ by $f(t_0)$.
Using Lemma~\ref{delay}, we obtain the claimed estimates.
 %$N_q(f_0)$ by $N_q(f_0)(1+\|u\|_{\Ld^1(0,t_0;\Ld^\infty(\T^3))}^q)$.}
\end{proof}

\section{Regularity estimates for solutions of the Vlasov-Navier-Stokes system}
\label{sec:LinfH1}

This section is devoted to the following two tasks: 
\begin{itemize}
\item obtaining a precise short time control for the $\Ld^\infty$ norm of $\rho_f$ and $j_f$ (relying on local estimates and Lemma~\ref{delay});
\item obtaining %$\Ld^p_t \Ld^\infty_x$ estimates for $u$ (By Sobolev embedding, this will result from 
$\Ld^\infty_t \H^1_x \cap \Ld^2_t \H^{2}_x$ estimates for $u$, on time intervals \emph{away} from zero, as developed in Proposition~\ref{coro:ns3D}.
\end{itemize}
 Such estimates will be crucial to prove Theorem~\ref{realthm}, combined with the higher order estimates proved in Section~\ref{sec:higherorder}. 

We shall also introduce in this section the notion of \emph{strong existence times} (see Definition \ref{def:tadm}). Loosely speaking, this corresponds to times $t$ for which the solution $u$ of the Navier-Stokes equation is \emph{strong} on the interval of time $[0,t]$, which means in this context that it enjoys $\H^{1/2}(\T^3)$ regularity. A smallness criterion bearing both on $u$ and on the Brinkman force $j_f - \rho_f u$ (see~\eqref{ineq:smallnessvns}) will be used.
\begin{nota}\label{nota:lesssim}
From now on, $A\lesssim_{0}B$ will mean
\begin{equation*}
A \leq \ffi\left(\|u_0\|_{\H^{1/2}(\T^3)}+M_\alpha f_0+N_q(f_0)+\E(0) + 1 \right)B,
\end{equation*}
where $\ffi:\R_+\rightarrow\R_+$ is onto, continuous and nondecreasing, and $q>4$ and $\alpha>3$ are the exponents given in the statements of Theorem~\ref{thm} and Theorem~\ref{realthm}. Note that $\lesssim_{0}$ may depend on the integration exponents appearing in the inequality, but this will always be harmless. 
\end{nota}

\begin{nota}\label{nota:FS}
We will use the following notations:
\begin{align*}
F:=j_f-\rho_f u, \qquad S:=F-(u\cdot \nabla)u.
\end{align*}
\end{nota}

\subsection{Local estimates}
In this paragraph we establish local estimates on both the fluid and the particle densities. Namely, we prove $u\in\Ld^1_{\textnormal{loc}}(\R_+;\Ld^\infty(\T^3))$ and deduce from this estimate that $\rho_f,j_f \in\Ld^\infty_{\textnormal{loc}}(\R_+;\Ld^\infty(\T^3))$ and then $F\in\Ll^2(\R_+;\Ld^2(\T^3))$.
\begin{propo}\label{propo:estvns3D}
%Consider an admissible initial data $(u_0,f_0)$ for the case $d=3$. If furthermore $M_\alpha f_0 <\infty$ for some $\alpha>3$ and $u_0\in\H^{1/2}(\T^3)$, then 
We have $u\in\Ll^1(\R_+;\Ld^\infty(\T^3))$ and $\rho_f,j_f \in\Ll^\infty(\R_+;\Ld^\infty(\T^3))$. Moreover there exists a continuous nondecreasing function $\eta:\R_+\rightarrow\R_+$ such that 
\begin{align}
\label{eq:ulinf3}
  \|u\|_{\Ld^1(0,t;\Ld^\infty(\T^3))}&\lesssim_0 \eta(t),\\
  \label{ineq:01}   \|\rho_f\|_{\Ld^\infty(0,t;\Ld^\infty(\T^3))}+ \|j_f\|_{\Ld^\infty(0,t;\Ld^\infty(\T^3))}  &\lesssim_0 \eta(t).
\end{align}
\end{propo}
\begin{proof}
  In the proof we denote by $\eta$ a generic continuous function (as in the statement of the proposition), which may vary from line to line.

  \medskip
  
  Since $M_2 f_0<+\infty$ (see the Definition~\ref{def:adm} of admissible initial data), we have also $M_{3}f_0\lesssim M_2 f_0+M_\alpha f_0 <+\infty$. 
We infer from the proof of Lemma~\ref{lem:rhojeasy} that 
\begin{align*}
\|\rho_f (t)\|_{\Ld^2(\T^3)} + 
\|j_f (t)\|_{\Ld^{3/2}(\T^3)} \lesssim_0 \eta(t).
\end{align*}
In particular, recalling the notation $S=j_f -\rho_f u-(u\cdot \nabla)u$, we infer, using Hölder's inequality and the Sobolev embedding $\H^{1}(\T^3)\hookrightarrow\Ld^6(\T^3)$ and the energy estimate \eqref{ineq:nrj}, 
\begin{align*}
\int_0^t \|S(s)\|_{\Ld^{3/2}(\T^3)}^2\,\dd s\,\lesssim_0 \eta(t).
\end{align*}
Now, if $\mathbb{P}$ stands for the Leray projector (that is the projection on divergence free vector fields), let $w$ be the unique solution of 
\begin{align*}
\partial_t w - \Delta w &= \mathbb{P} S,\\
\div\,w &= 0,\\
w(0) &=0,
\end{align*}
so that $u-w =e^{t \Delta}u_0$. Since $u_0 \in \H^{\frac{1}{2}}(\T^3)$, we infer from \cite[Lemma 3.3]{GIM} that
$u-w\in \Ld^2(\R_+;\Ld^\infty(\T^3))$ with the estimate
\begin{align*}
\int_0^\infty \|(u-w)(s)\|_{\Ld^\infty(\T^3)}^2 \,\dd s \lesssim  \|u_0\|_{\H^{\frac{1}{2}}(\T^3)}^2.
\end{align*}
Thanks to the $\Ll^2(\R_+;\Ld^{3/2}(\T^3))$ estimate on $S$ that we obtained above, we infer from the continuity of $\mathbb{P}$ on $\Ld^{3/2}(\T^3)$ and the maximal regularity of the heat operator on the torus (see Corollary~\ref{coro:heat}) that $\Delta w\in\Ll^2(\R_+;\Ld^{3/2}(\T^3))$. Therefore, from a standard elliptic estimate, we deduce $\D^2 w\in\Ll^2(\R_+;\Ld^{3/2}(\T^3))$  and thus $w\in\Ll^2(\R_+;\Ld^p(\T^3))$ for all $p<\infty$, by Sobolev's embedding. We have even more precisely (keeping track of the different constants)
\begin{align*}
  \int_0^t \|w(s)\|_{\Ld^{p}(\T^3)}^2\,\dd s \lesssim_0 \eta(t).
\end{align*}
Up to now we have thus established (for any $p<\infty$) that $u\in\Ll^2(\R_+;\Ld^p(\T^3))$ with
 \begin{align}
\label{eq:up}\int_0^t \|u(s)\|_{\Ld^{p}(\T^3)}^2\,\dd s \lesssim_0 \eta(t).
\end{align}
In particular, we get $u\in\Ll^1(\R_+;\Ld^{\alpha+3}(\T^3))$. Using estimate \eqref{ineq:Malpha} of Lemma~\ref{lem:propmo} we first have 
\begin{align*}
M_\alpha f(t) \lesssim \left(M_\alpha f_0 + e^{\frac{3t}{\alpha+3}}\int_0^t \|u(s)\|_{\alpha+3}\,\dd s\right)^{\alpha +3} \lesssim_0 \eta(t).
\end{align*}
We use the interpolation estimate \eqref{ineq:interpo} with $k=\alpha$ and $\ell\in\{0,1\}$ to obtain this time 
\begin{align*}
\|\rho_f(t)\|_{\Ld^{\frac{\alpha+3}{3}}(\T^3)} + \|j_f(t)\|_{\Ld^{\frac{\alpha+3}{4}}(\T^3)} \lesssim_0 \eta(t),
\end{align*}
where the integration exponents are strictly larger than $3/2$. Using \eqref{eq:up} we can estimate $(u\cdot \nabla)u$ in some $\Ll^\gamma(\R_+;\Ld^r(\T^3))$ for $\gamma>1$ and $r>3/2$ leading to the following estimate on the source $S$ :
\begin{align*}
\int_0^t \|S(s)\|_{\Ld^r(\T^3)}^\gamma \,\dd s \lesssim_0 \eta(t).
\end{align*} 
Since $r>3/2$, using like before the maximal regularity of the heat operator we eventually infer by the Sobolev embedding $\W^{2,r}(\T^3)\hookrightarrow\Ld^\infty(\T^3)$
\begin{align*}
\int_0^t \|w(s)\|_{\Ld^\infty(\T^3)}^\gamma\,\dd s\lesssim_0 \eta(t).
\end{align*}
All in all, we have obtained that $u=(u-w)+w\in\Ll^1(\R_+;\Ld^\infty(\T^3))$. Finally using Lemma~\ref{delay} and the straightforward bound 
$$
\|\rho_f(t)\|_{\Ld^\infty(\T^3)} +\|j_f(t)\|_{\Ld^\infty(\T^3)} \lesssim N_q(f(t)),
$$
 we infer that both $\rho_f$, $j_f$ belong to $\Ll^\infty(\R_+;\Ld^\infty(\T^3))$ with the estimate
\begin{align*}
\|\rho_f(t)\|_{\Ld^\infty(\T^3)} +\|j_f(t)\|_{\Ld^\infty(\T^3)} \lesssim_0 \eta(t).
\end{align*}
\end{proof}
\begin{lem}\label{lem:FL2}
Recalling Notation~\ref{nota:FS}, we have $F\in\Ll^2(\R_+;\Ld^2(\T^3))$ and moreover
\begin{align*}
\int_0^t \|F(s)\|_{\Ld^2(\T^3)}^2\,\dd s \leq \min(\E(0),\mathscr{E}(0))\sup_{s\in[0,t]}\|\rho_f(s)\|_{\Ld^{\infty}(\T^3)}.
\end{align*}
\end{lem}
\begin{proof}
By Cauchy-Schwarz's inequality, we have a.e.,
\begin{align*}
|F| = \left|\int_{\R^3} f(v-u)\,\dd v\right| \leq \rho_f^{1/2} \left(\int_{\R^3} f|v-u|^2\,\dd v\right)^{1/2},
\end{align*}
from which we infer for almost all $s\geq 0$, 
\begin{align*}
\|F(s)\|_{\Ld^2(\T^3)}^2\leq \|\rho_f(s)\|_{\Ld^{\infty}(\T^3)}\D(s),
\end{align*}
where $\D$ is the dissipation introduced in \eqref{eq:dissip}. The estimate follows thus from the energy \eqref{ineq:nrj} and modulated energy \eqref{eq:choikwonweak} estimates.
\end{proof}

%\subsection{The 2D case}

%The 2D case is well-understood thanks to Proposition~\ref{coro:ns2D} and to \cite{HM3}.

%\begin{propo}\label{propo:estvns2D}

%Let $d=2$.
%For all $t_0 >0$, the following estimates hold.
%\begin{align}
%\label{eq:ulinf2}
%\|u\|_{\Ld^1(0,t_0;\Ld^\infty(\T^2))}&\leq \C_0, \\
%\label{ineq:leq1}\sup_{[0,t_0]}\left\{\|\rho_f(t)\|_{\Ld^\infty(\T^2)} + \|j_f(t)\|_{\Ld^\infty(\T^2)}\right\} &\leq \ffi\left(N_q(f_0)+ \E(0) + \frac{1}{\nu}\right),
%\end{align}
%for $\C_0>0$ depending on the initial condition and $t_0$   and some onto nondecreasing continuous function $\ffi$, and for almost all $t \geq t_0$:
%\begin{align}\label{ineq:estvns2D}
%\|\nabla u(t)\|_{\Ld^2(\T^2)}^2 +\nu\int_{t_0}^t \|\Delta u(s)\|_{\Ld^2(\T^2)}^2\,\dd s \leq \ffi\left(\E(0)+ \mathscr{E}(0)+\nu^{-1}\right)(1+\sup_{s\in[0,t]}\|\rho_f(s)\|_\infty),
%\end{align}
%where $\ffi:\R_+\rightarrow\R_+$ is the (universal) onto nondecreasing continuous function given in Corollary~\ref{coro:ns2D}
%\end{propo}
%\begin{proof}
%The uniqueness and 
%The regularity estimates~\eqref{eq:ulinf2} and~\eqref{ineq:leq1} are directly extracted from \cite[Lemma 3 and Corollary 1]{HM3}. As for the $\H^1$ energy estimate~\eqref{ineq:estvns2D}, we apply Proposition~\ref{coro:ns2D} and Lemma~\ref{lem:FL2}.
%\end{proof}

\subsection{Parabolic regularization for the fluid}

We state here a consequence of the parabolic regularization result of Proposition~\ref{coro:ns3D} of the appendix. This roughly establishes the instantaneous gain of two derivatives for the Navier-Stokes equation, if the right-hand side is square-integrable. However, such an estimate can only be obtained if a suitable smallness condition is satisfied.

\begin{propo}\label{propo:estvns3Dbis}
Assume that for some $T>0$ there holds
\begin{align}
\label{ineq:smallnessvns}\|u_0\|_{\H^{1/2}(\T^3)}^2  + \C_\star \int_0^T \|F(s)\|_{\H^{-1/2}(\T^3)}^2\,\dd s < \frac{1}{\C_\star^2},
\end{align}
where $\C_\star$ is the universal constant given by Proposition~\ref{coro:ns3D}. Then one has for  all $1/2\leq t \leq T$ the  estimate 
\begin{align}\label{ineq:estvns3D}
\|\nabla u(t)\|_{\Ld^2(\T^3)}^2 + \int_{1/2}^t \|\Delta u(s)\|_{\Ld^2(\T^3)}^2\,\dd s \lesssim  \E(0) \Big(1+\sup_{[0,t]} \|\rho_f(s)\|_{\Ld^\infty(\T^3)}\Big),
\end{align}
where $\lesssim$ depends only on $\C_\star$.
\end{propo}
\begin{proof}
If \eqref{ineq:smallnessvns} is indeed satisfied, we can directly use the well-posedness framework given by Proposition~\ref{coro:ns3D}. Thanks to Lemma~\ref{lem:FL2} we have also \eqref{ineq:capital} which here reduces to \eqref{ineq:estvns3D} because the decay of the energy \eqref{ineq:nrj} ensures $\A(t)\leq \E(0)$.
\end{proof}

%{\color{red}
%\subsection{Set-up of the bootstrap}
%
%}

\subsection{Strong existence times}

Thanks to Proposition~\ref{propo:estvns3D}, we know that $\rho_f$ and $j_f$ both belong to $\Ll^\infty(\R_+;\Ld^\infty(\T^3))$. We can therefore focus on the boundedness over $[1,+\infty)$. For this purpose, the following notations will be convenient.

\begin{defi}
\label{def:M}
We set for $t\geq 1$ 
\begin{align}
\M_{\rho_f}(t) &:= \sup_{[1,t]} \|\rho_f(s)\|_{\Ld^\infty(\T^3)},  \quad \M_{j_f}(t)  := \sup_{[1,t]} \|j_f(s)\|_{\Ld^\infty(\T^3)}, \\
 \M_{\rho_f, j_f} (t) &:=  \M_{\rho_f}(t) + \M_{j_f}(t).%\\
%\M_0 &:= \sup_{[0,1]} \|\rho_f(s)\|_{\Ld^\infty(\T^3)}+ \|j_f(s)\|_{\Ld^\infty(\T^3)}.
\end{align}
\end{defi}
%Note that the estimate \eqref{ineq:01} that we proved for the case $d=3$ admits obviously an equivalent for $d=2$, which is in fact simpler :
%\begin{align}
%\label{ineq:leq1}\sup_{[0,1]}\left\{\|\rho_f(t)\|_\infty + \|j_f(t)\|_\infty\right\} \leq \ffi\left(N_q(f_0)+ \E(0) + \frac{1}{\nu}\right).
%\end{align}
%for some onto nondecreasing continuous function $\ffi$. 
In order to use the regularization offered by Proposition~\ref{propo:estvns3Dbis}, we need to ensure that the smallness condition \eqref{ineq:smallnessvns} remains satisfied. For this reason, we introduce the following definition.

\begin{defi}[Strong existence times]\label{def:tadm}
A real number $T\geq 0$ will be said to be a \emph{strong existence time} whenever \eqref{ineq:smallnessvns} holds.

\end{defi}
The following lemma asserts that within our set of assumptions, we have a lower bound for strong existence times.
\begin{lem}\label{lem:strong1}
The smallness condition \eqref{ineq:Emodsmall3D} of Theorem~\ref{thm} suffices to ensure that $T=1$ is a strong existence time in the sense of Definition \ref{def:tadm}. 
\end{lem}

\begin{proof}
Using Lemma~\ref{lem:FL2} and estimate and~\eqref{ineq:01}, we straightforwardly  have
\begin{align*}
\int_0^1 \|F(s)\|_{\H^{-1/2}(\T^3)}^2 \, \dd s &\leq \int_0^1 \|F(s)\|_{\Ld^2(\T^3)}^2 \, \dd s  \\
&\leq \min(\E(0),\mathscr{E}(0))\sup_{s\in[0,1]}\|\rho_f(s)\|_{\Ld^\infty(\T^3)} \\
&\lesssim_0 \mathscr{E}(0),
\end{align*}
and recalling the meaning of $\lesssim_0$ (see Notation~\ref{nota:lesssim}), one sees that the smallness condition \eqref{ineq:Emodsmall3D} is indeed sufficient.
\end{proof}

\section{Estimates on the convection and the Brinkman force} 
\label{sec:higherorder}

Our ultimate bootstrap argument requires high order estimates bearing on $u$, for which, as in the proof of Proposition~\ref{propo:estvns3D}, we will see the Navier-Stokes equation as
\begin{align*}
\partial_t u -\Delta u = \mathbb{P}F - \mathbb{P}(u\cdot\nabla)u,
\end{align*}
where $\mathbb{P}$ is the Leray projector. We use the maximal regularity of the heat operator on the previous identity to get estimates on $\D^2 u$ in terms of the Brinkman force $F$ and the the convection term $(u\cdot \nabla)u$. In this short section we explain in Proposition~\ref{propo:regmaxex} the maximal regularity argument and give $\Ld^p_t\Ld^q_x$ estimates for the source terms in Lemma~\ref{lem:convinterp} and Lemma~\ref{lem:sourcelpr}. As explained in Corollary~\ref{coro:ulip}, these estimates are already sufficient to justify the $\Ld^1_t\W^{1,\infty}_x$ regularity needed to express the condition \eqref{ineq:nabla5-t0} (and a quantitative version with the required smallness will be provided afterwards).

\begin{propo}\label{propo:regmaxex}
Fix $a,b,r\in (1,\infty)$ and $\lambda > 0$.  For any $t\geq 1$, and any exponent $1\leq q \leq a,b$ there holds (with a possible infinite right-hand side)
  \begin{multline}\label{ineq:regmaxex}
    \int_1^t e^{-\lambda s} \|\D^2 u(s)\|_{\Ld^r(\T^3)}^q \,\dd s 
    \\\lesssim_0 \Phi(\lambda)\Big(1 +\|(u\cdot\nabla)u\|_{\Ld^a(1/2,t; \Ld^{r}(\T^3))}^q + \|F\|_{\Ld^b(1/2,t; \Ld^{r}(\T^3))}^q\Big),
  \end{multline}
  where $\Phi:\R_+\rightarrow\R_+$ is nonincreasing.
\end{propo}
\begin{proof}
Similarly to what we have done in the proof of Proposition~\ref{propo:estvns3D}, we introduce $w_1$ and $w_2$ as the unique divergence-free solutions on $[1/2,+\infty)$ of 
\begin{align*} 
\partial_t w_1 -  \Delta w_1 &= \mathbb{P}(u\cdot \nabla)u,\\
\partial_t w_2 -  \Delta w_2 &= \mathbb{P} F,
 \end{align*}
with initial conditions  $w_1(1/2)=w_2(1/2)=0$ so that, denoting $u_h:=u-(w_1+w_2)$, we have $u_h(t+1/2)=e^{t  \Delta}u(1/2)$. Now, thanks to the maximal regularity of the heat operator (see Corollary~\ref{coro:heat}) and the continuity of $\mathbb{P}$ on $\Ld^r(\T^3)$, we infer for $t\geq 1/2$
\begin{align}
\label{ineq:maxdemi1}\left(\int_{1/2}^t \|\D^2 w_1(s)\|_{\Ld^r(\T^3)}^{a}\,\dd s\right)^{1/a} &\lesssim  \left(\int_{1/2}^t \|(u\cdot \nabla u)(s)\|_{\Ld^r(\T^3)}^{a}\,\dd s\right)^{1/a}, \\
\label{ineq:maxdemi2}\left(\int_{1/2}^t \|\D^2 w_2(s)\|_{\Ld^r(\T^3)}^{b}\,\dd s\right)^{1/b} &\lesssim  \left(\int_{1/2}^t \|F(s)\|_{\Ld^r(\T^3)}^{b}\,\dd s\right)^{1/b}.
\end{align}
On the other hand, since $u_h(t+1/2)=e^{t\Delta}u(1/2)$, where we write
\begin{align*}
u(1/2,x)=:\sum_{k\in\mathbb{Z}^{3}} c_k e^{2i\pi k\cdot x}\in\Ld^2(\T^3),
\end{align*}
we have for $t\geq 1/2$
\begin{align*}
u_h(t,x) = \sum_{k\in\mathbb{Z}^{3}} c_k e^{-(2\pi |k|)^2(t-1/2)}e^{2i\pi k\cdot x},
\end{align*}
and in particular for $t\geq 1$ and $\ell\geq 1$
\begin{align*}
\|u_h(t)\|_{\dot{\H}^\ell(\T^3)}^2 &=\sum_{k\in\mathbb{Z}^3} |c_k|^2|k|^{2\ell} e^{-(2\pi|k|)^2(t-1/2)}\\
&\lesssim  \sum_{k\in\mathbb{Z}^3} |c_k|^2 e^{-|k|^2(t-1/2)}\\
&\lesssim   \|u(1/2)\|_{\Ld^2(\T^3)}^2 e^{-(t-1/2)},
\end{align*}
so that for any $\ell\geq 1$ we obtain
\begin{align}\label{ineq:heat}
\int_1^{+\infty} \|u_h(s)\|_{\dot{\H}^\ell(\T^3)}^q \, \dd s \lesssim    \|u(1/2)\|_{\Ld^2(\T^3)}^q \int_1^{+\infty} e^{- q(s-1/2)/2}\,\dd s \lesssim    \|u(1/2)\|_{\Ld^2(\T^3)}^q.
\end{align}
By the energy estimate \eqref{ineq:nrj}, we have $\|u(1/2)\|_{\Ld^2(\T^3)} \lesssim_{0}1$, so using \eqref{ineq:heat} for $\ell$ large enough, we infer
\begin{align}
\label{ineq:maxdemi3}\left(\int_{1}^t \|\D^2 u_h(s)\|_{\Ld^r(\T^3)}^{q}\,\dd s\right)^{1/q} &\lesssim_0 1.
  \end{align}
Using the decomposition $u=w_1+w_2+u_h$ and combining \eqref{ineq:maxdemi1}, \eqref{ineq:maxdemi2} and \eqref{ineq:maxdemi3}, we infer by Hölder's inequality the estimate \eqref{ineq:regmaxex}.
  \end{proof}

\begin{lem}\label{lem:convinterp}
There exists $a\in(2,4)$ and $r_a>2$ such that the following interpolation estimate holds for $t\geq 1$:
\begin{align}
  \label{ineq:convinterp}\| (u\cdot \nabla)u\|_{\Ld^a(1/2,t; \Ld^{r_a}(\T^3))}
  \lesssim_0 1+\M_{\rho_f,j_f}(t).
\end{align}
\end{lem}
\begin{proof}
The proof boils down to the  interpolation inequality
  \begin{align*}
  \| (u\cdot \nabla)u\|_{\Ld^a(1/2,t; \Ld^{r_a}(\T^3))}
  \leq \|u\|_{\Ld^\infty(1/2,t; \Ld^6(\T^3))}\|\nabla u\|_{\Ld^2(1/2,t; \Ld^6(\T^3))}^{\frac{2}{a}}\|\nabla u\|_{\Ld^\infty(1/2,t; \Ld^2(\T^3))}^{1-\frac{2}{a}}.
\end{align*}
Indeed, if the later is satisfied, since $t$ is a strong existence time, we have thanks to the regularization estimate \eqref{ineq:estvns3D} and the energy estimate \eqref{ineq:nrj}, together with the Sobolev embedding $\H^1(\T^3)\hookrightarrow\Ld^6(\T^3)$,
\begin{align*}
\|u\|_{\Ld^\infty(1/2,t; \Ld^6(\T^3))}+\|\nabla u\|_{\Ld^2(1/2,t; \Ld^6(\T^3))}+\|\nabla u\|_{\Ld^\infty(1/2,t; \Ld^2(\T^3))} \lesssim_0 1+\M_{\rho_f,j_f}(t).
  \end{align*}
To justify the interpolation above, notice that for any $a>2$, we have by Hölder inequality and interpolation $[(2,6),(\infty,2)]_\theta$,
\begin{align*}
  \|u\cdot\nabla u\|_{\Ld^a(1/2,t; \Ld^{r_a}(\T^3))} \leq \|u\|_{\Ld^\infty(1/2,t; \Ld^6(\T^3))} \|\nabla u\|_{\Ld^2(1/2,t; \Ld^6(\T^3))}^\theta \|\nabla u\|_{\Ld^\infty(1/2,t; \Ld^2(\T^3))}^{1-\theta},
\end{align*}
with the following equality 
\begin{align*}
\left(\frac1a,\frac{1}{r_a}\right) = \left(0,\frac{1}{6}\right)+\theta\left(\frac12,\frac{1}{6}\right)+(1-\theta)\left(0,\frac12\right).
\end{align*}
We deduce $\theta=2/a$. %and we recover the exponents in estimate \eqref{ineq:convinterp}.
From the previous identity we also deduce the value of $r_a$, because $\frac{1}{r_a}=\frac{1}{6}(1+\frac{2}{a})+\frac12(1-\frac{2}{a})$. In the limit case $a=2$ we get $r_a=3$, so that taking $|a-2|$ small enough we have indeed $r_a>2$ and $a\in(2,4)$.
\end{proof}

\begin{lem}\label{lem:sourcelpr}
For any finite $b>4$, the following estimate holds for some $r_b>3$ and all strong existence times $t\geq 1$ :
\begin{align}
\label{ineq:sourcelpr}\|F\|_{\Ld^b(1/2,t; \Ld^{r_b}(\T^3))}\lesssim_{0} 1+\M_{\rho_f,j_f}(t)^{\frac32-\frac2b}.
\end{align}
\end{lem}
\begin{proof}
Thanks to Lemma~\ref{lem:FL2} and \eqref{ineq:01} we have 
\begin{align}
\label{ineq:sourcebis}\|F\|_{\Ld^2(1/2,t; \Ld^{2}(\T^3))} \lesssim 1+\M_{\rho_f}(t)^{1/2} \leq 1+ \M_{\rho_f, j_f} (t)^{1/2}.
\end{align}
By interpolation $[(2,2);(\infty,6)]_\theta$, we have
\begin{align}
\label{ineq:interpF}\|F\|_{\Ld^b(1/2,t; \Ld^{r_b}(\T^3))} \leq \|F\|_{\Ld^2(1/2,t; \Ld^{2}(\T^3))}^\theta \|F\|_{\Ld^\infty(1/2,t; \Ld^{6}(\T^3))}^{1-\theta},
\end{align}
where $\theta$ and $r_b$ are defined by the equality $(\frac1b,\frac{1}{r_b})=\theta(\frac12,\frac12)+(1-\theta)(0,\frac16)$ from which we get $\theta=2/b$ and $\frac{1}{r_b}=\frac{2}{3b}+\frac16$ ; we notice that $b>4$ implies $r_b>3$. 

By the triangle inequality, we get
\begin{align*}
\|F\|_{\Ld^\infty(1/2,t; \Ld^{6}(\T^3))}&=\|j_f-\rho_f u\|_{\Ld^\infty(1/2,t; \Ld^{6}(\T^3))} \\
& \lesssim_0 (1+\M_{\rho_f, j_f} (t))(1+\|u\|_{\Ld^\infty(1/2,t; \Ld^{6}(\T^3))}).
\end{align*}
Using the Sobolev embedding $\H^1(\T^3)\hookrightarrow\Ld^6(\T^3)$ together with \eqref{ineq:estvns3D} and the energy estimate \eqref{ineq:nrj} we have $\|u\|_{\Ld^\infty(1/2,t;\Ld^6(\T^3))}\lesssim_0 \M_{\rho_f,j_f}(t)^{1/2}$ which implies 
\begin{align*}
\|F\|_{\Ld^\infty(1/2,t; \Ld^{6}(\T^3))} \lesssim_0 1+ \M_{\rho_f, j_f}(t)^{3/2}.
\end{align*}
Combining the previous estimate with \eqref{ineq:sourcebis} in \eqref{ineq:interpF} we therefore get 
\begin{align*}
\| F \|_{\Ld^b(1/2,t; \Ld^{r_b}(\T^3))} \lesssim 1+\M_{\rho_f, j_f}(t)^{3/2-\theta},
\end{align*}
which is exactly \eqref{ineq:sourcelpr} because $b=2/\theta$.
\end{proof}
\begin{coro}\label{coro:ulip}
For any strong existence time $t\geq 1$, one has $\nabla u \in\Ld^1(1,t;\Ld^\infty(\T^3))$.
\end{coro}
\begin{proof}
Thanks to Proposition~\ref{propo:estvns3D}, the right-hand sides of estimates~\eqref{ineq:convinterp} and~\ref{ineq:sourcebis} are finite. %for strong existence times, thanks to the regularization estimate \eqref{ineq:estvns3D} and the energy estimate \eqref{ineq:nrj}. 
By  Lemmas~\ref{lem:convinterp} and~\ref{lem:sourcelpr}, we can therefore take $r>3/2$ in \eqref{ineq:regmaxex}  and thus, by Sobolev's embedding and H\"older's inequality, we finally obtain the claimed regularity.
  \end{proof}

\section{Exponential decay of the modulated energy} %{Proof of Theorem~\ref{realthm} in the case $\mathscr{E}(0)$ small} %in dimension $3$}
\label{sec:E(0)small}
In this section, we finish the proof of Theorem~\ref{realthm} by setting up  a bootstrap procedure. Define
\begin{align}
\label{def:tstar2}t^\star := \sup\left\{\text{strong existence times }t\text{ such that  }\, \int_1^t \|\nabla u(s)\|_{\Ld^\infty(\T^3)}\,\dd s < \delta\right\}. 
\end{align}
where $\delta$ is given in Lemma~\ref{charac}. Thanks to the change of variables of Section~\ref{sec:changesofvariables}, we have that $\M_{\rho_f, j_f}(t)\lesssim_{0} 1$
on for $t<t^\star$ (see Proposition~\ref{prop:boundM}). The main goal will be to prove that $t^\star=+\infty$.
In order to do so, we shall combine the higher order estimates of Section~\ref{sec:higherorder} with the exponential decay estimates provided by Lemma~\ref{decay+}. %to prove that everything is controlled by quantities depending only on the initial data.
\begin{propo}
\label{prop:boundM}
We have $t^\star >1$. Moreover, for any $t<t^\star$, one has $\M_{\rho_f, j_f}(t)\lesssim_{0} 1$.
\end{propo}
\begin{proof}
By a view of the proof of Lemma~\ref{lem:strong1} (reducing $\mathscr{E}(0)$ if necessary), we remark that for $\eps>0$ small enough, $t=1+\ep$ is a strong existence time, and Corollary~\ref{coro:ulip} ensures that for $t$ close enough to $1$, the inequality $\int_1^t \|\nabla u(s)\|_{\Ld^\infty(\T^3)}\,\dd s<\delta$ is satisfied, ensuring $t^\star>1$.

\vspace{2mm}

For $t\in[1,t^\star)$ we can invoke Lemma~\ref{lem:delay} with $t_0=1$ and \eqref{eq:ulinf3}, to obtain that $\M_{\rho_f}(t)\lesssim_{0} 1$ and 
\begin{align*}
\|j_f(t)\|_{\Ld^\infty(\T^3)} \lesssim_0 e^{-t} \int_1^t e^s \|u(s)\|_{\Ld^\infty(\T^3)}\,\ \dd s.
\end{align*}
Thanks to Sobolev's embedding $\H^2(\T^3)\hookrightarrow\Ld^\infty(\T^3)$ we infer 
\begin{align*}
\int_1^t e^s \|u(s)\|_{\Ld^\infty(\T^3)}\,\dd s \lesssim \int_1^t e^s \|u(s)\|_{\Ld^2(\T^3)}\,\dd s + \int_1^t e^s \|\D^2 u(s)\|_{\Ld^2(\T^3)}\,\dd s,
\end{align*}
and therefore (using Cauchy-Schwarz's inequality)
\begin{multline*}
\int_1^t e^s \|u(s)\|_{\Ld^\infty(\T^3)}\,\dd s 
\lesssim (e^t-1)\sup_{[1,t]} \|u(s)\|_{\Ld^2(\T^3)} \\+ \left(\int_1^t e^{2s}\,\dd s\right)^{1/2}\left(\int_1^t \|\D^2 u(s)\|_{2}^2\,\dd s\right)^{1/2}.
\end{multline*}
Thanks to \eqref{ineq:estvns3D} and the energy estimate \eqref{ineq:nrj} we eventually infer 
\begin{align*}
e^{-t} \int_1^t e^s \|u(s)\|_{\Ld^\infty(\T^3)}\,\dd s \lesssim_{0} 1+ \M_{\rho_f}(t),
\end{align*}
and we have already proved that $\M_{\rho_f}(t) \lesssim_{0} 1$. We deduce that  $\M_{j_f}(t)\lesssim_{0} 1$ and this concludes the proof.
\end{proof}

%{$\mathscr{E}(0)$ small, $\nu$ arbitrary}
%
%
% \begin{nota}\label{nota:dim2nuarbi}
% In this paragraph the smallness condition on the initial data is only expressed in terms of $\mathscr{E}(0)$. For this reason, the symbol $\lesssim$ will refer to constants of the form 
% \begin{align*}
% \ffi\left(\|u_0\|_2+\E(0)+N_q(f_0)+\frac{1}{\nu}\right),
% \end{align*}
%  where $\ffi:\R_+\rightarrow\R_+$ is continuous, nondecreasing and onto. This function will always be independent of the initial data and $\nu$ but may depend on some of the exponents that appear in the estimates.
% \end{nota}

We now combine Proposition~\ref{propo:regmaxex} with Lemma~\ref{decay+}.

\begin{lem}\label{lem:decayfinal}
  Assume that $t_\star<\infty$. For any $\alpha\in[1/2,1)$, $c\in[1,\infty)$ and any finite $a,b\geq \max(1,c\alpha)$, the following estimate holds (with a possible infinite right-hand side)
      \begin{multline}
\label{ineq:regmaxexbis}    \left(\int_1^{t^\star} \|\nabla u(s)\|_{\Ld^p(\T^3)}^c \,\dd s\right)^{1/c}  \\\lesssim_0 \mathscr{E}(0)^{\frac{1-\alpha}{2}}\Big(1 +\|(u\cdot\nabla)u\|_{\Ld^a(1/2,t^\star; \Ld^{r}(\T^3))}^{\alpha } + \|F\|_{\Ld^b(1/2,t^\star; \Ld^{r}(\T^3))}^{\alpha}\Big)
  \end{multline}
  for $p\in[1,\infty]$ and $r\in(1,\infty)$ satisfying
  \begin{align}
\label{eq:idp}\frac1p=\frac13+\alpha\left(\frac{1}{r}-\frac23\right)+\frac{1-\alpha}{2}.
\end{align}
  \end{lem}
  \begin{proof}
  Owing to Lemma~\ref{decay+}, if $t^\star<+\infty$, there is, on $[0,t^\star]$, an exponential decay of the modulated energy with decay rate $\lambda^\star$. The Gagliardo-Nirenberg-Sobolev estimate of Theorem~\ref{lem:gns} for $(j,m,q)=(1,2,2)$ allows us to write for any $\alpha\in[1/2,1)$ and $s\geq 1$
\begin{align*}
\|\nabla u(s)\|_{\Ld^p(\T^3)} \lesssim \|\D^2 u(s)\|_{\Ld^r(\T^3)}^\alpha \|u(s)-\langle u(s)\rangle \|_{\Ld^2(\T^3)}^{1-\alpha},
\end{align*}
for $p,r$ satisfying \eqref{eq:idp}. By definition of the modulated energy and using its exponential decay on $[1,t^\star]$, we have therefore
\begin{align*}
\|\nabla u(s)\|_{\Ld^p(\T^3)} \lesssim \mathscr{E}(0)^{\frac{1-\alpha}{2}}  e^{-\lambda s}\|\D^2 u(s)\|_{\Ld^r(\T^3)}^\alpha,
\end{align*}
for $\lambda = \lambda^\star (1-\alpha)$.
We apply Proposition~\ref{propo:regmaxex} to infer that for any exponent $c$ such that $c\alpha\leq a,b$
  \begin{multline*}
    \int_1^{t^\star} \|\nabla u(s)\|_{\Ld^p(\T^3)}^c \,\dd s 
    \\
    \lesssim_0 \Phi(\lambda^\star)\mathscr{E}(0)^{c\frac{1-\alpha}{2}}\Big(1 +\|(u\cdot\nabla)u\|_{\Ld^a(1/2,t^\star; \Ld^{r}(\T^3))}^{c\alpha } + \|F\|_{\Ld^b(1/2,t^\star; \Ld^{r}(\T^3))}^{c\alpha}\Big),
  \end{multline*}
  where $\Phi$ is nonincreasing. But by Lemma~\ref{decay+}, $\lambda^\star$ itself is a nonincreasing function of $\M_{\rho_f}(t^\star)\lesssim_0 1$, which yields \eqref{ineq:regmaxexbis}.

    \end{proof}

    \begin{lem}\label{lem:almost}
      There exists $\gamma>0$ such that, if $t^\star<+\infty$, then the following estimate holds
      \begin{align}\label{ineq:thatsit}
\int_1^{t^\star} \|\nabla u(s)\|_{\Ld^\infty(\T^3)}\,\dd s \lesssim_0 \mathscr{E}(0)^\gamma.
        \end{align}
      \end{lem}
      \begin{proof}
  We start by combining Lemma~\ref{lem:decayfinal} with Lemma~\ref{lem:convinterp} and Lemma~\ref{lem:sourcelpr}. Since (by Proposition~\ref{prop:boundM}) $\M_{\rho_f,j_f}\lesssim_0 1$ on $[1,t^\star]$, these results give us for some $b>4>a>2$ and $r=\min(r_a,r_b)>2$, the following estimate
  \begin{align}\label{ineq:almost}
\left(\int_1^{t^\star} \|\nabla u(s)\|_{\Ld^p(\T^3)}^c \,\dd s\right)^{1/c} \lesssim_0 \mathscr{E}(0)^{\frac{1-\alpha}{2}},
  \end{align}
  which holds for any $\alpha\in[1/2,1)$ and $p$ defined by \eqref{eq:idp}, provided that that $\alpha c\leq\min(a,b)$.

  \bigskip

  It is important to note that $p=\infty$ is not yet reachable at this stage, due to the constraint $\alpha\in[1/2,1]$. However, we can first use Lemma~\ref{lem:decayfinal}  with $c=a<b$ in \eqref{ineq:almost}. In that case, going back to \eqref{eq:idp}, we see that the limit case $\alpha=1$ leads to the equality
\begin{align*}
\frac1p = \frac1r-\frac13,
\end{align*}
which, since $r>2$, implies $\frac1p<\frac16$, that is $p>6$. Taking $\alpha\in[1/2,1)$ close enough to $1$, we therefore infer the existence of $p>6$ such that, 
\begin{align*}
\|\nabla u\|_{\Ld^a(1,t^\star; \Ld^p(\T^3))} \lesssim_{0} \mathscr{E}(0)^{(1-\alpha)/2}.
\end{align*}
Since $p>6$, we infer from Hölder's inequality, for some $\widetilde{r}_a>3$, that 
\begin{align*}
 \Big(\int_{1/2}^{t^\star} \|(u\cdot\nabla) u(s)\|_{\Ld^{\tilde{r}_a}(\T^3)}^a \,\dd s \Big)^{1/a} 
&\leq \|u\|_{\Ld^\infty(1/2,t^\star; \Ld^6(\T^3))}\|\nabla u\|_{\Ld^a(1/2,t^\star; \Ld^p(\T^3))} \\
                                    &\lesssim_{0} \mathscr{E}(0)^{(1-\alpha)/2}\|u\|_{\Ld^\infty(1/2,t^\star; \Ld^6(\T^3))},\\
  &\lesssim_0 \mathscr{E}(0)^{(1-\alpha)/2},
\end{align*}
The point is that this last inequality can now replace Lemma~\ref{lem:convinterp} that we used earlied : we can perform the same analysis as before with the advantage that, now $\widetilde{r}_a>3$. This yields that $\widetilde{r}:=\min(r_b,\widetilde{r}_a)>3$ and hence taking
\begin{align*}
\widetilde{\alpha}=5\left(7-\frac{6}{\widetilde{r}}\right)^{-1} <1,
\end{align*}
we can check that $\widetilde{\alpha}\in[1/2,1)$ and satisfies
\begin{align*}
0=\frac13+\widetilde{\alpha}\left(\frac{1}{\widetilde{r}}-\frac23\right)+\frac{1-\widetilde{\alpha}}{2}.
\end{align*}
So we invoke Lemma~\ref{lem:decayfinal} another time with $r=\widetilde{r}>3$, $c=1$ and $\widetilde{\alpha}$ as above to infer 
\begin{align*}
\int_1^{t^\star} \|\nabla u(s)\|_{\Ld^\infty(\T^3)}\,\dd s\lesssim_{0}\mathscr{E}(0)^{(1-\tilde{\alpha})/2}\left(1+\mathscr{E}(0)^{(1-\alpha)/2}\right),
\end{align*}
which is an estimate of the form \eqref{ineq:thatsit}.
        \end{proof}
        
        We are finally in position to conclude the proof of Theorem \ref{realthm}.

\begin{proof}

 Applying Proposition~\ref{prop:boundM}, the question thus reduces to ensure $t^\star=+\infty$. Assuming $t^\star<+\infty$, we will reach a contradiction by proving  (for a small enough $\mathscr{E}(0)$) the existence of $t>t^\star$  which  is still a strong existence time
 and for which the inequality \eqref{def:tstar2} is satisfied.
  \medskip

The first task is to exhibit strong existence times larger than $t^\star$.  %In order to conclude we will also have to  show that the set of strong existence times is the whole $\R_+$.  
Thanks to Proposition~\ref{prop:boundM} and Proposition~\ref{propo:estvns3D}, recalling the meaning of the symbol $\lesssim_{0}$ (see Notation~\ref{nota:lesssim}), we have the existence of nondecreasing function $\ffi$ such for any $t\in[1,t^\star]$, 
\begin{multline}
\label{ineq:Madmi}\sup_{s\in[0,1]}\|\rho_f(s)\|_{\Ld^\infty(\T^3)} + \M_{\rho_f, j_f}(t) 
\\
\leq \ffi\left(\|u_0\|_{\H^{1/2}(\T^3)}+M_\alpha f_0+\E(0)+N_q(f_0) + 1 \right).
\end{multline}
Recall that  by assumption, we have $\|u_0\|_{\H^{1/2}(\T^3)}^2< \frac{1}{\C_\star^2}$. 
Using Lemma~\ref{lem:FL2}, we thus infer that for all strong existence times $t\leq t^\star$
\begin{align*}
\|u_0\|_{\H^{1/2}(\T^3)}^2 & + \C_\star\int_0^t \|F(s)\|_{\H^{-1/2}(\T^3)}^2\,\dd s \\
& \leq \|u_0\|_{\H^{1/2}(\T^3)}^2 + \C_\star\int_0^t \|F(s)\|_{\Ld^2(\T^3)}^2\,\dd s \\
&\leq \|u_0\|_{\H^{1/2}(\T^3)}^2+ \mathscr{E}(0)\C_\star \left(\M_{\rho_f, j_f}(t)+\sup_{s\in[0,1]}\|\rho(s)\|_{\Ld^\infty(\T^3)} \right),
\end{align*}
where we used the embedding $\Ld^2(\T^3)\hookrightarrow \H^{-\frac{1}{2}}(\T^3) $, with constant $1$. Combining this with~\eqref{ineq:Madmi}, we get for some nondecreasing function still denoted $\ffi$ 
\begin{multline*}
\|u_0\|_{\H^{1/2}(\T^3)}^2  + \C_\star \int_0^t \|F(s)\|_{\H^{-1/2}(\T^3)}^2\,\dd s\\
\leq \|u_0\|_{\H^{1/2}(\T^3)}^2 + \mathscr{E}(0) \ffi\left(\|u_0\|_{\H^{1/2}(\T^3)}+ M_\alpha f_0+\E(0)+N_q(f_0)+ 1 \right),
\end{multline*}
Therefore, choosing $\mathscr{E}(0)$ small enough so that 
\begin{multline*}
\ffi \left(N_q(f_0)+M_\alpha f_0+ \E(0)+\|u_0\|_{\H^{1/2}(\T^3)}+1  \right) \mathscr{E}(0)\\
 <\min\left(1,\frac{1}{\C_\star^2}-\|u_0\|_{\H^{1/2}(\T^3)}^2\right),
\end{multline*}
we deduce that
\begin{align*}
\|u_0\|_{\H^{1/2}(\T^3)}^2 +\C_\star\int_0^{t^\star} \|F(s)\|_{\H^{-1/2}(\T^3)}^2\,\dd s < \frac{1}{\C_\star^2},
\end{align*}
hence proving  by continuity the existence of strong existence times larger than $t^\star$.

To check that \eqref{def:tstar2} is satisfied after $t^\star$ we use Lemma~\ref{lem:almost} to infer the existence of an universal onto nondecreasing continuous function $\ffi:\R_+\rightarrow\R_+$ such that
  \begin{align*}
\int_1^{t^\star} \|\nabla u(s)\|_{\Ld^\infty(\T^3)}\,\dd s \leq \ffi\left(\|u_0\|_{\H^{1/2}(\T^3)}+M_\alpha f_0+N_q(f_0)+\E(0) + 1\right)\mathscr{E}(0)^\gamma,
\end{align*}
and we  observe that a smallness condition as \eqref{ineq:Emodsmall3D} ensures 
$$
\int_1^{t^\star} \|\nabla u(s)\|_{\Ld^\infty(\T^3)}\,\dd s < \delta.
$$
Therefore we can find a strong existence time $t> t^\star$ such that 
$$
\int_1^{t} \|\nabla u(s)\|_{\Ld^\infty(\T^3)}\,\dd s < \delta.
$$
This is a contradiction with the definition of $t^\star$ and finally concludes the proof.

\end{proof}

 \section{Further description of the asymptotic state }
\label{sec-asymp}

Once the exponential decay of the modulated energy is established, Proposition~\ref{prop-rhobar} leads to the existence of a profile $\rho^\infty\in\Ld^\infty(\T^3)$ which allows to describe the asymptotic behavior of $f$ in the space variable. The content of Proposition~\ref{prop-rhobar} is quite implicit as the profile is obtained by an abstract argument. It is in fact possible to describe $\rho^\infty$ in a finer way (but still, \emph{via} implicit equations) : this is the purpose of Proposition~\ref{propo-infini} that we aim at proving in this last section.

\medskip

Before doing so, it is interesting to compare the statement of Proposition~\ref{propo-infini} with the explicit asymptotic behavior of solutions to the linearized equation when $\langle u_0+j_{f_0}\rangle = 0$, that is the Vlasov equation with friction
\begin{align*}
 \partial_t f + v\cdot \nabla_x f +\div_v (-vf) =0,
\end{align*}
 for which we recall we  have 
\begin{align*}
 \W_1\left( f(t,x,v) , \widetilde\rho_0\otimes \delta_{0} \right) \conv{t}{\infty} 0,
\end{align*}
 with 
\begin{align*}
 \widetilde\rho_0(x) := \int_{\R^3} f_0\left(x-v, v\right) \, \dd v.
\end{align*}
From~\ref{eq-Xinfini} ,  we therefore see that the deviation from the linearized behavior is small, as
\begin{align*}
\Y_{\infty,x,v}^ 0 -(x-v)=-\int_0^{+\infty}  u(\tau,\Y^\tau_{\infty,x,v})\,\dd \tau,
  \end{align*}
is small in $\Ld^{\infty}(\T^3\times\R^3)$, as it is controlled by the initial modulated energy $\mathscr{E}(0)$ and 
$|\det \mathscr{A}\left(\infty,x,v\right)| -1$
 is also small in $\Ld^{\infty}(\T^3\times\R^3)$, as we will see in the upcoming proof.

\medskip

We will detail the proof of Proposition~\ref{propo-infini} only in the particular case $\langle u_0+j_{f_0}\rangle= 0$ for which the computations are a bit less tedious. The general case is a straightforward generalization  (see Remark~\ref{rem:Zt}).

\begin{proof}[Proof of Proposition~\ref{propo-infini} in the case $\langle u_0+j_{f_0}\rangle= 0$]
  
Recall the map $\Gamma_{t,x}:v\mapsto \V(0;t,x,v)$ that we already used in Lemma~\ref{charac} of Section~\ref{subsec:straight} : this very lemma ensures that, for $\delta$ small enough ($\delta e^\delta < 1/9$ is sufficient), if \eqref{ineq:assnab} is satisfied, $\Gamma_{t,x}$ is a $\mathscr{C}^1(\R^3)$-diffeomorphism. In order to capture the asymptotic profile of $\rho_f(t)$ we look at its action on a continuous function $\psi$ : 
\begin{align*}
\int_{\T^3} \rho_f(t,x)\,\psi(x)\, \dd x.
\end{align*}
Since $\rho_f$ does not solve a transport equation we cannot link it to the initial density $\rho_f(0)$, however we can write
\begin{align*}
\int_{\T^3} \rho_f(t,x)\,\psi(x)\, \dd x &= \int_{\T^3\times\R^3} f(t,x,v)\,\psi(x)\,\dd v\,\dd x\\
&= \int_{\T^3\times\R^3} e^{3t} f_0(\X(0;t,x,v),\V(0;t,x,v))\,\psi(x)\,\dd v\,\dd x\\
                                         &= e^{3t}
                                           \int_{\T^3\times\R^3} f_0(\Y(0;t,x,v),v)\,\psi(x)|\det \D_v \Gamma_{t,x}|^{-1}\,\dd v\,\dd x,
\end{align*}
where $\Y(0;t,x,v):=\X(0;t,x,\Gamma_{t,x}^{-1}(v))$. Recall that
\begin{align*}
\Gamma_{t,x}(v) = e^t v - \int_0^t e^{\tau} u(\tau,\X(\tau;t,x,v)\, \dd \tau,
\end{align*}
hence (with the notation  $\Y(\tau;t,x,v):=\X(\tau;t,x,\Gamma_{t,x}^{-1}(v)) $)
 \begin{align}\label{eq:gam}
\Gamma^{-1}_{t,x}(v) = e^{-t} v + \int_0^t e^{\tau-t} u(\tau,\Y(\tau;t,x,v))\, \dd \tau,
 \end{align}
 from which we infer
  \begin{align*}
e^t \D_v \Gamma^{-1}_{t,x}(v) = \I_3 +  \int_0^t e^{\tau} \nabla u(\tau,\Y(\tau;t,x,v))\,\D_v\Y(\tau;t,x,v) \dd \tau.
 \end{align*}
 All in all, introducing the variable $z:=(x,v)$ and denoting $\Y^s_{t,z}:=\Y(s;t,z)$,  we have established 
\begin{align}\label{eq:rhot}
\int_{\T^3} \rho_f(t,x)\,\psi(x)\, \dd x =  \int_{\T^3\times\R^3} f_0(\Y^0_{t,z},v)\psi(x)\left|\det \mathscr{A}(t,z)\right|\,\dd z,
\end{align}
where
\begin{align}\label{eq:At}
\mathscr{A}(t,z):=\I_3+\int_0^t e^\tau \nabla u(\tau,\Y^\tau_{t,z})\,\D_v\Y^\tau_{t,z}\,\dd\tau.
  \end{align}
In order to understand the behavior of $\rho_f(t)$ as $t\rightarrow+\infty$ it is therefore natural to follow the curves $t\mapsto \Y^s_{t,z}$ as $t\rightarrow +\infty$, and this is the purpose of the following
\begin{lem}\label{lem:Zt}
For $\delta>0$ small enough, the following holds.
For all $0\leq s\leq t$ and $z:=(x,v)\in\T^3\times\R^3$ we have
  \begin{align}
\label{ineq:dx}    |\D_x \Y^s_{t,z}|& \leq 2,\\
\label{ineq:dv}    |e^s \D_v \Y^s_{t,z}|& \leq 4.
    \end{align}
Furthermore, the family of maps $(s,z)\mapsto \Y^s_{t,z}$ converges in $\mathscr{C}^0(\R_+;\mathscr{C}^1(\T^3\times\R^3))$, as $t\rightarrow +\infty$, to a map $(s,z)\mapsto \Y^s_{\infty,z}$ that satisfies
$$\Y^s_{\infty,z} = x-e^{-s}v -\int_0^{+\infty}   \Big[\mathbf{1}_{[0,s]}(\tau)e^{\tau-s}+\mathbf{1}_{\tau\geq s} \Big]u(\tau, \Y^\tau_{\infty,z}) \, \dd \tau.
$$
\end{lem}
\begin{proof}
We start by recalling 
\begin{align*}
\X(s;t,x,v) = x + (1- e^{t-s}) v + \int_{s}^t \left( e^{\tau-s}-1\right) u(\tau, \X(\tau;t,x,v)) \, \dd \tau,
\end{align*}
from which, together with \eqref{eq:gam}, we deduce the following formula for $s\leq t$
 \begin{align}\label{eq:Z}
\Y^s_{t,z} =   x + (e^{-t}- e^{-s}) v + \int_0^{+\infty}   \Big[ e^{\tau-t}\mathbf{1}_{\tau\leq t}-e^{\tau-s}\mathbf{1}_{\tau \leq s}-\mathbf{1}_{s\leq \tau\leq t} \Big]u(\tau, \Y^\tau_{t,z}) \, \dd \tau.
 \end{align}
From the previous expression we infer for $s\leq t$,
 \begin{align*}
|\D_x \Y^s_{t,z}| \leq 1 + 2\int_0^{+\infty} \mathbf{1}_{\tau\leq t}| \nabla u(\tau,\Y^\tau_{t,z}) \D_x \Y^\tau_{t,z}|\,\dd \tau.
 \end{align*}
In particular, this implies
\begin{align*}
\sup_{0 \leq \tau\leq t}|\D_x \Y^\tau_{t,z}|\leq  1 + 2 \sup_{0\leq \tau\leq t}|\D_x \Y^\tau_{t,z}|  \int_0^{+\infty} \|\nabla u(\tau)\|_{\Ld^\infty(\T^3)} \,\dd \tau,
\end{align*}
which together with the assumption \eqref{ineq:assnab} implies for $s\leq t$
\begin{align*}
|\D_x \Y^s_{t,z}| \leq \sup_{0\leq \tau \leq t} |\D_x \Y^\tau_{t,z}| \leq \frac{1}{1-2\delta}, 
\end{align*}
that implies \eqref{ineq:dx} for $\delta \leq 1/4$. Similarly and returning to \eqref{eq:Z} we have for $s\leq t$ 
\begin{align*}
|e^s \D_v \Y^s_{t,z}| = 2 + 2 \int_0^t e^\tau \|\nabla u(\tau)\|_{\Ld^\infty(\T^3)} |\D_v \Y^\tau_{t,z}|\,\dd \tau,
\end{align*}
and we can proceed in the same way to obtain \eqref{ineq:dv}. To establish the existence of $(s,z)\mapsto \Y^s_{\infty,z}$, we shall prove that $(s,z)\mapsto \Y^s_{t,z}$ satisfies Cauchy's criterion as $t\rightarrow +\infty$, with respect to the local uniform metric. Since $\langle u_0+j_{f_0}\rangle = 0$, we have by Lemma~\ref{lem:moy} and by definition of the modulated energy that $t\mapsto \langle u(t)\rangle$ is integrable over $\R_+$ (due to its exponential decay). In particular we infer the integrability over $\R_+$ of $t\mapsto \|u(t)\|_{\Ld^\infty(\T^3)} \leq |\langle u(t) \rangle|+\|\nabla u(t)\|_{\Ld^\infty(\T^3)}$, thanks to the assumption \eqref{ineq:assnab}. In particular, by dominated convergence we infer that 
 \begin{align}
\label{eq:Zt}\Y^s_{t,z}=  x - e^{-s} v - \int_0^{+\infty}   \Big[e^{\tau-s}\mathbf{1}_{\tau\leq s}+\mathbf{1}_{s\leq \tau \leq t} \Big]u(\tau, \Y^\tau_{t,z}) \, \dd \tau + \text{o}(1),
\end{align}
where the notation $\text{o}(1)$ refers  a term going to $0$ in $\Ll^\infty(\R_+ \times \T^3 \times \R^3)$ in the limit $t\rightarrow +\infty$. 
\begin{rem}\label{rem:Zt}
In the general case $\langle u_0+j_{f_0}\rangle \neq 0$, $t\mapsto \langle u(t)\rangle$ is not integrable, as it converges to $\langle u_0 +j_{f_0}\rangle/2$. One needs to replace $u(\tau,\Y^\tau_{s,t})$ by $u(\tau,\Y^\tau_{s,t})-\langle u_0 +j_{f_0}\rangle/2$ in the integrand of \eqref{eq:Zt} and, by doing so, adds a diverging drift term to the equation. In a similar fashion as the proof of Proposition~\ref{prop-rhobar}, this can be counterbalanced by considering the renormalized characteristics $\Y(\tau;t,x+\langle u_0+j_{f_0}\rangle/2,v)$ instead of $\Y(\tau,t,x,v)$. The equations for these shifted trajectories are a bit different, but the convergence properties are proved in the same way, resulting in the implicit equation \eqref{eq-Xinfini}.
  \end{rem}
In particular, taking the difference of this identity \eqref{eq:Zt} at times $t_1<t_2$
 \begin{align*}
   |\Y^s_{t_2,z}- \Y^s_{t_1,z}| &\leq 2 \int_0^{+\infty} \mathbf{1}_{\tau\leq  t_2} |u(\tau,\Y^\tau_{t_2,z})-u(\tau,\Y^\tau_{t_1,z})| \,\dd \tau  + \int_{t_1}^{t_2} |u(\tau,\Y^\tau_{t_1,z}))| \,\dd \tau + \text{o}(1) \\
   &\leq 2 \int_0^{+\infty} \mathbf{1}_{\tau\leq t_2}\|\nabla u(\tau)\|_{\Ld^\infty(\T^3)} |\Y^\tau_{t_2,z}- \Y^\tau_{t_1,z}|\,\dd \tau  + \int_{t_1}^{t_2} \|u(\tau)\|_{\Ld^\infty(\T^3)} \,\dd \tau + \text{o}(1),
\end{align*}
where $\text{o}(1)$ refers here to the asymptotic $t_1 \wedge t_2 \rightarrow +\infty$, with the same uniformity as before. Using once more the integrability of $t\mapsto \|u(t)\|_{\Ld^\infty(\T^3)}$, for any compact $K\subset\R_+\times\T^3\times\R^3$, if 
\begin{align*}
  \ffi(t_1,t_2):=\sup_{(\tau,z)\in K} |\Y^\tau_{t_2,z}-\Y^\tau_{t_1,z}|,
\end{align*}
 we have established
\begin{align*}
\ffi(t_1,t_2) \leq 2 \ffi(t_1,t_2)\int_0^{+\infty} \|\nabla u(\tau)\|_{\Ld^\infty(\T^3)}\,\dd \tau + \text{o}(1),
\end{align*}
with a similar (uniform) asymptotic term $\text{o}(1)$. From assumption \eqref{ineq:assnab}, this proves
\begin{align*}
\sup_{(s,z)\in K} |\Y^s_{t_2,z}- \Y^s_{t_1,z}| = \text{o}(1),
\end{align*}
which yields  Cauchy's criterion. We deduce the existence of $(s,z)\mapsto \Y^s_{\infty,z}$, as the (local uniform) limit of $(s,z)\mapsto \Y^s_{t,z}$ as $t\rightarrow +\infty$. By dominated convergence and continuity of $u$ for positive times, $\Y^s_{\infty,z}$ must satisfy the equation
\begin{align}
\label{eq:Zinf}\Y^s_{\infty,z} = x-e^{-s}v -\int_0^{+\infty}   \Big[\mathbf{1}_{[0,s]}(\tau)e^{\tau-s}+\mathbf{1}_{\tau\geq s} \Big]u(\tau, \Y^\tau_{\infty,z}) \, \dd \tau.
\end{align}
For now $\Y^\tau_{\infty,z}$ is merely continuous (as a uniform limit) in all its variables. But it turns out that the derivatives $(s;t,z) \mapsto \D_z \Y^s_{t,z}$ enjoys the same Cauchy criterion as $\Y^s_{t,z}$. Indeed, going back to \eqref{eq:Z}, we infer, using integrability of $\tau\mapsto \|\nabla u(\tau)\|_{\Ld^\infty(\T^3)}$ over $\R_+$ and dominated convergence 
 \begin{align*}
   \D_z \Y^s_{t,z}= - \int_0^{+\infty}   \Big[e^{\tau-s}\mathbf{1}_{\tau\leq s}+\mathbf{1}_{s\leq \tau \leq t} \Big] \nabla u(\tau, \Y^\tau_{t,z}) \D_z \Y^\tau_{t,z} \, \dd \tau + \text{o}(1)+r_{s,z},
\end{align*}
where $\text{o}(1)$ refers to the asymptotic $t\rightarrow +\infty$ and is locally uniform in $s,z$, while $r_{s,z}$ is some irrelevant function which does not depend on $t$. For any $t_1<t_2$ we thus have
\begin{multline*}
|\D_z \Y^s_{t_2,z}-\D_z \Y^s_{t_1,z}| \leq 2 \int_0^{+\infty} \mathbf{1}_{\tau\leq  t_2} |\nabla u(\tau,\Y^{\tau}_{t_2,z})\D_z \Y^\tau_{t_2,z} - \nabla u(\tau,\Y^{\tau}_{t_1,z})\D_z \Y^\tau_{t_1,z}|\,\dd \tau \\ + \int_{t_1}^{t_2} |\nabla u(\tau,\Y^\tau_{t_1,z})\D_z \Y^\tau_{t_1,z}|\,\dd \tau +\text{o}(1),
\end{multline*}
where $\text{o}(1)$ refers to $t_1 \wedge t_2 \rightarrow +\infty$ and is locallly uniform in $s,z$. Owing to the integrability of $\tau\mapsto \|\nabla u(\tau)\|_{\Ld^\infty(\T^3)}$ over $\R_+$ and the uniform bound on $(s,t,z)\mapsto \mathbf{1}_{s\leq t} \D_z \Y^s_{t,z}$ due to estimates \eqref{ineq:dx} -- \eqref{ineq:dv}, we infer 
\begin{multline*}
  |\D_z \Y^s_{t_2,z}-\D_z \Y^s_{t_1,z}| \leq 2\int_0^{+\infty} \mathbf{1}_{\tau\leq  t_2} |\nabla u(\tau,\Y^{\tau}_{t_2,z})\big[\D_z \Y^\tau_{t_2,z} -\D_z \Y^\tau_{t_1,z}\big]|\,\dd \tau\\
  +\int_0^{+\infty} \mathbf{1}_{\tau\leq  t_2} |\big[\nabla u(\tau,\Y^{\tau}_{t_2,z}) - \nabla u(\tau,\Y^{\tau}_{t_1,z})\big] \D_z \Y^\tau_{t_1,z}|\,\dd \tau    +\text{o}(1).
\end{multline*}
Since $(\Y^s_{t,z})_t \rightarrow \Y^s_{\infty,z}$ pointwisely, the continuity of $\nabla u$ for positive times, its belonging to $\Ld^1(\R_+;\Ld^\infty(\T^3))$ and the  aforementioned uniform boundedness of $(s,t,z)\mapsto \mathbf{1}_{s\leq t} \D_z \Y^s_{t,z}$ entail, by dominated convergence,
\begin{multline*}
  |\D_z \Y^s_{t_2,z}-\D_z \Y^s_{t_1,z}| \leq 2 \int_0^{+\infty} \mathbf{1}_{\tau\leq  t_2} \|\nabla u(\tau)\|_{\Ld^\infty(\T^3)} |\D_z \Y^\tau_{t_2,z} -\D_z \Y^\tau_{t_1,z}|\,\dd \tau +\text{o}(1),
\end{multline*}
and we can then proceed as we have done for $\Y^s_{t,z}$ to establish the local uniform Cauchy criterion.
\end{proof}
If $f_0$ was assumed to be continuous in the space variable, we would now able to pass to the limit into formula \eqref{eq:rhot} ; indeed we would have then by dominated convergence, using the bounds that we have established on $(s,t,z)\mapsto  e^{s} \D_v \Y^s_{t,z}$ and the integrability of $v\mapsto \sup_{\T^3} f_0(\cdot,v)$,
\begin{align*}
\int_{\T^3} \rho_f(t,x)\psi(x)\,\dd x \operatorname*{\longrightarrow}_{t\rightarrow+\infty} \int_{\T^3\times\R^3} f_0(\Y^0_{\infty,z},v)\psi(x) |\det \mathscr{A}(\infty,z)|\,\dd z,
\end{align*}
with 
$$\mathscr{A}(\infty,z)=\I_3+\int_0^{+\infty} e^\tau \nabla u(\tau,\Y^\tau_{\infty,z})\,\D_v\Y^\tau_{\infty,z}\,\dd\tau.
$$ 
Notice that here the convergence $z\mapsto \mathscr{A}(t,z)$ towards $z\mapsto \mathscr{A}(\infty,z)$ is also locally uniform in $z$. However, we are not in position to replace $f_0$ by a regularized version : to do so we would need a uniqueness result for the whole coupling, and such a result is only known in dimension 2 (see \cite{HM3}). It turns out that the above convergence holds, but to establish it we have to use another change of variable. More precisely, in  \eqref{eq:rhot} we consider the change of variable $x\mapsto \Lambda_{t,v}(x):=\Y^0_{t,x,v}$. It is an admissible one thanks Lemma~\ref{lem:diff} and the estimate
\begin{align}\label{ineq:dxid}
  \|\D_x \Y^0_{t,z}-\I_3\|_\infty \leq \frac19,
\end{align}
 which itself is a consequence of \eqref{eq:Z}, \eqref{ineq:dx} and assumption \eqref{ineq:assnab}, if $\delta$ is small enough. We have therefore 
\begin{multline}
  \int_{\T^3} \rho_f(t,x)\,\psi(x)\, \dd x\\
  \label{eq:rhotbis}= \int_{\T^3\times\R^3} f_0(x,v)\,\psi(\Lambda_{t,v}^{-1}(x))|\det \mathscr{A}(t,\Lambda_{t,v}^{-1}(x),v)\det \D_x \Lambda_{t,v}^{-1}(x)|\,\dd z.
\end{multline}
The long-time behavior of $\Lambda_{t,v}^{-1}(x)$ is given by
\begin{lem}\label{lem:lam}
For all $v\in\R^3$ the map $\Lambda_{\infty,v}:x\mapsto \Y^0_{\infty,x,v}$ is a $\mathscr{C}^1$-diffeomorphism from $\T^3$ onto itself and we have $\Lambda_{t,v}^{-1}(x)\rightarrow_t \Lambda_{\infty,v}^{-1}(x)$ in $\mathscr{C}^1(\T^3\times\R^3)$, as $t\rightarrow+\infty$ and also $|\det \D_x\Lambda_{t,v}^{-1}(x)|\leq 2$ for all $x,v,t$.
\end{lem}
\begin{proof}
First, we infer from \eqref{ineq:dxid} the same estimate (by uniform convergence) for $\Lambda_{\infty,v}$, which is therefore also (thanks to Lemma~\ref{lem:diff}) a $\mathscr{C}^1$-diffeomorphism. The same lemma gives also $\det \Lambda_{t,v} \geq 1/2$ for all $t\in[1,\infty]$. Again thanks to Lemma~\ref{lem:diff}, we infer also uniformly in $t,x,v$,  $|\D_x \Lambda_{\infty,v}^{-1}(x)|\leq 9/8$ and $\det \D_x \Lambda_{t,v}(x)\geq 1/2$.  For the convergence, we write 
  \begin{align*}
    |\Lambda_{t,v}^{-1}(x)-\Lambda_{\infty,v}^{-1}(x)| &= |\Lambda_{\infty,v}^{-1} \circ \Lambda_{\infty,v} \circ \Lambda_{t,v}^{-1}(x) -\Lambda_{\infty,v}^{-1}(x)| \\
                                                       &\leq \frac98|\Lambda_{\infty,v}\circ \Lambda_{t,v}^{-1}(x) - x|\\
                                                       &= \frac98|\Lambda_{\infty,v}\circ \Lambda_{t,v}^{-1}(x) - \Lambda_{t,v}\circ \Lambda_{t,v}^{-1}(x)|,
  \end{align*}
  that goes to $0$ locally uniformly in $x,v$ thanks to Lemma~\ref{lem:Zt}. Since the inversion map is $\mathscr{C}^1$ on $\text{GL}_3(\R)$, using the previous lower bound on the determinants, we infer from the equatity $\D_x \Lambda_{t,v}^{-1} = (\D_x\Lambda_{t,v})^{-1}\circ \Lambda_{t,v}^{-1}$ and the previous convergence the announced convergence in $\mathscr{C}^1(\T^3\times\R^3)$.
\end{proof}
Since $\mathscr{A}(t,z)$ is uniformly bounded and continuous and converges (locally uniformly) towards $\mathscr{A}(\infty,z)$, we infer from Lemma~\ref{lem:lam} and the dominated convergence theorem (using $f_0\in\Ld^1(\T^3\times\R^3)$),
\begin{multline*}
  \int_{\T^3} \rho_f(t,x)\,\psi(x)\, \dd x\\
 \operatorname*{\longrightarrow}_{t\rightarrow+\infty} \int_{\T^3\times\R^3} f_0(x,v)\,\psi(\Lambda_{\infty,v}^{-1}(x))|\det \mathscr{A}(\infty,\Lambda_{\infty,v}^{-1}(x),v)\det \D_x \Lambda_{\infty,v}^{-1}(x)|\,\dd z,
\end{multline*}
and using back the change of variable $x\mapsfrom \Lambda_{\infty,v}(x)$ (which is admissible thanks to Lemma~\ref{lem:lam}) we have eventually proved
\begin{align*}
(\rho_f(t))_t \operatorname*{\rightharpoonup}_{t\rightarrow+\infty} \rho^\infty,
\end{align*}
where
\begin{align*}
  \rho^\infty(x) := \int_{\R^3} f_0(\Y^0_{\infty,x,v},v) |\det \mathscr{A}(\infty,x,v)|\,\dd v,
\end{align*}
which concludes the proof.
\end{proof}

\section{Appendix}
\label{sec:appendix}
\subsection{Wasserstein distance} 
\label{sec:Wasserstein}
To simplify the presentation, $X$ here will denote either $\T^3$ or $\T^3\times\R^3$.
\begin{defi}%[Wasserstein distance]
For $m>0$ we denote by $\mathcal{M}_{1,m}(X)$ the set of all measures $\mu$ such that 
\begin{align*}
\int_X |z| \,\dd \mu(z) < +\infty,\qquad \mu(X)=m.
\end{align*}
\end{defi}
\begin{defi}
  Fix $m>0$ and consider $\mu$ and $\nu$ in $\mathcal{M}_{1,m}(X)$. The \emph{Wassertein distance} between $\mu$ and $\nu$ is
\begin{align*}
\W_{1} (\mu, \nu):= \inf_{\gamma \in \Gamma (\mu, \nu)} \int_{X^2} |z-z'| \, \dd \gamma (z,z'),
\end{align*}
where $\Gamma (\mu, \nu)$ denotes the collection of all measures on $X\times X$ with first and second marginal respectively equal to $\mu$ and $\nu$.
\label{defi:Wasserstein}
\end{defi}

\begin{propo}[$\W_1$ metrizes the weak-$\star$ convergence]
Fix $m>0$. Given $(\mu_n)_n\in\mathcal{M}_{1,m}(X)^\N$ and $\mu\in\mathcal{M}_{1,m}(X)$, the two following facts are equivalent
\begin{itemize}
\item[(i)] For all $f\in\mathscr{C}^0_b(X)$, 
\begin{align*}
\int_X (f(z)+|z|)\,\dd\mu_n(z) \conv{n}{+\infty} \int_X (f(z)+|z|)\,\dd\mu(z).
\end{align*}
\item[(ii)] $(\W_1(\mu_n,\mu))_n \rightarrow_n 0$.
\end{itemize}
\end{propo}

\begin{propo}[Monge-Kantorovitch duality]
\label{MK}
Fix $m>0$ and consider $\mu$ and $\nu$ in $\mathcal{M}_{1,m}(X)$. Then
\begin{align*}
\W_1(\mu,\nu) = \sup \left\{\int_X \phi(z) \dd\mu(z) - \int_X \phi(z) \dd\nu(z)\,:\, \phi\in\textnormal{Lip}(X), \|\nabla \phi\|_\infty \leq 1\right\}.
\end{align*}
%where the supremum runs over the unit ball of $\textnormal{Lip}(X)$.
\end{propo}

\subsection{Exponential decay}

\begin{lem}\label{lem:gronexp}
  Consider $u:\R_{+}\rightarrow\R_{+}$ a non-increasing integrable function satisfying for some $\lambda >0$ and almost all $t\geq 0$ 
  \begin{align*}
\lambda \int_t^\infty u(s)\,\dd s \leq u(t).
  \end{align*}
  Then for $t\geq 0$ there holds
  \begin{align*}
u(t) \lesssim_{u(0),\lambda} e^{-\lambda t}.
    \end{align*}
  \end{lem}
  \begin{proof}
    If $v(t)$ denotes the integral in the estimate, $v\in\W^{1,\infty}(\R_+)$ satisfies $v'\leq -\lambda v$, so the standard version of the Gronwall Lemma implies $v(t)\leq v(0)e^{-\lambda t}$. Since $u\leq u(0)$ w.l.o.g. we can assume $t\geq 1$ and since $u$ is non-increasing, we have
    \begin{align*}
 u(t) \leq \int_{t-1}^t u(s)\,\dd s \leq v(t-1) \leq v(0)e^{-\lambda t} \leq \frac{1}{\lambda}u(0) e^{-\lambda t}.\qquad\qedhere
    \end{align*}
    \end{proof}

\subsection{Perturbation of the identity map}

% \begin{thm}[Hadamard]
% \label{thm:Had}
% Let $n \geq 1$. Let $Z$ be $\mathscr{C}^1$ map from $\R^n$ to $\R^n$.
% Then $Z$ is a $\mathscr{C}^1$ diffeomorphism if and only if $\det \nabla Z$ never vanishes and $f$ is proper, i.e.
% $$
% | Z(y)| \to +\infty, \quad \text{as  } |y| \to +\infty.
% $$
% \end{thm}

We use in this work the following version of the inverse function theorem.

\begin{lem}\label{lem:diff}
For $\Omega=\T^3$ or $\Omega=\R^3$, if $\phi:\Omega\rightarrow\Omega$ is $\mathscr{C}^1$ and satisfies $\|\nabla \phi\|_\infty < 1$, then $f:=\textnormal{Id}+\phi$ is a $\mathscr{C}^1$-diffeomorphism of $\Omega$ onto itself satisfying $\|\nabla f\|_\infty \leq (1-\|\nabla \phi\|_\infty)^{-1}$. If furthermore $\|\nabla \phi\|_\infty \leq 1/9$, then $\det \nabla f \geq 1/2$.
\end{lem}

\subsection{Maximal regularity}
The maximal regularity estimate for the heat equation, on the whole space, can be stated in the following way.
\begin{thm}\label{thm:heatwhole}
  For $p,q\in (1,\infty)$, and  $\varphi\in\mathcal{S}(\R\times\R^3)$ such that $\varphi(0,\cdot)=0$, there holds
  \begin{align*}
\|\Delta \varphi\|_{\Ld^p(\R_+;\Ld^q(\R^3))} \lesssim_{p,q} \|\partial_t \varphi-\Delta \varphi\|_{\Ld^p(\R_+;\Ld^q(\R^3))}.
    \end{align*}
  \end{thm}
This estimate is for instance a consequence of \cite{hieb}. Naturally, one expects an analogous estimate on the torus, but we did not manage to exhibit a precise reference in the literature. For the sake of completeness we give therefore a proof of the following corollary.
  \begin{coro}\label{coro:max}
  For $p,q\in (1,\infty)$ and $\psi\in\mathcal{S}(\R\times\R^3)$ which is $\mathbb{Z}^3$-periodic in the space variable and such that $\psi(0,\cdot)=0$, there holds
  \begin{align*}
\|\Delta \psi\|_{\Ld^p(\R_+;\Ld^q(\T^3))} \lesssim_{p,q} \|\partial_t \psi-\Delta \psi\|_{\Ld^p(\R_+;\Ld^q(\T^3))}.
    \end{align*}
  \end{coro}
  \begin{proof}
    Let's use the Dirac comb $\Sh:=\sum_{n\in\mathbb{Z}^3} \delta_n$ as a getway between functions defined on $\R^3$ and $\mathbb{Z}^3$-periodic functions (identified as functions defined on the torus $\T^3$). In the sequel $C$ denotes the open unit cube $(0,1)^3$.
    \begin{lem}
      For any $g\in\mathcal{S}(\R^3)$ which is $\mathbb{Z}^3$-periodic there exists $h\in\mathscr{D}(C)$ such that $g = \Sh\star h$. Furthermore, for any such function $h$, and for any $\in[1,\infty]$ there holds $\|g\|_{\Ld^q(\R^3)} = \|h\|_{\Ld^q(\T^3)}$.
    \end{lem}
    \begin{proof}
Fix a non-zero $\theta\in\mathscr{D}(C)$, then $h:=g\theta/(\Sh \star \theta)$  is a well-defined element of $\mathscr{D}(C)$ satisfying $g=\Sh \star h = \sum_{n\in\mathbb{Z}^3} \tau_n h$. Since $h\in\mathscr{D}(C)$, the functions $\tau_n h$ have disjoint supports which justifies the equality of the $\Ld^q$-norms.
      \end{proof}
      Obviously the previous lemma holds also when adding a time variable. In particular we have the existence of $\varphi\in\mathcal{S}(\R\times\R^3)$, such that $\varphi(0,\cdot)=0$, $\ffi(t,\cdot)\in\mathscr{D}(C)$ for all $t$ and $\psi = \Sh\star \varphi$ (here the convolution is to be understood in the space variable only). From this representation formula we also deduce $\Delta \psi = \Sh\star \Delta \varphi$ and $\partial_t \psi -\Delta \psi=\Sh\star(\partial_t \varphi -\Delta \varphi)$, where the spatial support of $\Delta \varphi$ and $\partial_t \varphi -\Delta \varphi$ are still included in $C$ : the previous lemma applies therefore to write
      \begin{align*}
        \|\Delta \psi\|_{\Ld^p(\R_+;\Ld^q(\T^3))} &= \|\Delta \varphi\|_{\Ld^p(\R_+;\Ld^q(\R^3))} \\
        &\lesssim_{p,q} \|\partial_t \varphi-\Delta \varphi\|_{\Ld^p(\R_+;\Ld^q(\R^3))} = \|\partial_t \psi-\Delta \psi\|_{\Ld^p(\R_+;\Ld^q(\T^3))},
      \end{align*}
      where the inequality is obtained from \eqref{thm:heatwhole}.
        \end{proof}
In the current article we will use the following consequence of Corollary~\ref{coro:max}, which is obtained by a standard approximation argument.
  \begin{coro}\label{coro:heat}
For $p,q\in(1,\infty)$ and $T>0$ if $S\in\Ld^p(0,T;\Ld^q(\T^3))$, the unique tempered solution $u$ of 
    \begin{equation}
\label{heat with viscosity}
\partial_t u  - \Delta u = S, \qquad u|_{t=0} = 0, 
\end{equation} satisfies
\label{parab}
\begin{equation}
\| \Delta u \|_{\Ld^p (0,T; \Ld^q(\T^3))} \lesssim_{p,q}  \| S \|_{\Ld^p (0,T; \Ld^q(\T^3))} .
\end{equation}
\end{coro}

\subsection{Interpolation}

The following classical interpolation estimate can be for instance found in \cite[Thm 1.5.2]{chemilani}.
\begin{thm}[Gagliardo-Nirenberg-Sobolev]\label{lem:gns}
Consider $1\leq p,q,r\leq \infty$ and $m\in\N$. Assume that $j\in\N$ and $\alpha\in \R$ satisfy
\begin{align*}
\frac{1}{p} &= \frac{j}{3}+ \left(\frac{1}{r}-\frac{m}{3}\right)\alpha + \frac{1-\alpha}{q},\\
\frac{j}{m} &\leq \alpha \leq 1,
\end{align*}
with the exception $\alpha<1$ if $m-j-3/r\in\N$. Then, the following holds. For any  $g\in\Ld^q(\T^3)$, if $\D^m g\in\Ld^r(\T^3)$, then $\D^j g\in\Ld^p(\T^3)$ and we have the following estimate for $g$ 
\begin{align*}
\|\D^j g\|_{\Ld^p(\T^3)} \lesssim \| \D^m g\|_{\Ld^r(\T^3)}^\alpha \|g\|_{\Ld^q(\T^3)}^{1-\alpha} + \|g\|_{\Ld^q(\T^3)},
\end{align*}
where the constant behind $\lesssim$ does not depend on $g$. If $\langle \D^j g \rangle =0$, then the term $\|g\|_{\Ld^q(\T^3)}$ in the right-hand side can be dispensed with.
\end{thm}

\subsection{Parabolic regularization for the Navier-Stokes equations with a source term}
\label{sec:proofH1}
%
% This section is dedicated to $\H^1$ energy estimates for the Navier-Stokes system. Such higher order energy estimates for the Navier-Stokes estimate seem to be folklore; yet we need precise enough versions keeping track of the viscosity dependence.

The main result of this section is Proposition~\ref{coro:ns3D}, which gives higher order energy estimates for the Navier-Stokes system together with a form of regularization along  time. These estimates seem to be folklore but we give here the proof for the sake of completeness.

%\begin{propo}\label{coro:ns2D}
%Fix $u_0\in\Ldiv^2(\T^2)$ and $F\in\Ll^2(\R_+;\Ld^2(\T^2))$. The Leray solution of the Navier-Stokes system with source term $F$, initiated by $u_0$ satisfies $u\in\mathscr{C}^0(\R_+^*;\H^1(\T^2))$ and $u\in\Ll^2(\R_+^*;\H^2(\T^2))$ and furthermore for $t\geq 1$
%\begin{align*}
%\|\nabla u(t)\|_{\Ld^2(\T^2)}^2 +\nu \int_1^t \|\Delta u(s)\|_{\Ld^2(\T^2)}^2\,\dd s 
%  \leq \ffi\left(\A(t)+\frac{1}{\nu}\right) \left(1+\int_0^t  \|F(s)\|_{\Ld^2(\T^2)}^2 \,\dd s%\right),
%\end{align*}
%where $\ffi$ and $\A$ are given in Proposition~\ref{propo:nsreg}.
%\end{propo}

\begin{propo}\label{coro:ns3D}
There exists a universal constant $\C_\star>0$ such that the following holds. Consider  $u_0\in\Hdiv^{1/2}(\T^3)$, $F\in\Ll^2(\R_+;\H^{-1/2}(\T^3))$ and $T>0$ such that %\eqref{ineq:smallness} is satisfied for the constant $\C_\star$ given by Proposition~\ref{propo:esns3D}. 
\begin{align}
\label{ineq:smallness}\|u_0\|_{\H^{1/2}(\T^3)}^2 + \C_\star \int_0^T \|F(s)\|_{\H^{-1/2}(\T^3)}^2\,\dd s \leq \frac{1}{\C_\star^2}.
\end{align} 
Then, there exists on $[0,T]$ a unique Leray solution of the Navier-Stokes system with source $F$ and with initial data $u_0$. This solution $u$ belongs to $\Ld^\infty([0,T];\H^{1/2}(\T^3))\cap \Ld^2(0,T;\H^{3/2}(\T^3))$ and satisfies for a.e. $0\leq t \leq T$ 
\begin{multline}
\label{eq-nrj12}
\|u(t)\|_{\H^{1/2}(\T^3)}^2 + \int_0^t \|\nabla u(s)\|_{\dot{\H}^{1/2}(\T^3)}^2\,\dd s\\ \leq \|u_0\|_{\H^{1/2}(\T^3)}^2 + \C_\star \int_0^t \|F(s)\|_{\H^{-1/2}(\T^3)}^2\,\dd s.
\end{multline} 
Furthermore, if $F\in\Ld_{\textnormal{loc}}^2(\R_+;\Ld^2(\T^3))$, we have for a.e. $1/2\leq t \leq T$ 
\begin{align}
\label{ineq:capital}\|\nabla u(t)\|_{\Ld^2(\T^3)}^2 + \int_{1/2}^t \|\Delta u(s)\|_{\Ld^2(\T^3)}^2\,\dd s   \lesssim \A(t)+\int_0^t  \|F(s)\|_{\Ld^2(\T^3)}^2 \,\dd s,
\end{align} 
where $\lesssim$ depends only on $\C_\star$, and $\A$ is defined by
\begin{align}
  \label{def:A}\A(t) := \frac{1}{2}\sup_{[0,t]} \|u(s)\|_{\Ld^2(\T^3)}^2 + \int_0^t \|\nabla u(s)\|_{\Ld^2(\T^3)}^2\,\dd s.
                       \end{align}
\end{propo}
\begin{proof}
  The proof proceeds in two different steps. First, if such a Leray solution of the Navier-Stokes exists, because of the interpolation estimate
\begin{align*}
  \|\cdot\|_{\dot{\H}^1(\T^3)}^4\leq \|\cdot\|_{\dot{\H}^{1/2}(\T^3)}^2 \|\cdot\|_{\dot{\H}^{3/2}(\T^3)}^2,
\end{align*}
we have in particular $u\in\Ld^4(0,T;\H^1(\T^3))$. This is a known case of weak-strong uniqueness, see for instance the stability result \cite[Theorem 3.3]{CDGG}. The second step is to prove that such a solution indeed exists. This follows by a simple compactness argument, using Proposition~\ref{propo:nsreg} and Proposition~\ref{propo:esns3D} below, choosing for $\gamma$ an appropriate regularization of $t\mapsto 2t\mathbf{1}_{0\leq t\leq 1/2}+\mathbf{1}_{t>1/2}$.
\end{proof}
In order to prove the existence of a solution as in Proposition~\ref{coro:ns3D}, we rely on the following standard  approximation procedure: we  consider, for $\chi\in\mathscr{C}^\infty(\T^3)$, the  regularized system:
\begin{align}
\label{eq:nsreg}\partial_t u + (\widetilde{u}_\chi\cdot\nabla)u-\Delta u +\nabla p &= F,\\
\label{eq:nsreg2}\div\,u&=0,\\
\label{eq:nsreg3}u(0,\cdot)&=u_0,
\end{align}
where $\widetilde{u}_\chi := u\star \chi$. When $u_0$ and $F$ are smooth, the existence of a unique smooth solution to system \eqref{eq:nsreg} -- \eqref{eq:nsreg3} is standard.
\begin{propo}\label{propo:nsreg}
Consider a nondecreasing function $\gamma\in\mathscr{C}^1_b(\R)$ vanishing at $0$ and such that $\|\gamma\|_{\W^{1,\infty}(\R)}\leq 1$. There exists $\C>0$ and an onto nondecreasing continuous function $\ffi:\R_+\rightarrow\R_+$, such that for any $u_0\in\mathscr{C}^\infty_{\div}(\T^3)$, $F\in\mathscr{C}^\infty(\R_+\times\T^3)$ and any $\chi\in\mathscr{C}^\infty(\T^3)$ such that $\|\chi\|_1 =1$, the unique solution $u$ of \eqref{eq:nsreg} -- \eqref{eq:nsreg3} satisfies for $t\geq 0$, 
\begin{multline}
\label{ineq:nsreg}  \gamma(t) \|\nabla u(t)\|_{\Ld^2(\T^3)}^2 + \int_0^t \gamma(s) \|\Delta u(s)\|_{\Ld^2(\T^3)}^2\,\dd s \\
  \lesssim  \left(\A(t)+\int_0^t \gamma(s) \|F(s)\|_{\Ld^2(\T^3)}^2 \,\dd s\right)\Phi(h(t)),
\end{multline}
where the constant behind $\lesssim$ is universal, $\Phi(z):=(1+z)e^z$,  $\A$ is given by \eqref{def:A} and
\begin{align}
\label{def:h3}h(t) &:= \C\int_0^t \|\nabla u(s)\|_{\Ld^3(\T^3)}^2\,\dd s.
\end{align}
\end{propo}
\begin{proof}
%In this proof the symbol $\lesssim$ will mean $\leq$ to  the same dependencies as $\C$ and $\ffi$ . 
We multiply the equation by $-\gamma(t) \Delta u$, and use adequate integrations by parts together with Young's and Hölder's inequality, to get
\begin{multline}
\label{ineq:23D}\frac{1}{2}\frac{\dd}{\dd t}\Big\{\gamma(t)\|\nabla u(t)\|_{\Ld^2(\T^3)}^2\Big\} + \frac{\gamma(t)}{2}\|\Delta u(t)\|_{\Ld^2(\T^3)}^2\\ \leq \frac12 \gamma'(t)\|\nabla u(t)\|_{\Ld^2(\T^3)}^2 \\
+   \frac{\gamma(t)}{2} \|F(t)\|_{\Ld^2(\T^3)}^2 + \gamma(t)\|\Delta u(t)\|_{\Ld^2(\T^3)} \|u(t)\|_{\Ld^6(\T^3)}\|\nabla u(t)\|_{\Ld^3(\T^3)}.
\end{multline}
We use then another time Young's inequality and the Sobolev embedding $\H^1(\T^3)\hookrightarrow\Ld^6(\T^3)$  to write 
\begin{multline*}
\frac{\dd}{\dd t}\Big\{\gamma(t)\|\nabla u(t)\|_{\Ld^2(\T^3)}^2\Big\} +  \gamma(t)\|\Delta u(t)\|_{\Ld^2(\T^3)}^2\\ \lesssim \gamma'(t)\|\nabla u(t)\|_{\Ld^2(\T^3)}^2 +   \gamma(t) \|F(t)\|_{\Ld^2(\T^3)}^2 + \gamma(t)\|u(t)\|_{\H^1(\T^3)}^2\|\nabla u(t)\|_{\Ld^3(\T^3)}^2.
\end{multline*}
Using the definition \eqref{def:A} of $\A(t)$ and the fact that $\|\gamma\|_{\W^{1,\infty}(\R)}\leq 1$, we infer, introducing $\ell(t):=\gamma(t)\|\nabla u(t)\|_{\Ld^2(\T^3)}^2$
\begin{multline*}
\ell'(t) + \gamma(t)\|\Delta u(t)\|_{\Ld^2(\T^3)}^2 \\
 \lesssim \|\nabla u(t)\|_{\Ld^2(\T^3)}^2 + \A(t)\|\nabla u(t)\|_{\Ld^3(\T^3)}^2 + \gamma(t) \|F(t)\|_{\Ld^2(\T^3)}^2 + \ell(t)\|\nabla u(t)\|_{\Ld^3(\T^3)}^2,
\end{multline*}
which implies by Gronwall's inequality (since $\ell(0)=0$), using once again the definition of $\A(t)$,
\begin{multline*}
\ell(t) + \int_0^t \gamma(s)\|\Delta u(s)\|_{\Ld^2(\T^3)}^2\,\dd s \\
\lesssim \left(\A(t)(1+ h(t)) + \int_0^t \gamma(s)\|F(s)\|_{\Ld^2(\T^3)}^2\,\dd s\right)\exp(h(t)),
\end{multline*}
where $h$ is given by \eqref{def:h3}, for some univeral constant $\C>0$ ; this last estimate can be recasted into \eqref{ineq:nsreg}.
\end{proof}

Recall the notation $\|\cdot \|_{\dot{\H}^s(\T^3)}$ for the $\Ld^2$ norm associated with the  multiplier $|\xi|^s$. %In particular for $s\neq 0$, $\|\cdot\|_{\dot{\H}^s(\T^3)}$ is a norm on the space of vanishing-mean functions (and only a semi-norm without this assumption).
\begin{propo}\label{propo:esns3D}
There exists a universal constant $\C_\star$ such that the following holds. For any $u_0\in\mathscr{C}^\infty_{\div}(\T^3)$, $F\in\mathscr{C}^\infty(\R_+\times\T^3)$ and any $\chi \in\mathscr{C}^\infty(\T^3)$ such that $\|\chi\|_1 =1$, if for some $T>0$ one has
\begin{align}
\label{ineq:smallness-proof}\|u_0\|_{\H^{1/2}(\T^3)}^2 +\C_\star\int_0^T \|F(s)\|_{\H^{-1/2}(\T^3)}^2\,\dd s \leq \frac{1}{\C_\star^2},
\end{align}
then the unique solution $u$ of \eqref{eq:nsreg} -- \eqref{eq:nsreg3} satisfies for $t\in[0,T]$,  
\begin{align*}
\|u(t)\|_{\H^{1/2}(\T^3)}^2 +\int_0^t \|\nabla u(s)\|_{\dot{\H}^{1/2}(\T^3)}^2\,\dd s \leq \|u_0\|_{\H^{1/2}(\T^3)}^2 +\C_\star\int_0^t \|F(s)\|_{\H^{-1/2}(\T^3)}^2\,\dd s.
\end{align*}
In particular, recalling the definition \eqref{def:h3}, on $[0,T]$ we have $h\leq \C/\C_\star$ where $\C$ is the universal constant given in Proposition~\ref{propo:nsreg}.
\end{propo}
\begin{proof}
 Let us first recall the fundamental energy estimate  
\begin{align}
\label{eq:fond}\frac{\dd}{\dd t}\|u\|_{\Ld^2(\T^3)}^2 + \|u\|_{\dot{\H}^{1}(\T^3)}^2 \lesssim \|F\|_{\H^{-1}(\T^3)}^2.
\end{align}
Consider $\Lambda$ the Fourier multiplier associated with $|\xi|$.
After taking the scalar product with $\Lambda u$, thanks to Plancherel's formula, Hölder's inequality, to the continuity of the Leray projector $\mathbb{P}$ on $\Ld^{3/2}(\T^3)$, we can obtain as well
\begin{multline*}
\frac{\dd }{\dd t}\| u\|_{\dot{\H}^{1/2}(\T^3)}^2 + \|\nabla u\|_{\dot{\H}^{1/2}(\T^3)}^2 \\
\lesssim \|\Lambda u\|_{\Ld^3(\T^3)} \|\nabla u\|_{\Ld^3(\T^3)} \|u\|_{\Ld^3(\T^3)} + \|\Lambda u\|_{\dot{\H}^{1/2}(\T^3)}\|F\|_{\dot{\H}^{-1/2}(\T^3)}.
\end{multline*} 
Using Young's inequality and combining with \eqref{eq:fond} we infer
\begin{align*}
\frac{\dd }{\dd t}\| u\|_{\H^{1/2}(\T^3)}^2 + \|\nabla u\|_{\dot{\H}^{1/2}(\T^3)}^2 \lesssim \|\Lambda u\|_{\Ld^3(\T^3)} \|\nabla u\|_{\Ld^3(\T^3)}  \|u\|_{\Ld^3(\T^3)}  + \|F\|_{\H^{-1/2}(\T^3)}^2.
\end{align*} 
 we therefore have 
\begin{align*}
\frac{\dd }{\dd t}\| u\|_{\H^{1/2}(\T^3)}^2 + \|\nabla u\|_{\dot{\H}^{1/2}(\T^3)}^2 \lesssim \|\Lambda u\|_{\Ld^3(\T^3)}  \|\nabla u\|_{\Ld^3(\T^3)}  \|u\|_{\Ld^3(\T^3)}  + \|F\|_{\H^{-1/2}(\T^3)}^2.
\end{align*} 
We have by Sobolev embedding
\begin{align*}
\|g-\langle g\rangle\|_{\Ld^3(\T^3)} &\lesssim  \|g\|_{\dot{\H}^{1/2}(\T^3)},\\
\|g\|_{\Ld^3(\T^3)} &\lesssim  \|g\|_{\H^{1/2}(\T^3)}.
\end{align*}
Since $\Lambda u$ and $\nabla u$ have a vanishing mean, % and share the same $\dot{\H}^{1/2}(\T^3)$ norm that is, 
we therefore have 
\begin{align*}
\frac{\dd }{\dd t}\| u\|_{\H^{1/2}(\T^3)}^2 + \|\nabla u\|_{\dot{\H}^{1/2}(\T^3)}^2 \lesssim \|\nabla u\|_{\dot{\H}^{1/2}(\T^3)}^2\|u\|_{\H^{1/2}(\T^3)} +  \|F\|_{\H^{-1/2}(\T^3)}^2.
\end{align*} 
% and another use of Young's inequality leads to 
% \begin{align*}
% \frac{\dd }{\dd t}\| u\|_{\dot{\H}^{1/2}(\T^3)}^2 + \nu\|u\|_{\dot{\H}^{3/2}(\T^3)}^2 \lesssim \frac{1}{\nu}\Big(\|u\|_{\dot{\H}^1(\T^3)}^4  + \|F\|_{\H^{-1/2}(\T^3)}^2\Big).
% \end{align*} 
% Using the interpolation estimate $\|u\|_{\dot{\H}^1} \lesssim \|u\|_{\dot{\H}^{1/2}}^{1/2}\|u\|_{\dot{\H}^{3/2}}^{1/2}$ and a last time Young's inequality, we recover 
% \begin{align*}
% \frac{\dd }{\dd t}\| u\|_{\dot{\H}^{1/2}(\T^3)}^2 + \nu\|u\|_{\dot{\H}^{3/2}(\T^3)}^2 \lesssim \frac{1}{\nu}\Big(\|u\|_{\dot{\H}^{3/2}(\T^3)}^2 \|u\|_{\dot{\H}^{1/2}(\T^3)}^2 + \|F\|_{\dot{\H}^{-1/2}(\T^3)}^2\Big),
% \end{align*} 
This is a differential inequality of the form 
\begin{align*}
x'(t) +  y(t) \leq \C\Big( x(t)^{1/2}y(t) + z(t)\Big),
\end{align*} 
where $\C$ is some universal constant and
\begin{align}
\label{eq:xyz}
  x(t) = \|u(t)\|_{\H^{1/2}(\T^3)}^2, \quad  y(t) = \|\nabla u(t)\|_{\dot{\H}^{1/2}(\T^3)}^2,\quad z(t) = \|F(t)\|_{\dot{\H}^{-1/2}(\T^3)}^2. 
\end{align}
After integration, we hence have 
\begin{align*}
x(t) + \int_0^t y(s)\,\dd s \leq x(0) + \int_0^t y(s)\left(\C x(s)^{1/2}-1\right)\, \dd s + \C\int_0^t z(s)\,\dd s.
\end{align*}
In particular, if for some $T>0$ one has (this precisely corresponds to the assumption \eqref{ineq:smallness})
\begin{align}
\label{ineq:initC}x(0)+\C\int_0^T z(s)\,\dd s\leq\frac{1}{\C^2},
\end{align}
then by a standard continuity argument we can show that the inequality $x(t)^{1/2}\leq 1/C$ which is true for $t=0$ remains valid up to $t=T$, entailing on $[0,T]$,
\begin{align}
x(t) +  \int_0^t y(s)\,\dd s \leq x(0) + \C \int_0^t z(s)\,\dd s,
\end{align}
which corresponds to the desired inequality, recalling~\eqref{eq:xyz}.
\end{proof}

\bigskip

\noindent {\bf Acknowledgements.} We thank Young-Pil Choi for pointing out a mistake in a previous version of the paper. DHK and AM were partially supported by the grant ANR-19-CE40-0004. IM was partially supported by the European Research Council (ERC) MAFRAN grant under the European's Union Horizon 2020 research and innovation programme (grant agreement No 726386) while he was a research associate in PDEs at DPMMS, University of Cambridge.

\bibliographystyle{abbrv}
\bibliography{vns}
\end{document}